\documentclass[11pt, oneside]{amsart}
\usepackage{pb-diagram,amssymb,epic,eepic,verbatim,braket, xcolor}
\usepackage{wasysym,graphics,epsfig,psfrag,braket,marginnote}
\usepackage{hyperref,url}
\hypersetup{colorlinks}
\usepackage{color}
\usepackage[numbers]{natbib}

\usepackage[shortlabels]{enumitem}
\usepackage{soul}
\usepackage{bm}
\newcommand{\beqn}{\begin{equation*}}
\newcommand{\eeqn}{\end{equation*}}
\newcommand{\beq}{\begin{equation}}
\newcommand{\eeq}{\end{equation}}

\usepackage{amsfonts,transparent}
\DeclareMathAlphabet{\mathpgoth}{OT1}{pgoth}{m}{n}
\DeclareMathAlphabet{\mathesstixfrak}{U}{esstixfrak}{m}{n}
\DeclareMathAlphabet{\mathboondoxfrak}{U}{BOONDOX-frak}{m}{n}
\usepackage{amsmath,amssymb}
\usepackage[bbgreekl]{mathbbol}
\usepackage{yfonts,mathtools}
\definecolor{darkred}{rgb}{0.5,0,0}
\definecolor{darkgreen}{rgb}{0,0.5,0}
\definecolor{darkblue}{rgb}{0,0,0.5}
\hypersetup{ colorlinks,
linkcolor=darkblue,
filecolor=darkgreen,
urlcolor=darkred,
citecolor=darkblue }
\newtheorem{theorem}{Theorem}[section]

\newtheorem{corollary}[theorem]{Corollary}

\newtheorem{proposition}[theorem]{Proposition}
\newtheorem{lemma}[theorem]{Lemma}
\newtheorem{lem}[theorem]{}
\theoremstyle{definition}
\newtheorem{definition}[theorem]{Definition}
\theoremstyle{remark}
\newtheorem{discussion}[theorem]{Discussion}
\newtheorem{remark}[theorem]{Remark}
\newtheorem{example}[theorem]{Example}
\newcommand{\blem}{\begin{lem} \rm}
\newcommand{\elem}{\end{lem}}

\newcommand\M{\mathcal{M}}
\newcommand\D{\mathcal{D}}

\renewcommand\M{\mathcal{M}}
\renewcommand\D{\mathbb{D}}

\newcommand\XX{\mathbb{X}}

\renewcommand\S{\mathcal{S}}

\newcommand{\LL}{\mathcal{L}}

\renewcommand{\L}{\mathcal{L}}

\newcommand{\J}{\mathcal{J}}

\newcommand{\U}{\mathcal{U}}

\newcommand{\F}{\mathcal{F}}
\newcommand{\N}{\mathbb{N}}

\newcommand{\R}{\mathbb{R}}

\newcommand{\C}{\mathbb{C}}
\newcommand{\CP}{\mathbb{C}P}

\newcommand{\cS}{\mathcal{S}}
\newcommand{\cM}{\mathcal{M}}

\newcommand{\h}{\mathfrak{h}}

\newcommand{\cU}{\mathcal{U}}

\newcommand{\Z}{\mathbb{Z}}
\newcommand{\Q}{\mathbb{Q}}

\newcommand{\cct}{\frac{d}{dt}}

\newcommand{\ccs}{\frac{d}{ds}}

\newcommand{\mvP}{\widetilde{P}}

\newcommand{\ddt}{\frac{\d}{\d t}}

\newcommand{\ccrho}{\frac{\partial}{\partial \rho}}
\newcommand{\ppth}{\frac{\partial}{\partial \theta}}

\renewcommand{\P}{\mathbb{P}}

\newcommand\lie[1]{\mathfrak{#1}}

\newcommand{\g}{\lie{g}}

\newcommand{\on}{\operatorname}

\newcommand\white{{\includegraphics[width=.05in]{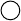}}}

\newcommand\black{{\includegraphics[width=.05in]{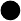}}}

\newcommand{\ainfty}{{$A_\infty$\ }}
\newcommand{\ab}{\ab}

\newcommand{\vdim}{\on{vdim}}

\newcommand{\dual}{\vee}

\renewcommand{\top}{{\on{top}}}

\newcommand{\Edge}{\on{Edge}}

\newcommand{\Cliff}{\on{Cliff}}

\newcommand{\Ver}{\on{Vert}}

\newcommand{\fin}{{\on{fin}}}

\newcommand{\Aut}{ \on{Aut} }

\newcommand{\Hol}{ \on{Hol} }

\newcommand{\Hom}{ \on{Hom}}
\newcommand{\Ind}{ \on{Ind}}
\renewcommand{\ker}{ \on{ker}}
\newcommand{\coker}{ \on{coker}}

\newcommand{\codim}{\on{codim}}

\newcommand\dirac{/\kern-1.2ex\partial} 
\newcommand\qu{/\kern-.7ex/} 
\newcommand\lqu{\backslash \kern-.7ex \backslash} 

\newcommand\dr{r_+ \kern-.7ex - \kern-.7ex r_-}
 
\newcommand{\labell}\label

\renewcommand{\d}{{\on{d}}}
\newcommand{\ol}{\overline}
\newcommand{\olp}{\ol{\partial}}

\newcommand\Phinv{\Phi^{-1}}

\newcommand\eps{\epsilon}

\newcommand\om{\omega}

\newcommand{\lan}{\langle}
\newcommand{\ran}{\rangle}

\newcommand{\ti}{\tilde}

\newcommand\cE{\mathcal{E}}

\newcommand\cR{\mathcal{R}}
\newcommand\cT{\mathcal{T}}
\newcommand\cF{\mathcal{F}}
\newcommand\cI{\mathcal{I}}

\newcommand\cA{\mathcal{A}}
\newcommand\cP{\mathcal{P}}
\newcommand\cJ{\mathcal{J}}

\newcommand\mE{\mathcal{E}}

\newcommand\curv{\on{curv}}

\newcommand\Map{\on{Map}}

\newcommand\ev{\on{ev}}

\newcommand\Vect{\on{Vect}}
\renewcommand\ul{\underline}
\newcommand\mO{\mathcal{O}}
\newcommand\G{\mathcal{G}}
\newcommand\E{\mathcal{E}}
\newcommand\B{\mathcal{B}}

\newcommand\grad{\on{grad}}
\newcommand\reg{{\on{reg}}}

\newcommand\bdefn{\begin{definition}}
\newcommand\edefn{\end{definition}}
\newcommand\bea{\begin{eqnarray*}}
\newcommand\eea{\end{eqnarray*}}
\newcommand\bcv{\left[ \begin{array}{r} }
\newcommand\ecv{\end{array} \right] }

\newcommand\bma{\left[ \begin{array}{l} }
\newcommand\ema{\end{array} \right]}
\newcommand\ben{\begin{enumerate}}
\newcommand\een{\end{enumerate}}
\newcommand\bex{\begin{example}}
\newcommand\bsj{\left\{ \begin{array}{rrr} }
\newcommand\esj{\end{array} \right\}}

\newcommand\Cone{\on{Cone}}

\newcommand\univ{{\on{univ}}}

\newcommand\eex{\end{example}}
\newcommand\crit{{\on{crit}}}

\newcommand\sx{*\kern-.5ex_X}

\newcommand{\Bl}{\on{Bl}}

\newcommand{\bGamma}{\mathbb{\Gamma}}
\newcommand{\bGam}{\mathbb{\Gamma}}

\newcommand{\cyl}{\on{cyl}}

\def\mathunderaccent#1{\let\theaccent#1\mathpalette\putaccentunder}
\def\putaccentunder#1#2{\oalign{$#1#2$\crcr\hidewidth \vbox
to.2ex{\hbox{$#1\theaccent{}$}\vss}\hidewidth}}

\renewcommand{\ab}{\on{ab}}

\newcommand{\DD}{\mathbb{D}}

\newcommand{\bv}{\mathbb{v}}

\newcommand{\bh}{\mathbb{h}}

\newcommand\cwo[1]{ { \color{darkred}  } }

\newcommand{\diam}{{\mathbin{\lozenge}}}
\makeatletter
\newcommand\bigDiamond{\mathop{\mathpalette\bigDi@mond\relax}}
\newcommand\bigDi@mond[2]{%
  \vcenter{\hbox{\m@th
    \scalebox{\ifx#1\displaystyle 2\else1.2\fi}{$#1\Diamond$}%
  }}%
}
\newcommand\bigLozenge{\mathop{\mathpalette\bigL@zenge\relax}}
\newcommand\bigL@zenge[2]{%
  \vcenter{\hbox{\m@th
    \scalebox{\ifx#1\displaystyle 2\else1.2\fi}{$#1\blacklozenge$}%
  }}%
}
\makeatother

\definecolor{darkred}{rgb}{0.5,0,0}
\definecolor{darkpurple}{rgb}{0.5,0,.5}
\definecolor{darkpink}{rgb}{0,.5,0.5}
\definecolor{cyan}{rgb}{.25,0,0.75}
\definecolor{darkgreen}{rgb}{0,0.5,0}
\definecolor{darkblue}{rgb}{0,0,0.5}

\renewcommand{\diam}{\white}

\begin{document}

\title[Augmentation varieties and disk potentials]{Augmentation varieties and disk potentials I}


\author{Kenneth Blakey}
\address{Department of Mathematics, MIT, 182 Memorial Drive, Cambridge, MA 02139, U.S.A.} \email{kblakey@mit.edu}

\author{Soham Chanda}
\address{Mathematics-Hill Center, Rutgers University, 110 Frelinghuysen Road, Piscataway, NJ 08854-8019, U.S.A.} 
\email{sc1929@math.rutgers.edu}

\author{Yuhan Sun}
\address{Department of Mathematics, Imperial College London, South Kensington, London,
SW7 2AZ, U.K.}
\email{yuhan.sun@imperial.ac.uk}

\author{Chris Woodward}
\address{Mathematics-Hill Center, Rutgers University, 110 Frelinghuysen Road, Piscataway, NJ 08854-8019, U.S.A.}
\email{ctw@math.rutgers.edu}

\thanks{Chanda and Woodward were partially supported by NSF grant DMS 2105417 and Sun was partially supported by the EPSRC grant EP/W015889/1. Any opinions, findings, and conclusions or recommendations expressed in this material are those of the author(s) and do not necessarily reflect the views of the National Science Foundation.}

\begin{abstract} 
This is the first in a sequence of papers where we develop a field theory for Lagrangian cobordisms between 
Legendrians in circle-fibered contact
manifolds.   In particular, 
we show that Lagrangian fillings such as the Harvey-Lawson filling in any dimension define augmentations of Chekanov-Eliashberg differential graded algebras by counting configurations of holomorphic disks connected by gradient trajectories, as in Aganagic-Ekholm-Ng-Vafa \cite{vafaetal}; we also prove that for Legendrian lifts of monotone tori, the augmentation variety  is the zero level set of the Landau-Ginzburg potential of the Lagrangian projection, as suggested by Dimitroglou-Rizell-Golovko \cite{dr:bs}.   
In this part, we set up the analytical foundation of moduli spaces of pseudoholomorphic buildings.
\end{abstract}

\maketitle

\tableofcontents

\section{Introduction}

Contact homology as introduced by Eliashberg, Givental, and Hofer \cite{egh} is a homology theory whose generators come from closed Reeb orbits and whose differential counts holomorphic curves in the symplectization of a contact manifold.  Its relative version, Legendrian contact homology, was introduced by 
Chekanov \cite{chekanov:inv} and 
Eliashberg \cite{eliashberg:icm}, \label{rep:classic} and 
has been developed by Ekholm-Etnyre-Sullivan \cite{ees:lch, ees:orient, ees:leg} for the contactization of Liouville manifolds.  This is 
the first of a sequence of papers in which we develop
a version of Legendrian contact homology in the case that the contact manifold fibers over a monotone symplectic manifold, extending work of Sabloff \cite{sabloff}, Asplund \cite{asplund} and Dimitroglou-Rizell-Golovko \cite{dr:bs}. See also recent developments \cite{KPS, Petr}.  In particular, 
we give a definition of augmentation varieties of Legendrians \label{rep:legen} such as the Legendrian lift of a Clifford torus in any dimension.  

Our definition of the Legendrian contact homology is related to the work 
\cite{vafaetal} cited above.  An important theme in \cite{vafaetal} is the relation between {\em  Lagrangian fillings} of a Legendrian and {\em  augmentations} of its Chekanov-Eliashberg algebra.   \label{rep:filling}
The standard definition of a Lagrangian filling assumes
 a cylindrical end 
modeled on the product of the real line with a Legendrian.    However, \label{rep:important}
in important examples such as the Harvey-Lawson filling considered here, the Lagrangian is only asymptotic to a cylinder over a Legendrian, rather than equal to it outside a compact set. Moreover, the Harvey-Lawson filling is not an exact Lagrangian submanifold, which means it may bound non-trivial holomorphic disks. Therefore,  \cite{vafaetal}
suggests that in order for such Lagrangian fillings to define augmentations of Chekanov-Eliashberg differential graded algebra,  one should count configurations of holomorphic disks connected by  gradient trajectories of certain vector fields. Such configurations are otherwise known as treed disks, clusters (in the language of Cornea-Lalonde 
\cite{cornea:cluster}), or pearly trajectories (in the language of Biran-Cornea \cite{biran-cornea:quantumstruc}).  By counting such configurations, we give a construction of the Chekanov-Eliashberg algebra in which the zeroes of the corresponding vector fields are generators,  \label{rep:vfields}
similar to the definition of immersed Lagrangian Floer homology
developed by Akaho-Joyce \cite{akaho}.  The resulting version of the Chekanov-Eliashberg algebra gives rich Legendrian isotopy invariants which distinguish many Legendrian lifts of Lagrangians in monotone symplectic manifolds such as Vianna's exotic tori \cite{vianna:inf}; the non-isotopy of these tori was conjectured by Dimitroglou-Rizell-Golovko \cite{dr:bs}.   

The use of the treed disk approach resolves an obstruction to fillings giving augmentations, as already noted in   \cite{vafaetal}. Namely, even in simple examples of Lagrangian fillings, one has disk bubbling in the one-dimensional moduli spaces of holomorphic disks bounding the filling.  In the treed disk approach, the moduli space ``continues"  after the formation of a disk-bubble as a moduli space of configurations where two disks are separated by a gradient line, so that the true boundary (in the unobstructed case) consists of configurations with a broken trajectory. In this broken configuration, there are two levels, in the sense of symplectic field theory, meaning that this configuration is again a contribution to differential of the map associated to the filling.  

In this first paper in the series, we construct the moduli spaces of treed holomorphic disks with punctures, and prove results about their compactness, regularization and orientations. In sequels \cite{BCSW2, BCSW3}, these moduli spaces will be used to perform the following tasks: \label{rep:tasks}
\begin{itemize}\label{item1}
    \item Define the Chekanov-Eliashberg algebra, and show the resulting Legendrian contact homology is independent of various choices made.
    \item \label{rep:tame} Define a chain map between the Chekanov-Eliashberg algebras for a tame Lagrangian cobordism between Legendrians (see Definition \ref{def:tamepair}).
    \item Define an augmentation variety for a Legendrian and show it is a Legendrian isotopy invariant.
    \item Establish a concrete relationship between the disk potential functions of certain monotone Lagrangians such as Vianna's tori and the augmentation varieties of their Legendrian lifts,  thus distinguishing infinitely many Legendrians that are mutually not Legendrian isotopic.
    \item Compute the augmentation varieties for the Clifford-type and Hopf-type Legendrians, and the augmentation variety in particular for disconnected Legendrians, in relation to 
    the conjectures of \cite{vafaetal} on the parametrization of components by partitions.
    \end{itemize}

Now we will describe our geometric setup and moduli spaces. Consider a Legendrian $\Lambda$ in a compact contact manifold $Z$ that is a circle-fibration over a symplectic base $Y$ as in Theorem \ref{def:circle}. In particular, we have the following example in mind:  $Z$ is the unit circle bundle of a complex line bundle over a Fano projective toric variety $Y$ with a connection whose curvature is the symplectic form on $Y$, and $\Lambda$ is a horizontal lift \label{rep:horlift} (with respect to the connection) of a Lagrangian in the base. One fundamental example, as in Dimitroglou-Rizell-Golovko \cite{dr:bs}, is a horizontal lift of the Clifford torus
\begin{equation} \label{picliff} 
\Pi_{\Cliff} = \Set{ [ z_1 : \ldots : z_n ] \ | \ |z_i|^2 = |z_j|^2 \text{ for all } i,j } \cong (S^1)^{n-1}
\subset \C P^{n-1} 
\end{equation}
in complex projective space $Y = \CP^{n-1}$. \label{rep:horizlift}
A horizontal lift of the Clifford torus to the unit circle bundle of the tautological line bundle is a Legendrian torus in the standard contact sphere $Z = S^{2n-1}$: 
\begin{equation} \label{cliffleg}    \Lambda_{\Cliff} = \Set{ (z_1,\ldots, z_n) \in \C^n |  \begin{array}{l}
                                                     |z_1|^2 = \ldots
                                                     = |z_n|^2 = \frac{1}{n} \\  z_1 
    \ldots z_n \in (0,\infty)  \end{array} } \cong (S^1)^{n-1} \subset 
  S^{2n-1} .\end{equation}
%
The other lifts are obtained by the $S^1$ action on $S^{2n-1}.$ The map
\[ \C^n  - \{ 0 \} \to \C P^{n-1}, \quad z \mapsto \on{span}(z) \] 
restricts an $n$-fold cover of $\Pi_{\Cliff}$.  
Similar examples are given by taking $\Pi$ to be a Lagrangian in a closed monotone symplectic manifold $Y$, $Z \to Y$ a circle bundle with connection with
rational holonomies, and $\Lambda$ 
a 
horizontal lift of $\Pi$. 
\label{rep:horlift2}

An example of an asymptotically-cylindrical Lagrangian filling is given by the Harvey-Lawson filling of the Legendrian in the previous paragraph.  Let
\[ a_1,\ldots, a_n \in \R \]
be 
real numbers with exactly two being zero.  Define as in Joyce
\cite[(37)]{joyce}
\begin{equation} \label{hl}
L_{(1)} = \Set{ (z_1,\ldots, z_n) \in \C^n | \begin{array}{l} |z_i|^2 - a_i^2
    = |z_j|^2 - a_j^2 , \forall i, j  \\
                                       z_1 \ldots z_n \in
                                       [0,\infty) \end{array} } .\end{equation} 
A picture of the moment image of this filling is shown as the line segment in Figure 
\ref{hlfig} below; the intersections with toric moment fibers are generically tori of codimension one.   Note that this filling is the inverse image of the diagonal in the symplectic quotient  of $X$ whose moment polytope is the triangle shown in Figure \ref{hlfig}.
\label{rep:hl}
To understand this {\em  Harvey-Lawson filling} better, consider the map 
\[ {\pi} : \C^n \to \C, \;\; {\pi}(z_1,\dots, z_n) = z_1z_2\dots z_n \] 
which is a symplectic fibration away from the critical locus. The symplectic horizontal bundle of the fibration ${\pi}$ is the symplectic complement of the vertical bundle. A direct computation shows that the symplectic horizontal bundle $H_z$ at a point $z\in \C^n$ is given as 

\[H_z = \text{Span}_\C \bigg \{ \bigg (\frac{z_1}{|z_1|^2}, \dots, \frac{z_n}{|z_n|^2}\bigg ) \bigg\}.\]
Consider the torus 
    \[ T_1 =\{z \in {\pi}^{-1}(1) \mid |z_1|^2 - a_1^2
    = |z_2|^2 - a_2^2 = \ldots =
                                       |z_n|^2 -a_n^2 \} .\]
Since the derivative of $|z_i|^2 - |z_j|^2$ along any vector in $H_z$ vanishes, we conclude that the parallel transport of $T_1$ along $[0,\infty)$ lies in $L_{(1)}$. Thus $L_{(1)}$ can be viewed as the union of parallel transports of $T_1$ along $[0,\infty).$

We briefly describe the moduli spaces of pseudoholomorphic curves needed for the definition of the differential associated to the Legendrian and the chain maps associated to Lagrangian cobordisms. The construction is a special case of symplectic field theory with 
Lagrangian boundary conditions, where the contact manifold is circle-fibered. 
The simplifications in this case are comparable to that of genus-zero relative Gromov-Witten theory compared to full symplectic field theory with contact boundary.  To orient the reader, we list here two technical parts of our construction:
\begin{enumerate}
    \item the Boothby-Wang contact form \cite{boothby} 
    on the circle-fibered contact manifold is of Morse-Bott non-degenerate type. We introduce a Morse-Bott model for the Chekanov-Eliashberg algebra; and 
    \item we consider more general {\em tame} Lagrangian cobordisms, which may bound non-trivial holomorphic disks, to define chain maps.
\end{enumerate}

Now we explain these two points in more detail. By a {\em  punctured disk} \label{rep:pundisk}  we mean a surface  obtained by removing a finite collection of points on the boundary of a disk.   More precisely, 
our configurations are {\em  treed punctured disks}
obtained by adding a finite collection of trajectories of some vector fields,
on both the space of Reeb chords and the Legendrian itself, as in Figure \ref{tdisk}, and to be explained
below.  In Figure \ref{tdisk}, the dotted lines represent Reeb chords, and the trajectories can either connect the Reeb chords or connect points on the Lagrangian boundary condition of the disks.  The disk shown would contribute to the coefficient of a word of length eight in the differential applied to a Reeb chord.   Note that disk components without punctures (of which one is shown in Figure \ref{treevs2}) appear 
in the configurations defining the chain maps associated to cobordisms.  \label{rep:without}
Such configurations without punctures do not occur as contributions to the differential on our Chekanov-Eliashberg algebras since non-constant holomorphic disks have at least one incoming  puncture, see Corollary \ref{onepos} below.  
In  the sequel \cite[Section 4]{BCSW2}  we assume the existence of suitable bounding \label{rep:boundingcochain}
cochains, as in the theory of \ainfty algebras, so that disk bubbling in our cobordisms does not prevent the counts from defining chain maps.  The strata corresponding to disk bubbling are in the interior of the one-dimensional moduli spaces; instead, the true boundary components arise from breaking of edges of the segments connecting the disks.  We use Cieliebak-Mohnke \cite{cm:trans} styled domain-dependent perturbations to regularize the moduli spaces.   \label{rep:explain}
 \begin{figure}[ht]\label{tdisk}
     \centering
     \scalebox{.7}{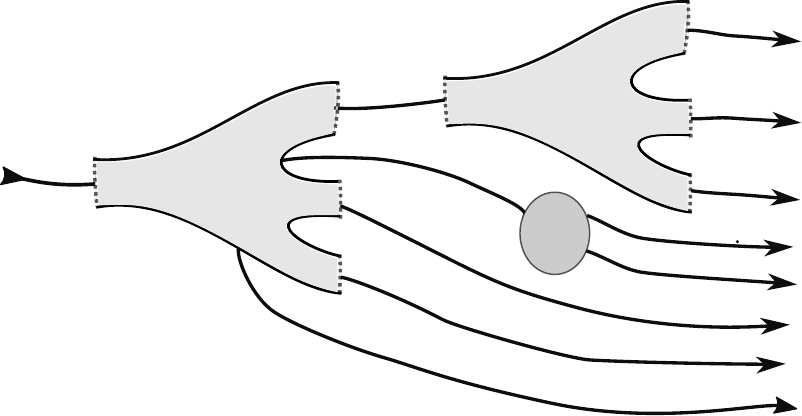}
     \caption{A treed punctured disk}
     \label{treevs2}
 \end{figure}
The chain complex $CE(\Lambda)$
is additively generated by words made of chains on the space of Reeb chords and chains on the Legendrian; the addition
of these {\em  classical Morse generators} on the Legendrian do not occur in the version of Legendrian contact
homology over exact symplectic manifolds in Ekholm-Etnyre-Sullivan \cite{ees:leg}.

The generators arising from chains on the Legendrian arise because of nodes developing 
along points in the interior of the cobordism, which then lead to trajectories as in 
Ekholm-Ng \cite{ekholmng:higher}; in particular, \cite{ekholmng:higher} assigns
augmentations to the Harvey-Lawson filling in the three-dimensional setting of knot contact homology while the construction here is intended to work in all dimensions, under some restrictions. Recall that the unit cosphere bundle of $S^3$ with respect to the round metric is a circle-fibered contact manifold over $\C P^1\times \C P^1$. In particular, the constructions here do apply to the fillings appearing in knot contact homology. It seems likely that in the cases of overlap the two theories agree, as what we develop here is essentially a Morse-Bott version of that theory, but the relationship remains to be determined.

The current paper first (Section \ref{setup}) introduces geometric constructions of  circle-fibered contact manifolds and their Legendrians. Then we set up moduli spaces of treed punctured disks solving the Cauchy-Riemann equation (Section \ref{tocob}), and prove their compactness, regularity and orientability (Section \ref{foundsec}).

\section{Fibered contact manifolds and Lagrangian cobordisms}
\label{setup}

In this section,  we describe various constructions of Legendrians in circle-fibered contact manifolds 
and Lagrangian cobordisms.    The Harvey-Lawson filling is described both as a (non-exact) asymptotic filling and as a compact filling of a Legendrian with respect to a perturbation of the standard contact structure on the sphere. 

\subsection{Fibered contact manifolds and Legendrians}

The contact manifolds we consider are circle fibrations over  compact symplectic manifolds.  Let $n$ be a positive integer and $Z$ a smooth 
manifold of dimension $2n-1$.  Let 
\[ \Omega(Z) = \bigoplus_{j=0}^{\dim(Z)} \Omega^j(Z) \] 
be the algebra of smooth differential forms. 
Recall, as in for example Geiges
\cite{geiges}, the following definitions.

\begin{definition}  Let $Z$ be a $2n-1$-manifold.  A {\em  contact form} on $Z$ is a one-form
\[ \alpha \in \Omega^1(Z), \quad ( \alpha \wedge (\d
\alpha)^{n-1})(z) \neq 0 \quad \forall z \in Z .\]

\vskip .1in \noindent 
The kernel of $\alpha$ is the {\em  contact distribution}
\[ \xi:= \ker(\alpha) \subset TZ \]  
associated to $\alpha$.  The form $\alpha$ defining $\xi$ is unique up to multiplication by a non-zero function.

For a given contact form $\alpha$ on $Z$, the {\em  Reeb vector field} 
\[ R_{\alpha} \in \Vect(Z) \] 
is the unique vector field on $Z$ defined by
\[
\d \alpha \left( R_{\alpha} \right) 
  = 0, \quad \alpha \left( R_{\alpha} \right) 
  = 1.
\]
A {\em  Reeb orbit} is a closed orbit of the Reeb vector field:
\[ \gamma: [0,T] \to Z , \quad \d\alpha \left( \cct \gamma \right) 
  = 0, \quad \alpha \left( \cct \gamma \right) 
  = 1, \quad \gamma(0)= \gamma(T). 
\]

\vskip .1in \noindent 
A {\em  Legendrian submanifold} of $Z$ is a
submanifold $\iota: \Lambda \to Z$ of dimension $n-1$ on which
$\alpha$ vanishes:
\[  \iota^* \alpha  = 0 \in \Omega^1(\Lambda) .\] 
We will assume our submanifolds are embedded, rather than immersed, unless otherwise stated.

\vskip .1in \noindent 
More generally, as in Cieliebak-Volkov \cite{cv:stable}, a {\em  stable Hamiltonian structure} on $Z$ is a pair
of a two-form and a one-form
\[ \omega \in \Omega^2(Z), \quad \alpha \in \Omega^1(Z) \] 
such that 
\[ \alpha \wedge \omega^{n-1} \neq 0, \quad \ker(\omega) \subset \ker(\d \alpha) \subset TZ .\]
A manifold with a stable Hamiltonian structure will be called a {\em  stable Hamiltonian manifold}. In particular, if  $\omega = \d \alpha$, then the pair $(Z,\alpha)$ is a contact manifold.

\vskip .1in \noindent A {\em  Legendrian submanifold} of a stable Hamiltonian
manifold is a maximally isotropic submanifold for $\omega$ on which $\alpha$
vanishes.  This ends the Definition. 
\end{definition}

The particular contact manifolds we study in this article are given by circle bundles, as in Boothby-Wang \cite{boothby}.   To set up our conventions, let $G$ be a compact Lie group (for example, 
the circle group) and $X$ a smooth manifold equipped with a left $G$-action $G \times X \to X$. 
The {\em generating vector field} for an element $\xi \in \g$ is the vector field 
\[ \xi_X: X \to TX, \quad x \mapsto \ddt |_{t = 0} \exp(- t \xi) x .\]
The sign is chosen so that the natural map from the Lie algebra $\g$ of $G$ to the space of vector fields in $X$ is a Lie algebra homomorphism.    A {\em connection one-form} is an invariant 
one-form $\kappa \in \Omega^1(X,\g)$ so that $\kappa(\xi_X) = \xi$ for all $\xi \in \g$.  The 
two-form
\[ \widetilde{ \curv}(\kappa) =  \d \kappa + [\kappa,\kappa]/2  \] 
descends to a two-form $\curv(\kappa)$ on the base $X/G$
taking values in the adjoint bundle $X(\g) := (X \times \g)/G$, where the quotient
is by the diagonal action. 
\label{rep:convent}

These conventions take the following form for circle group actions.  Let 
\[ G = S^1 = \{ e^{2\pi it}, t \in \R \} =\{ e^{i \theta}, \theta \in \R \}  \] 
be the circle group.  We identify $\R$ with the Lie algebra $\g = T_1 G$ by multiplication by $2 \pi i$, and call the image of $1 \in \g$ in $\Vect(X)$ the {\em generating vector field} for the action.    For example, let $S^{2n-1}, n > 0$, be an odd-dimensional sphere viewed as a subset of $\R^{2n} \cong \C^n$.  The generating vector field for the action of the circle group $S^1 = \{ e^{2 \pi it} \}$ on $S^{2n-1}$ by complex multiplication is 
\[ \left( 2 \pi \ppth \right)_{S^{2n-1}} = \ddt |_{t =0} e^{-2 \pi i t} (q_1,\ldots, q_n,p_1,\ldots, p_n) = 2 \pi \sum_{i=1}^n 
 p_i 
\frac{\partial}{\partial q_i} - q_i 
\frac{\partial}{\partial p_i} .\]
A connection one-form for the circle action is 
\[ \kappa = \frac{1}{2 \pi} \sum_{i=1}^n  p_i \d q_i - q_i \d p_i \in \Omega^1(S^{2n-1}) \]
in the sense that $\kappa( (2 \pi \partial_\theta)_{S^{2n-1}}) = 1$.  The two-form $\d \kappa$  descends 
to the opposite of the standard Fubini-Study form on $S^{2n-1}/S^1 = \CP^{n-1}$ normalized
so that the volume of $\CP^{n-1}$ is  the volume of the standard $(n-1)$-simplex
\[ \int_{\CP^{n-1}} \exp( \curv(- \kappa) ) = (n!)^{-1}, \quad \int_{\CP^{n-1}} \curv(-\kappa)^{n-1} = 1. \]
%
Note that the opposite one-form 
\[ \alpha = - \kappa \]
has the property that the exterior derivative is the restriction of the standard symplectic form
to $S^{2n-1}$, up to a factor of $\pi$.   We consider $\alpha$ the standard contact form on $S^{2n-1}$.
More generally, 
such contact forms exist for any integral cohomology class; a detailed proof can be found in \cite[Theorem 7.2.4]{geiges}. 

\begin{theorem}\label{def:circle}
    Let $(Y,\omega_Y)$ be a closed symplectic manifold such that $[\omega_Y] \in H^2(Y,\Z)$ is an integral cohomology class.  There exists a principle $S^1$-bundle $p:Z\to Y$ whose Euler class (i.e., first Chern class) is 
   $  - [\omega_Y] \in  H^2(Y,\Z)$
    %
    together with a connection one-form form $\kappa \in \Omega^1(Z)$ such that:
    \begin{itemize}
        \item $\alpha = - \kappa$ is a contact form on $Z$;
        \item the curvature form of $\kappa$ is $
        - \omega_Y$, that is, $\d\kappa = 
        - p^*\omega_Y$;
        \item and the generating vector field defining the principal $S^1$-action on $Z$ coincides with opposite of the Reeb vector field $R_\alpha$ of $\alpha$.
    \end{itemize}
\end{theorem}


\begin{definition}
A {\em circle-fibered contact manifold} is defined as a circle bundle with contact form $\alpha$ as in Theorem \ref{def:circle}.  A {\em circle-fibered stable Hamiltonian manifold} is a circle bundle $Z$ over a symplectic manifold $Y$ with projection $p:(Z,\alpha)\to (Y,\omega_Y)$, such that 
$(\alpha, p^*\omega_Y)$ is a stable Hamiltonian structure on $Z$.
\end{definition}

\begin{remark}
    The Reeb orbits of the Boothby-Wang contact form on a fibered contact manifold are multiple covers of the circle fibers.   Note that a circle-fibered contact manifold is a special type of circle-fibered stable Hamiltonian structure where $\d\alpha = p^* \omega_Y$.  The two-form $\omega_Y$ determines a circle bundle $Z$, and we want to allow $\alpha$ to vary; see Example \ref{cliffleg3}. 
    \end{remark}

\begin{definition}\label{def:sympl}
For a contact manifold $Z$ with a contact form $\alpha$, the
{\em symplectization of $Z$}  is the symplectic manifold
\begin{equation}
    (\R\times Z, \quad \omega:= \d(e^s\alpha)),
\end{equation}
where $s$ is the coordinate on $\R$. 
\end{definition} 

Often we work with truncations of the symplectization with convex and concave boundary components. 
The contact form $\alpha$ gives a volume form $\alpha\wedge(\d\alpha)^{n-1}$ on $Z$.  For any real numbers $\sigma
_-<\sigma_+$, the submanifold 
\[
([\sigma_-,\sigma_+]\times Z, \quad \omega= \d(e^s\alpha)) \] 
is a symplectic manifold with boundary. 
The induced orientation from the volume form matches the boundary orientation at $\{\sigma_+\}\times Z$ and is the opposite orientation at $\{\sigma_-\}\times Z$. The Liouville vector field, $\partial/\partial s$, points outward along $\{\sigma_+\}\times Z$ and points inward along $\{\sigma_-\}\times Z$. Therefore, we call $\{\sigma_+\}\times Z$ a {\em convex boundary}, and $\{\sigma_-\}\times Z$ a {\em concave boundary} of $[\sigma_-,\sigma_+] \times Z$.

\begin{lemma}\label{rep:proj}
Let $Z$ be a fibered contact manifold.
Any Legendrian submanifold $\Lambda$
in $Z$ projects to a (possibly immersed) Lagrangian
submanifold
\[ \Pi := {p}(\Lambda) \subset
Y \]  
with respect to the symplectic form $\omega_Y$.
\end{lemma}

\begin{proof} By the constant rank theorem, $\Pi \subset Y$ is an immersed
submanifold of dimension $\dim(\Pi) = n = \dim(Y)/2$.
Since $\alpha$ vanishes on $\Lambda$, so does $\d \alpha$.  Thus the symplectic form 
$\omega_Y $ vanishes on $\Pi$ as desired.
\end{proof}

\begin{example} \label{cliffleg3}  The following example arises from truncating the Harvey-Lawson 
Lagrangian by cutting off the complement of the unit ball.  For 
\[ \eps = (\eps_1,\ldots, \eps_n)\in \R^n \] 
let 
\begin{equation} \label{eq:lambdaeps}
    \Lambda_\eps := \Set{\lvert z_1\rvert^2 + \eps_1 = \ldots = \lvert z_n\rvert^2 +\eps_n,  \quad  z_1 \ldots z_n \in
(0,\infty) }  \cap S^{2n-1} . \end{equation}
The subset $\Lambda_\eps$ is a Legendrian submanifold for the stable Hamiltonian structure
$(\alpha_\eps ,p^* \omega_Y)$ defined by a one-form $\alpha_\eps$ on $S^{2n-1} \to \CP^{n-1}$ that is trivial on $\Lambda_\eps$.
For example, one could take 
\[ \alpha_\eps := \frac{1}{2\pi ( 1 + \sum_{i=1}^n \eps_i)} \sum_{j=1}^n (r_j^2 + \eps_j ) \d \theta_j 
\in \Omega^1(S^{2n-1}) \] 
in an open neighborhood of $\Lambda_\eps$. Here we use polar-coordinates $(r_1,\theta_1,r_2,\theta_2,\dots,r_n,\theta_n)$ of $\C^n$.  Indeed, the contraction of the 
vector fields spanning the tangent space of $T \Lambda_\eps$
\[ 2 \pi \left( \frac{\partial}{\partial \theta_i} -  \frac{\partial}{\partial \theta_j} \right) \in \Vect(S^{2n-1}),
\quad i,j \in \{ 1, \ldots, n \} \]
with the one-form $\alpha_\eps$ is 
\[  \iota \left( \frac{\partial}{\partial \theta_i} -  \frac{\partial}{\partial \theta_j}
\right)
\alpha_\eps =  \frac{(r_i^2 + \eps_i )  -  (r_j^2 + \eps_j )}{1 + \sum_{i=1}^n \eps_i}  = 0 \]
for all $i,j$.  If $\eps$ is sufficiently small, then $\alpha_\eps$ is close to $\alpha$
in $C^1$ and thus $\d \alpha_\eps$ is non-degenerate.  So $(\alpha_\eps ,p^* \omega_Y)$ defines a stable Hamiltonian structure, and moreover, $\alpha_\eps$
defines a contact structure, for small $\eps$, whose Reeb field is the infinitesimal field of the Hopf action.
\end{example} 


A condition for a Lagrangian in the base to admit a lift to an embedded Legendrian in the circle fibration is the following condition on the subgroup generated by areas of surfaces bounding the Lagrangian.  

\begin{definition}\label{rep:boundary} Let $Z \to Y$ be a fibered contact manifold and $ \Lambda \to \Pi$ a Legendrian with projection $\Pi$. For any compact oriented surface $u: S \to Y$ with $u(\partial S)\subset\Pi$, let 
\[ A(u) = \int_S u^* \omega  \]
denote its area. Denote by 
\[ \cA(\Pi) = \lan \{ A(u), u: S \to Y,  \ u(\partial S) \subset \Pi \}  \ran \subset \R \]
the {\em  area subgroup} generated by the 
areas $A(u)$.
\end{definition}

\begin{proposition} \label{discrete} Suppose that $p:Z \to Y$ is a fibered contact manifold, $Y$ is simply connected and  $\Pi \subset Y$ is a  connected embedded Lagrangian submanifold with respect to the symplectic form $\omega_Y$ on $Y$.
There exists an embedded Legendrian submanifold $\Lambda$ projecting to $\Pi$
if and only if the area subgroup $  \cA(\Pi) $  is  generated by a rational number.  The projection map from such a Legendrian, $\Lambda \to \Pi$, is a covering map.
Moreover, suppose that the areas in $\cA(\Pi)$ are integer multiples of $1/k$ for some non-zero integer $k \in\Z$,  then the order of any fiber divides $k$.
\end{proposition}

\begin{proof}   The proof is an application of the usual relationship between
the integral of the curvature and holonomy.  Let $y \in \Pi$ be a base point, 
$Z_y \subset Z$ the fiber over $y$ and $z \in Z_y$.
The restriction of $Z$ to $\Pi$ is flat and so 
\[ \ker(\alpha) \subset TZ  \] 
is an integrable distribution on $TZ|_{\Pi}$. 
Let 
\[ \Lambda \subset Z, \quad T \Lambda = \ker(\alpha) |_\Lambda  \] 
be the leaf of the associated foliation
through $z$.   
The holonomy of $\Lambda$ around any loop $\gamma $ bounding a disk
$u: D \to Y$ may be computed as follows.  
Let $\psi$ denote a trivialization  of the pull-back bundle. 
\[ \psi: D \times S^1 \to u^* Z. \] 
The trivializing section is section $\psi(\cdot, 1): D \to u^* Z$.
Then
\begin{eqnarray*} \Hol_{\Lambda}(\gamma) &=&  \exp\left({2 \pi i \int_{\partial D} \psi(\cdot, 1)^* \ti{u}^* \alpha }\right) \\
&=&  \exp \left({2 \pi i \int_{D} \d
\psi(\cdot, 1)^* \ti{u}^* \alpha} \right) \\
&=& 
 \exp\left({2 \pi i \int_D \psi(\cdot, 1)^* \ti{u}^* \d \alpha}\right) \\
 &=&   \exp \left({ 2\pi i \int_D u^* \omega_Y } \right),  \end{eqnarray*}
where $\ti{u}: u^* Z \to Z$ is the natural lift.
Suppose the area of any disk $u: D \to Y$ with boundary on $\Pi$ is divisible by $k$.  
Let $\gamma : S^1 \to \Pi $ be a loop.  Since $Y$ is simply connected, $\gamma$
bounds a disk smooth disk $u: D \to Y$, and the holonomy of $\partial u$ is the exponential
of $2 \pi i$ times the area of $u$.  Since the area is divisible by $k$, the holonomy is an $k$-th root of unity by assumption.  It follows that the fiber $\Lambda_y$ contains the finite orbit $ \Z_{k} z $, so $k$ divides $|\Lambda_y|$.
\end{proof}

\begin{example}\label{ex:degree1}   The Clifford Legendrian of \eqref{cliffleg} is a finite cover of the Clifford Lagrangian  in complex projective space, as follows:  Consider the fibration $S^{2n-1} \to \CP^{n-1}$.
The symplectization of $S^{2n-1}$ may be identified  with punctured affine space 
\[ \R^n \times S^{2n-1} \cong \C^n - \{ 0 \} .\]
The Maslov index two disks bounding the Clifford torus $\Pi_{\Cliff}$
are those induced by the maps to $\C^n - \{ 0 \} $
\[ u(z) = \frac{1}{\sqrt{n}} ( 1,\ldots, 1, z,1,\ldots, 1 ) \]
and have areas $1/n$, as in 
 Cho-Oh \cite{chooh:fano}. For a Clifford Legendrian, the action $z\to \lambda z$ where $\lambda\in\C$ is an $n$th root of unity will generate the deck transformations of the covering $\Lambda_{\Cliff} \to \Pi_{\Cliff}$.
  For example, in the case $n = 2$
the disks correspond to the two hemispheres of $\CP^1 = S^2$ and
the Legendrian lift 
\[ \Lambda_{\Cliff}  = \left\{ |z_1| = |z_2| = \frac{1}{\sqrt{2}} , 
z_1 z_2 \in (0,\infty) \right\} \] 
is a double cover of the equator $\Pi_{\Cliff}$. 
\end{example}

The following are standard constructions on fibered contact/stable Hamiltonian manifolds, that is, 
circle bundles with connection.

\begin{lemma}  \label{lem:constr}
\begin{enumerate} 
\item {\rm (Unions)} Let $Z_1 \to Y_1$ and $Z_2 \to Y_2$ be circle-fibered
contact resp. stable Hamiltonian manifolds of the same dimension.  The disjoint union, written $Z_1 \sqcup Z_2 \to Y_1 \sqcup Y_2$ is also a circle-fibered
contact resp. stable Hamiltonian manifolds. If $\Lambda_1 \subset Z_1$ and $\Lambda_2 \subset Z_2$  are immersed resp. embedded Legendrian
submanifold then the disjoint union $\Lambda_1 \sqcup \Lambda_2 $
is an immersed resp. embedded Legendrian submanifold of $Z_1 \sqcup 
\label{rep:sqcup}
Z_2$.
\item {\rm (Exterior tensor products)} If $Z_1 \to Y_1$ and $Z_2 \to Y_2$ are circle-fibered
contact/stable Hamiltonian manifolds then so is the exterior tensor product 
\[ (Z_1 \boxtimes Z_2) = (Z_1 \times Z_2)/S^1 \to Y_1 \times Y_2 \] 
where the $S^1$-action is the anti-diagonal action. If $\Lambda_1 \subset Z_1$ and $\Lambda_2 \subset Z_2$ are immersed Legendrian submanifolds then the image 
\[ \Lambda_1 \boxtimes \Lambda_2 = (\Lambda_1 \times \Lambda_2)/S^1  \subset Z_1 \boxtimes Z_2 \]
of $\Lambda_1 \times \Lambda_2$ in $Z_1 \boxtimes Z_2$ is also an immersed Legendrian submanifold.  If each projection $\pi_j | \Lambda_j, j = 1,2$ is an embedding, so that in particular $\Lambda_j$ is embedded, then $\Lambda_1 \boxtimes \Lambda_2$ is also embedded.
\item {\rm (Finite covers)}  Let $\Z_m \subset S^1$ denote the finite subgroup 
of order $m$.  If $Z \to Y$ is a circle-fibered contact/stable Hamiltonian manifold then so is the quotient
$Z/\Z_m \to Y$.  Note that the transition maps of $Z/\Z_m$ are the $m$-th powers of those of $Z$ and so the Chern classes $c_1(Z/\Z_m)$ and $c_1(Z)$ in $H^2(Y)$ are related by 
\[  c_1(Z/\Z_m) = m c_1(Z) . \] 
If $\Lambda \subset Z$ is an immersed Legendrian submanifold then the image $\ul{\Lambda}$ of $\Lambda$
in $Z/\Z_m$ is also an immersed Legendrian submanifold.  
\item {\rm (Tensor products)} 
If $Z_1 \to Y$ and $Z_2 \to Y$ are circle-fibered
contact/stable Hamiltonian manifolds  
\[ Z_1 \otimes Z_2 := \Delta^* ( Z_1 \boxtimes Z_2 \to Y \times Y) \] 
where $\Delta : Y \to Y \times Y $ is the diagonal,  is a circle-fibered
contact manifold over $Y$.   As a special case, 
if $Z \to Y$ is a circle-fibered contact/stable Hamiltonian 
manifold then so is the $k$-fold tensor product 
\[ Z^{\otimes k} := (Z \times_Y \ldots \times_Y Z)/(S^1)^{k-1} . \]
\item {\rm (Symplectic quotients)} 
Let $Z \to Y$ be a circle-fibered contact/stable Hamiltonian
manifold with an action of a Lie group $H$
preserving the contact form $\alpha$ and commuting
with the action of $S^1$. 
Let $\h$ be the Lie algebra of $H$ and $\h^\dual$
its dual.  The {\em  moment map} for the action 
of $H$ is the map 
\[ \Phi: Z \to \h^\dual, \quad z \mapsto ( \xi \mapsto \alpha(\xi_Z(z))) .\]
The map 
$\Phi$ is the pull-back of a moment map $\ol{\Phi}: Y \to \h^\dual$ for the action of $H$ on $Y$.   The symplectic quotient of $Z$ is 
\[ Z \qu H = \Phinv(0) / H .\]
Suppose $H$ acts properly and freely on $\ol{\Phi}^{-1}(0)$.  Then $Z \qu H$ is a circle-fibered contact manifold over 
\[ Y \qu H := \ol{\Phi}^{-1}(0)/H . \] 
Let $\Lambda \subset Z$ be an $H$-invariant
immersed resp. embedded Legendrian submanifold.  Then $\Lambda$ is contained in $\Phinv(0)$ and $\Lambda/H$ is an immersed resp. embedded Legendrian submanifold of $Z \qu H$.   
\end{enumerate}
\end{lemma} 
\label{rep:proper}

\begin{proof} Most of these claims are immediate and left to the reader. 
In the case of products, suppose that the connection forms
are $\alpha_1 \in \Omega^1(Z_1)$ and 
$\alpha_2 \in \Omega^1(Z_2)$.  The one-form 
\[ \pi_1 \alpha_1^* + \pi_2 \alpha_2^* \in \Omega^1(Z_1 \times Z_2) \]
is basic (that is, invariant and vanishes on the vertical subspace) for the anti-diagonal $S^1$ action.  Thus it descends to a one-form on $Z_1 \boxtimes Z_2$.  An easy check shows that this one-form is contact.  Since  $\pi_1 \alpha_1^* + \pi_2 \alpha_2^* $
vanishes on $\Lambda_1 \times \Lambda_2$, its image $\Lambda_1 \boxtimes \Lambda_2$ is Legendrian. 
\end{proof}

\begin{example}\label{ex:degree2} Continuing Example \ref{ex:degree1}, let $\Pi = \Pi_{\Cliff}$ be the Clifford Lagrangian in $\CP^{n-1}$.
Let  $Z \to \CP^{n-1}$ be the unit canonical bundle, that is, 
the set of unit vectors in the total space of the line bundle 
\[ \mO(-n) = \Lambda_\C^{n-1} T \CP^{n-1}  \] 
with respect to the metric induced
by the Fubini-Study metric on $\CP^{n-1}$.  Any horizontal lift $\Lambda$
\label{rep:horlift3} is diffeomorphic to $\Pi$ under the projection; the horizontal
lift $\Lambda \to \Pi$ is defined by setting
\[ s(y) \in \Lambda^{n-1}_\R T_y \Pi 
\cong \Lambda^{n-1}_\C T_y Y \] 
the unique unit volume form defining the orientation on $\Pi$.

\end{example}

Now we analyze the space of Reeb chords for the circle-fibered case we are considering. 
A {\em  Reeb chord} with boundary on a Legendrian $\Lambda$ is a path starting and ending at the Legendrian:
  \[ \gamma: [0,T] \to Z , \quad \d\alpha \left( \cct \gamma \right) 
  = 0, \quad \alpha \left( \cct \gamma \right) 
  = 1, \quad \gamma( \{ 0, T \}) \subset \Lambda . \]
The number $T$ is the {\em  action} of the Reeb chord. Recall our convention for the action of a simple Reeb orbit is $1$. Therefore, the action of a Reeb chord is simply the change in angle, divided by $2\pi$, in the fiber. Denote by ${\cR}(\Lambda)$ the set of Reeb chords:
\[ 
{\cR}(\Lambda) = \Set{ \gamma: [0,T]
\to Z_q | \; T>0, \;q\in \Pi, \; 
\alpha \left( \cct \gamma \right) 
  = 1, \ \gamma( \{ 0, T \}) \subset \Lambda }   . \]
Denote by $\cR(Z)$ the space of Reeb orbits, 
the space of maps $\gamma: [0,T] \to Z_q$
with $\alpha \left( \cct \gamma \right) 
  = 1$ and $\gamma(0) = \gamma(T)$.

\begin{lemma} \label{startend} Suppose that as above $\Lambda \subset Z$ is a compact Legendrian submanifold
whose image $\Pi \subset Y$ is an embedded Lagrangian submanifold.  There is a diffeomorphism 
${\cR}(\Lambda) \to (\Lambda \times_\Pi \Lambda) \times \Z_{\ge 0 }$.
\end{lemma}

\begin{proof} Given a Reeb chord $\gamma$ of action $T$,
 the starting and ending point $\gamma(0),\gamma(T)$ of the chords together with the number of times 
 $\gamma(t)$ crosses the starting point $\gamma(0)$ as $\gamma$ winds around the fiber
$Z_{p(\gamma)}$ uniquely specifies the chord. 
\end{proof}

\begin{corollary}  If $\Pi$ is embedded then each component of $\cR(\Lambda)$
is diffeomorphic to $\Lambda$.  If $\Pi$ is immersed with transverse self-intersection then the components are either points or diffeomorphic to $\Lambda$.
\end{corollary}

\begin{proof}  First consider the case that the Lagrangian in the base is embedded.  The map 
\[ \cR(\Lambda) \to \Lambda, \gamma \mapsto \gamma(0)  \] 
taking the initial point of a chord is discrete covering map, with the fiber given by the discrete set of Reeb chords   with the same initial point.  Parallel transport of Reeb chords using the given flat connection  over $\Pi$ provides a trivialization of this discrete covering, so the covering is trivial.    In particular, each connected component of $\cR(\Lambda)$ is diffeomorphic to $\Lambda$.

Next, consider the case that the Lagrangian in the base is immersed.  There are two diffeomorphism types among the connected components of $\cR(\Lambda)$, depending on whether two  the end points $\gamma(0), \gamma(1)$ of the Reeb chord belong to horizontal lifts of two different branches of $\Lambda$ at the self intersections of $\Pi$:
If a Reeb chord $\gamma$ has end points $\gamma(0),\gamma(1)$ on lifts of different branches
of $\Pi$ then the connected component of $\cR(\Lambda)$ containing the chord $\gamma$ is the singleton set $\{ \gamma \}$, by transversality,  If Reeb chord $\gamma(0), \gamma(1)$ has end points on the same the same branch
of $\Lambda$ then the argument is similar to the case when $\Pi$ is embedded.
\end{proof}

\subsection{Lagrangian cobordisms}

In this section, we define symplectic and Lagrangian cobordisms and introduce assumptions that guarantee that the counts of pseudoholomorphic disks give rise to chain maps between the Chekanov-Eliashberg algebras.

\begin{definition}\label{sympcob}
Let $(Z_{\pm}, Y_{\pm}, \alpha_{\pm},\omega_\pm)$ be two circle-fibered stable Hamiltonian manifolds.
A {\em  symplectic cobordism} with convex end $(Z_{+}, Y_{+}, \alpha_{+},\omega_+)$ and concave end $(Z_{-}, Y_{-}, \alpha_{-},\omega_-)$ is a symplectic manifold with boundary $(\widetilde{X},\partial \widetilde{X}, \omega)$ whose boundary is equipped with a partition
\[ \partial \widetilde{X}= \partial_{+} \widetilde{X}\cup \partial_{-} \widetilde{X} \] 
and diffeomorphisms 
\[\iota_\pm : Z_\pm \to \partial_\pm \widetilde{X}  \]
so that
\begin{enumerate} 
\item the two forms satisfy the proportionality relation
\[ \iota_\pm^* \omega\mid_{\partial_{\pm} \tilde{X}}=  \lambda_\pm  \omega_\pm
\]
for some positive constants $\lambda_\pm \in \R$, and 
\item the 
isomorphism  between the normal bundle of $Z_\pm$
in $\widetilde{X}$ and the kernel of $D p: TZ_\pm \to TY_\pm$
induced by the symplectic form is orientation preserving resp. reversing for the convex resp. concave boundary. 
\end{enumerate}
Denote by $X$ the interior of $\widetilde{X}$, viewed as a manifold with cylindrical ends, and, abusing terminology, 
call $X$ the cobordism. 
\end{definition}

\begin{remark} In the definition of cobordism
the one-forms $\alpha_\pm \in \Omega^1(Z_\pm)$ are not required to extend over $X$, and will not,
in our basic example of the Harvey-Lawson filling.    For the most part, the reader could 
restrict to the case that $Z_\pm$ are contact manifolds, keeping in mind that the truncation of 
asymptotically-cylindrical manifolds such as the Harvey-Lawson filling are only Legendrian
with respect to stable Hamiltonian structures on the boundary of the truncation.
\end{remark}

\begin{remark}  Because of the assumption on the orientations, the notion of symplectic cobordism
is not symmetric. 
\end{remark}

\begin{example} \label{ex:sympl} Truncations of symplectizations are particular examples of symplectic cobordisms.
Recall the symplectization of a contact manifold with one-form $(Z, \alpha)$ is the symplectic manifold
\[
\R \times Z, \quad \omega:=  \d(e^s\alpha)
\]
where $s$ is the coordinate on $\R$.  
For real numbers $\sigma_- < \sigma_+$
The truncation  
\[ X  = (\sigma_-,\sigma_+) \times Z \] 
is a symplectic cobordism. The convex end has a larger volume and the concave end has a smaller volume. The proportional constants $\lambda_\pm$ depend on the cut levels $(\sigma_-,\sigma_+)$. Usually 
we omit these in the notation.
\end{example} 

\begin{definition}  The {\em  compactification} of a symplectic cobordism $X$ is the symplectic manifold $\ol{X}$
 obtained by  collapsing the null foliation on the boundary:
 \begin{equation} \label{equiv} \ol{X} = \widetilde{X} / \sim  \quad = (\widetilde{X} - \partial \widetilde{X}) \sqcup (\partial \widetilde{X}/S^1) \end{equation}
 where $\sim$ is the trivial equivalence relation on $\widetilde{X} - \partial \widetilde{X}$ and the equivalence relation 
 \[  z \sim e^{2\pi i \theta} z,\quad  z \in Z_\pm, \theta \in \R \]
 generated by the circle action on the boundary.  
\end{definition} 

\begin{example} \label{ex:symp2} Continuing Example \ref{ex:sympl}, the compactification of the truncated symplectization can be described as an associated bundle.  Let $X = (\sigma_-,\sigma_+)\times Z$ be the symplectization of a contact manifold $Z$.  The compactification $\ol{X}$ is diffeomorphic to the quotient of $Z \times \P^1$ by the diagonal $S^1$-action.     Indeed, 
the quotient of $Z \times (\P^1 - \{ 0, \infty \})$ is diffeomorphic 
to the quotient of $Z \times (\R \times S^1)$, where the latter is $\R \times Z$.
Adding the sections at zero and infinity produces the compactification.
\end{example}


\begin{lemma} \label{lem:trunc} Let $X$ be a symplectic cobordism with concave end $Z_-$ and convex end $Z_+$. The compactification $\ol{X}$ is a smooth symplectic manifold containing $Y_-$ and $Y_+$ as codimension-two symplectic submanifolds. 
\end{lemma}

     The statement of the Lemma is a special case of the symplectic cut construction in Lerman \cite{le:sy}:   Given a symplectic manifold 
     $X$ with a proper Hamiltonian $S^1$-action with moment map $\Phi$, consider the product $X \times \C$
     with the diagonal $S^1$-action with moment map 
     \[ \ti{\Phi}: X \times \C \to \R, \quad ( s,z,w) \mapsto e^s \pm |w|^2/ 2 .\]
     The quotient of the zero level set of the moment map is the one-sided cut
     %
     \[   X_{\ge 0} := X \times \C \qu S^1 \cong  \Phinv(0)/S^1  \cup \Phinv( \R_{> 0})  . \]
     We will use the following notation for two-sided cuts:  Suppose that the standard action of $S^1$ on $S^2$ is equipped with a symplectic form wand moment map $\Phi_{S^2}: S^2 \to \R$ with moment image $-[a_-,a_+]$.  Define 
     \[ X_{[a_-,a_+]} = (X \times S^2) \qu S^1 
     \Phinv(a_-)/S^1  \cup \Phinv( (a_-,a_+) )  \cup \Phinv(a_+)/S^1 . \]
     If the action is free then the quotient is smooth by the Meyer-Marsden-Weinstein theorem.    

\begin{proof}[Proof of Lemma \ref{lem:trunc}]
     Let $U_\pm \subset X$ be an open neighborhood of $Z_\pm$ in $X$, which may be identified with a subset of $\R \times Z_\pm$ by the coisotropic embedding theorem.   The compactification $\ol{X}$ is obtained from $\R \times Z_\pm$   near the boundary  by taking the symplectic quotient by the given free action. 
\end{proof}

\begin{example}   
We continue our basic example from Example \ref{ex:degree1}.
    Let $p:Z= S^{2n-1}\to Y=\C P^{n-1}$ be the fibered contact manifold given by the Hopf fibration. Let $X=(\sigma_-,\sigma_+)\times Z$ be the truncated symplectization. The compactification $\ol{X}$ is a $\CP^1$-bundle over $\C P^n$. The Euler class of the normal bundle $N_{+,1}$ of $Y_+$ is the hyperplane class and the Euler class of $N_{-,1}$ is the opposite of the hyperplane class. Recall the Euler class of $Z_\pm$ is the negative hyperplane class at both convex and concave ends.
\end{example}

Conversely, closed symplectic manifolds equipped with disjoint symplectic hypersurfaces induce cobordisms, under the assumption of suitable positivity/negativity of the normal bundles.

\begin{lemma}  
Let $(\ol{X},\ol{\om})$ be a closed symplectic manifold  with two disjoint symplectic hypersurfaces $Y_{\pm}$. Let $\omega_{Y,\pm}$ be the restricted symplectic forms from $\ol{\omega}$. Consider the unit normal bundles
  \[ Z_\pm  := N_{\pm,1} \to Y_\pm .\]
Suppose that the Euler class of $N_{+,1}$ is negatively  proportional to $[\omega_{Y,+}]$ and the Euler class of $N_{-,1}$ is positively proportional to $[\omega_{Y,-}]$. Then the real blow up of $\ol{X}$ at $Y_\pm$ has the structure of a symplectic cobordism with concave end $(Z_-,\alpha_-,\lambda_-p_-^* \om_{Y,-})$ and convex end $(Z_+,\alpha_+, \lambda_+p_+^* \om_{Y_+})$ for some positive constants $\lambda_\pm$.
\end{lemma}

\begin{proof}  The real blow-up $\ti{X}$ is obtained by replacing the submanifolds $Y_\pm$ by their unit normal bundles.  There is a natural  blow-down map from $\ti{X}$ to $\ol{X}$.  The two-form $\ti{\om}$ on $\widetilde{X}$ is obtained by pull-back of the form $\ol{\om}$ on $\ol{X}$.
\end{proof}

We view the interiors of the cobordisms as manifolds with cylindrical ends.  Denote the interior of the cobordism  
\[ X = \widetilde{X} - \partial \widetilde{X} ; \]
%
recall that, abusing terminology, we also refer to $X$ as a cobordism with concave end $Z_-$ and convex end $Z_+$.  Tubular neighborhoods give rise to proper embeddings for some $\sigma_\pm > 0 $
 \[ \kappa_{\pm}:  \pm (0,\sigma_\pm) \times Z_\pm \to X .\] 
Equip $ \mp [0,1] \times Z_\pm$ with the two form 
\[ \omega_\pm := p_{Y_\pm}^* \omega_{Y_{\pm}} + \d ( s \alpha_\pm) .  \] 

\begin{lemma}  Possibly after shrinking $\sigma_\pm$, the restrictions of $\omega_\pm$ to $ \mp [0,\sigma_\pm] \times Z_\pm$  are symplectic and the maps $\kappa_{\pm}$ may be taken to be symplectomorphisms onto open subsets of $X$.
\end{lemma}

\begin{proof} The local model for constant rank embeddings in Marle \cite{marle} implies the existence of the claimed symplectomorphism (since the two forms, by definition, agree at $Z_\pm \times \{ 0 \}$).
    \end{proof}
    
\begin{definition} \label{def:lagcob}
For a symplectic cobordism $X$, with convex end $Z_{+}$ and concave end $Z_{-}$ so that $X$ is equipped with cylindrical ends, and two Legendrian submanifolds $\Lambda_{\pm}\subset Z_{\pm}$, a {\em  Lagrangian cobordism cylindrical near infinity} with convex end $\Lambda_{+}$ and concave end $\Lambda_{-}$ is a Lagrangian submanifold $L\subset X$ such that
\[ L\cap (\mp (0,\sigma_\pm)\times Z_\pm)= \mp (0,\sigma_\pm) \times \Lambda_\pm \]
for some intervals $\mp (0,\sigma_\pm) \subset \R$.
We call the pair $(X,L)$ a {\em  cobordism pair}.
\end{definition}

Given a Lagrangian cobordism that is cylindrical near infinity, we form a (singular) Lagrangian in the compactified symplectic cobordism
by taking closure.

\begin{definition}\label{def:clean}
     An immersed submanifold $\iota : K \to \ol{X}$ is  {\em  cleanly self-intersecting} \label{rep:cleanly} if the fiber product $K \times_{\iota,\iota} K \subset K \times K$ has the structure of a smooth submanifold so that at any points $k_1,k_2 \in K$ with 
$  \iota(k_1) = \iota(k_2)  = x $ 
\[ D \iota (T_{k_1} K )
\cap D \iota (T_{k_2}  K)  =   D\iota( T_{(k_1,k_2)}  ( K \times_{\iota,\iota} K) )  .\]
\end{definition}

Let $L \subset X$ be a Lagrangian cobordism and $\ol{L} \subset \ol{X}$ its closure.

\begin{proposition} \label{localmodel} \label{lem:pin} Let $(X,L)$ be a cobordism pair with concave end $(Z_-,\Lambda_-)$ and convex end $(Z_+,\Lambda_+)$ in the sense of Definition \ref{def:lagcob}.  There exists a manifold $\tilde{L}$ with boundary and an immersion $\phi: \tilde{L} \to \ol{X}$ with clean self-intersection so that $\phi(\phi^{-1}(X)) = L$.
\end{proposition} 

\begin{proof} To simplify notation we assume that $Z_-=\emptyset$ and write $Z = Z_+, Y = Y_+$; the proof for the case $Z_- \neq \emptyset$ is identical. The proof is based on taking a real-oriented blow-up of $Y$ in $\ol X$ and compactifying the Lagrangian $L$ inside the blow-up to get $\tilde L$.
Suppose that $\Lambda$ intersects some fiber
  $Z_y\cong S^1$ over $y \in Y$ in a subset
  $\Lambda \cap Z_y = \{ \theta_1,\ldots, \theta_k  \} . $
If we identify $Z_y \times_{\C^\times} \C \cong \C$ then 
the intersection of $L$ with a neighborhood $U_y \subset X$ of 
$y \in Y$ is given in the local model by 
\[
  \R_+ e^{2\pi i \theta_1} \cup \ldots \cup \R_+  e^{2\pi i \theta_k}  
\subset \C \cong 
 Z_y \times_{\C^\times} \C  .\] 
The union 
 \[ \hat{L}_y = (-\eps, \eps) e^{2\pi i \theta_1} \cup \ldots \cup (-\eps, \eps) e^{2\pi i \theta_k}  \]
 for $\eps$ small is a non-compact,
 cleanly self-intersecting Lagrangian containing $L \cap U_y$.
 Let $\hat{L}$ be the union of the local models 
 $\hat{L}_y, y \in Y$ with $L$.   Notice that  the local presentation above implies that the closure of $L$ in the real-oriented blow-up of $Y$ is a manifold $\tilde{L}$ with boundary. 
\end{proof}

\begin{remark} 
In particular, Proposition \ref{localmodel} implies that  $\ol{L}$ is a closed subset of a cleanly self-intersecting Lagrangian 
$\hat{L} \to \ol{X}$. 
\end{remark}

\begin{definition}\label{def:branchedLag}  Let $\ol{X}$ be a compact symplectic manifold.  
  A compact subset $\ol{L} \subset \ol{X}$ has a {\em  branch point} at $p \in \ol{L}$
  if there exists a Darboux chart 
\[ \phi:U \to (\C^n,\omega_0) \] 
where $U$ is an open neighborhood of $p$ and $\phi$ is a symplectomorphism sending $p$ to $0$, such that $\phi(\ol{L})$ is the union of a collection of disjoint open subsets of a finite collection 
\[ W_1,\dots, W_k \subset \C^n \] 
of  Lagrangian subspaces of $(\C^n,\omega_0)$ and an isotropic subspace $W_\cap \subset \C^n$ such that $\cap_i W_i = W_\cap$.  A subset $\ol{L} \subset \ol{X}$ is a {\em  branched Lagrangian} if for any $p \in \ol{L}$, either $\ol{L}$ is a smooth Lagrangian in an open neighborhood of $p$ or $p$ is a branch point of $\ol{L}$ as defined above.
\end{definition}

We have in mind the case that the Lagrangian is locally homeomorphic to the product of a smooth (lower dimensional) Lagrangian with a union of  half-lines ending at the origin in the complex plane, as in Figure \ref{lmod}.
 \begin{figure}[ht]
     \centering
     \scalebox{.2}{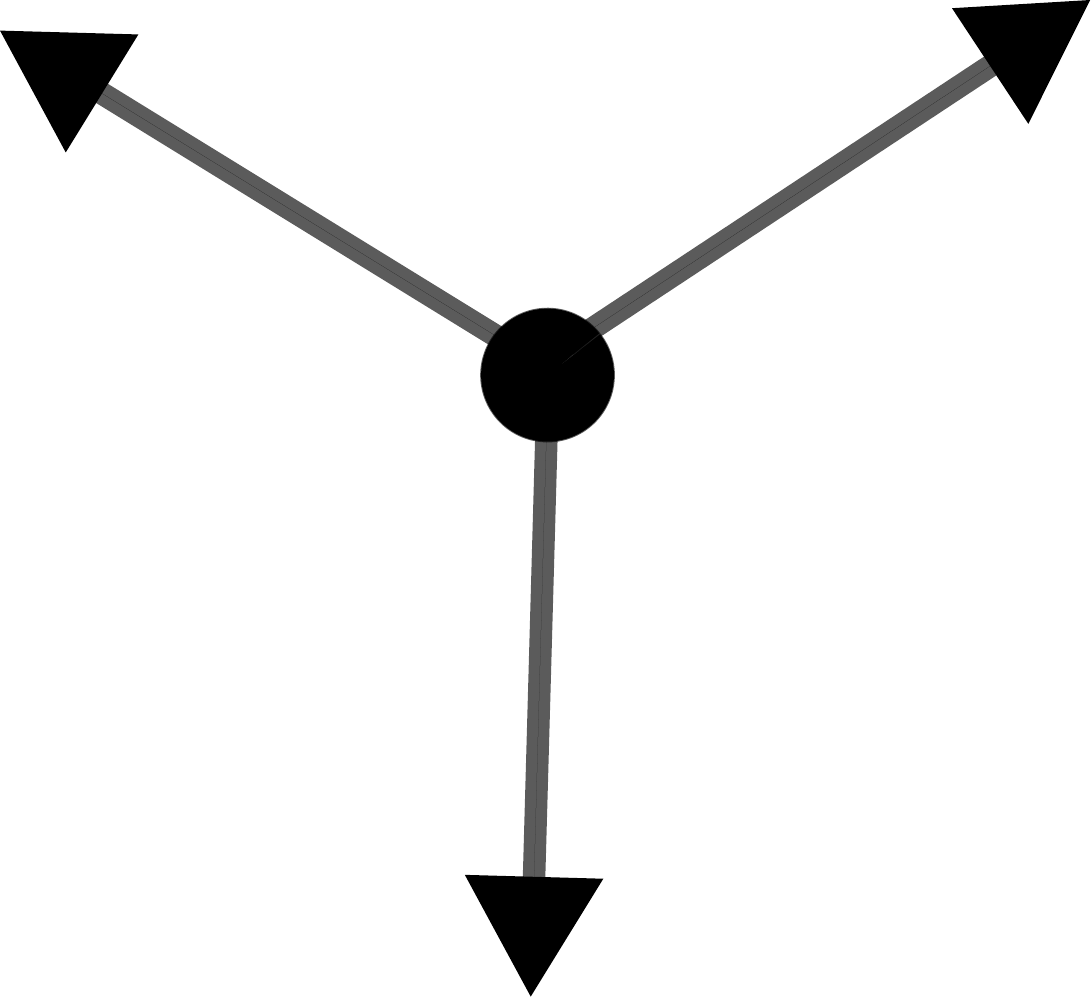}
     \caption{Local model for the Lagrangian $\ol{L}$ near $\Pi$}
     \label{lmod}
 \end{figure}
If $\ol{L}  \subset \ol{X}$ is a branched Lagrangian then the set of branch points
$L_\cap$ is an isotropic submanifold, 
since $\ol{L}$ is locally modelled on the subspace $W_\cap$ near any branch point.
We may consider $\ol{L} = L \cup L_\cap$
as a stratified space with top-dimensional stratum being the set $L$ of smooth points.

\begin{lemma} \label{isotopy}
Let $(Z,\alpha)$ be a fibered contact manifold. If $\Lambda_\pm$ are isotopic Legendrians in $(Z,\alpha)$ then there exists
an exact Lagrangian cobordism with convex end $\Lambda_+$ and concave end $\Lambda_-$ that is diffeomorphic to $\R \times \Lambda_\pm$.
\end{lemma} 

\begin{proof} (c.f. Chantraine \cite{chantraine:concordance} and Golovko \cite[Example 4.1]{golovko:note}.)   
We use a Legendrian isotopy to produce an isotopy of contact one forms, and from this a Lagrangian cobordism. 
By definition,  a Legendrian isotopy from $\Lambda_-$ to $\Lambda_+$ consists of a family of Legendrians
\[ \Lambda_\rho \subset Z, \rho \in \R, \quad \Lambda_{\rho} = \Lambda_-, \rho \ll 0, \quad \Lambda_{\rho} = \Lambda_+, \rho \gg 0 .\]
Choose a family of diffeomorphisms 

\[ \psi_s: Z \to Z, \quad  s \in \R \]
so that 
\[ \psi_s(\Lambda_-) = \Lambda_s , \forall s, \quad \partial_s \psi_s = 0 , |s| \gg 0 .\]
We obtain an isotopy of contact forms 
\[ \alpha_s = \psi_s^* \alpha \in \Omega^1(Z). \]

Using the isotopy of contact forms we construct a symplectic cobordism as follows.  
Let  
\[ \ti{\alpha} = (\alpha_s)_{s \in \R} \in \Omega^1(\R \times Z) \] 
and define
\[ \omega := \d_{\R \times Z} ( e^s \ti{\alpha} ) =  (\d s)\wedge (  e^s \alpha_s + e^s \partial_s \alpha_s)  + 
e^s \d_Z \alpha_s \in \Omega^2( \R \times Z) .\]
To show that $\omega$ is symplectic, note that the top exterior power is
\begin{equation}\label{topofnewsymp}
     \omega^n =  n (\d s) \wedge (  e^s \alpha_s + e^s \partial_s \alpha_s) \wedge ( \d_Z \alpha_s)^{n-1}.
\end{equation}
By taking  $\alpha_s$ to be slowly varying, we may assume that the following inequality holds point-wise with respect to some Riemannian metric on the bundle of forms:
\begin{equation}\label{derofalhpaineq}
    \| \d s \wedge \alpha_s \wedge \d_Z (\alpha_s)^{n-1} \| > \|\d s \wedge \partial_s\alpha_s \wedge \d_Z (\alpha_s)^{n-1} \|.
\end{equation}
\noindent Since $ \d s \wedge \alpha_s \wedge \d_Z (\alpha_s)^{n-1}$ is non-vanishing, equations (\ref{topofnewsymp}) and (\ref{derofalhpaineq}) imply that $\omega^n$ is non-vanishing. Thus, $\omega$ is a symplectic form on $\R \times Z$.   

Finally, we give the Lagrangian cobordism.  Define an embedding
\[ \iota: \R \times \Lambda_- \to \R \times Z, \quad (s,\lambda) 
\mapsto (s, \psi_s(\lambda)) .\]
The pull-back one-form is 
\[   \iota^* e^s  \ti{\alpha} = 0  \] 
so the image of the embedding is an exact Lagrangian.  
Thus $(\R \times Z, \omega)$ is a symplectic manifold containing $\iota(\R \times \Lambda_-)$ as a Lagrangian submanifold.   The diffeomorphisms $\psi_\rho$ identify the ends with the symplectizations 
$ \pm (|\rho|, \infty) \times Z$ for $|\rho|$ sufficiently large.    
By symplectic cutting on $\R \times Z$ as in Lemma \ref{lem:trunc}, we obtain a compact symplectic manifold $\ol{X}$ containing $Y_\pm$ as symplectic submanifolds, up to re-scaling of the 
symplectic forms $\omega_\pm$.  
\end{proof}

\label{asym}

Our motivating example of the Harvey-Lawson filling as described in \eqref{hl} is only asymptotically cylindrical rather than cylindrical near infinity, which we now explain. Recall that the space of submanifolds of a fixed dimension has a natural topology in which open neighborhoods
of a given submanifold are those submanifolds  corresponding to smooth sections of the normal 
bundle contained in some tubular neighborhood.  Choose tubular neighborhoods of $\Lambda$ in $Z$ given by 
diffeomorphisms of the normal bundle $N\Lambda$ onto an open neighborhood of $\Lambda$.
The tubular neighborhoods  of $\Lambda$ induce  tubular neighborhoods of $\R \times \Lambda$ given by maps
$  \R \times N \Lambda \to \R \times Z$. Let 
\[ \tau_{\sigma}: \R \times Z \to \R \times Z, \quad (s,z)
\mapsto (s+ \sigma,z) \] 
be the translation map. 

 \begin{definition}  \label{def:conv}
 Let $(Z,\alpha)$ be a fibered contact manifold.  Let $L_\nu \subset \R \times Z, \nu \in \N$
 be a sequence of submanifolds given as the graphs of sections %
\[ s_{L_\nu}: \R \times \Lambda \to \R \times N \Lambda \]
via the tubular embedding of $\R \times N\Lambda$ in $\R \times Z.$
The sequence $L_\nu$
$C^\ell$-{\em  converges} to $(\sigma_-,\sigma_+) \times \Lambda$ if  the sequence of sections $s_{L_\nu}$ converges to $0$  in the $C^\ell$-norm uniformly on  $(\sigma_-,\sigma_+) \times \Lambda$. 

 \vskip .1in \noindent A non-compact Lagrangian $L\subset X$ is {\em  asymptotically cylindrical} if 
the submanifolds 
\[ \tau_{-\sigma} ( L \cap ((\sigma,\sigma+1) \times Z)) \subset (0,1) \times Z \] 
converge in $C^1$ on $(0,1) \times Z$ in the sense 
above.   This ends the Definition.
\end{definition}

\begin{remark}
    Note that the Definition \ref{def:conv} of convergence does not depend on the choice of a tubular neighborhood obtained from a tubular neighborhood of $\Lambda$ in $Z$ since $\Lambda$ is compact.
\end{remark}

\begin{example} \label{cliff2} For the Clifford Legendrian $\Lambda_{\Cliff}$ in 
\eqref{cliffleg}, a natural asymptotically cylindrical filling was introduced by Harvey-Lawson
\cite{hl:cal}, see also Joyce \cite{joyce}.  Let
\[ a_1,\ldots, a_n \ge 0  \] 
\label{rep:nonneg}
be non-negative constants exactly two of which are zero (without loss of
generality, the last two $  a_{n-1} = a_n = 0 $.)
The {\em  Harvey-Lawson asymptotic filling}, 
denoted $L_{(1)}$, of the Clifford Legendrian is
given by \eqref{hl}, and is asymptotic to $\R_{> 0} \times \Lambda_{\Cliff}$ at
infinity in $\C^n$.   

By truncation we obtain fillings of Legendrians in the spheres of  fixed norm. 
Namely, each intersection with a ball $L_{1} \cap B_{e^\sigma}(0)$ fills the Legendrian
\[ \Lambda_\sigma = \left\{ \begin{array}{l}  
|z_1|^2 - a_1^2 = |z_2|^2 - a_2^2 = \ldots  = |z_n|^2  - a_n^2 \\ 
z_1 z_2 \ldots z_n \in (0,\infty) \end{array} \right\}  \cap e^\sigma S^{2n-1}  .\]
The submanifold $\Lambda_\sigma$ is a Legendrian with respect to the extension of the contact form $\alpha_\epsilon$ as defined in Example \ref{cliffleg3}. The projection is a Lagrangian torus orbit 
\begin{equation} \label{rather} \{ |z_1|^2 - a_1^2 = |z_2|^2 - a_2^2 = \ldots = |z_n|^2 -a_n^2  \}  \cap e^\sigma S^{2n-1} / S^1  .\end{equation}
This Lagrangian is a toric moment fiber over a point close, but not equal, to the barycenter $e^\sigma (1,1,\ldots, 1)/n$ of the moment polytope
$\Delta_{n-1}$ of $\CP^{n-1}$
\begin{equation} \label{rep:delta}
\Delta_{n-1} = \{ \lambda_1 + \ldots + \lambda_n = 1 \} \cap \R_{\ge 0}^n . \end{equation}
In particular, if $L$ is an asymptotically cylindrical Lagrangian with limits $\Lambda_\pm$ that have monotone
projections to $Y_\pm$, the truncations $\Lambda_{\sigma_+}$ may not have monotone projections. 
\end{example}

\begin{figure}[ht]
     \centering
     \scalebox{.7}{
\begingroup%
  \makeatletter%
  \providecommand\color[2][]{%
    \errmessage{(Inkscape) Color is used for the text in Inkscape, but the package 'color.sty' is not loaded}%
    \renewcommand\color[2][]{}%
  }%
  \providecommand\transparent[1]{%
    \errmessage{(Inkscape) Transparency is used (non-zero) for the text in Inkscape, but the package 'transparent.sty' is not loaded}%
    \renewcommand\transparent[1]{}%
  }%
  \providecommand\rotatebox[2]{#2}%
  \newcommand*\fsize{\dimexpr\f@size pt\relax}%
  \newcommand*\lineheight[1]{\fontsize{\fsize}{#1\fsize}\selectfont}%
  \ifx\svgwidth\undefined%
    \setlength{\unitlength}{153.96021805bp}%
    \ifx\svgscale\undefined%
      \relax%
    \else%
      \setlength{\unitlength}{\unitlength * \real{\svgscale}}%
    \fi%
  \else%
    \setlength{\unitlength}{\svgwidth}%
  \fi%
  \global\let\svgwidth\undefined%
  \global\let\svgscale\undefined%
  \makeatother%
  \begin{picture}(1,0.98992234)%
    \lineheight{1}%
    \setlength\tabcolsep{0pt}%
    \put(0,0){\includegraphics[width=\unitlength,page=1]{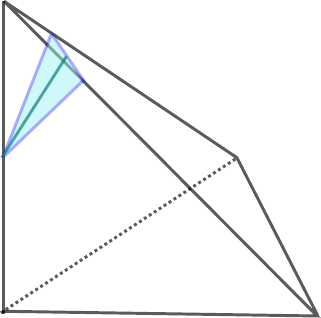}}%
    \put(0.25764036,0.79453763){\color[rgb]{0,0.50196078,0.50196078}\makebox(0,0)[lt]{\lineheight{1.25}\smash{\begin{tabular}[t]{l}$L$\end{tabular}}}}%
    \put(0.13448098,0.5628808){\color[rgb]{0,0,1}\makebox(0,0)[lt]{\lineheight{1.25}\smash{\begin{tabular}[t]{l}$Y_1$\end{tabular}}}}%
  \end{picture}%
\endgroup%
}
     \caption{The moment image of the Harvey-Lawson filling}
     \label{hlfig}
\end{figure}

For technical reasons, in the construction of moduli spaces of pseudoholomorphic curves it is convenient to have cylindrical-near-infinity Lagrangians rather than asymptotically cylindrical Lagrangians.  Presumably, the gluing and compactness results in symplectic field theory hold equally well for asymptotically cylindrical Lagrangians, but these
results have yet to appear in the literature.  The following lemma
allows us to replace asymptotically cylindrical Lagrangians
with those cylindrical near infinity.

\begin{remark} {\rm (Outline of straightening construction)}
Since the Lemma is somewhat technical, we give an outline of the construction first. We are given a Lagrangian asymptotically cylindrical cobordism that we wish to make cylindrical near infinity.  The intersection $\Lambda_{\sigma}$ of the Lagrangian  with any slice $\{ \sigma \} \times Z_\pm$ on the neck is not necessarily Legendrian.  Instead, we construct a nearby contact structure $\alpha_\sigma$ for which the slice is Legendrian.   The cylinder $L_\sigma = \R \times \Lambda_{\sigma}$ is Lagrangian for the symplectization of $Z$ with respect to $\alpha_\sigma$.   Using the equivariant coisotropic embedding theorem we deform $L$  to agree with $L_\sigma$ on an open subset, and replacing the end of $L$ with half of $L_\sigma$ gives the desired cobordism that is cylindrical near infinity.   The price of this construction is that the cylindrical end of the Lagrangian cobordism is now modelled on a different contact structure.   However, the new contact structure and Legendrian may be taken arbitrarily close to the original, and then invariance of holomorphic curves under deformation will show that the resulting Lagrangian cobordism still defines a chain map between Chekanov-Eliashberg algebras, as explained in \cite[Remark 3.15]{BCSW2}. 
\label{rep:straighten}
\end{remark}

\begin{lemma} \label{asymlem} {\rm (Straightening Lemma)}
Suppose that $X$ is a smooth symplectic manifold with ends modelled on $\R \times Z_\pm$ with symplectic form $\d(e^s \alpha_\pm)$. 
Let $L \subset X$ be an asymptotically cylindrical Lagrangian with limiting  Legendrians $\Lambda_\pm \subset Z_\pm$.  Suppose that $\Lambda_\pm$ fiber over Lagrangians $\Pi_\pm \subset Y$ with finite covering group 
the group of $k$-th roots of unity
\[ \Upsilon = \{ \zeta \in S^1 | \zeta^k = 1 \} \] 
and $L$ is $\Upsilon$-invariant on the ends.   For $\sigma_\pm > 0$ sufficiently large, let
\[ X_{[e^{\sigma_-},e^{\sigma_+}]}  := X -  (\sigma_+,\infty) \times Z_+
- (- \infty,-\sigma_-) \times Z_- \] 
denote the result of truncating the cylindrical ends at $s = \sigma_\pm$. Let 
\[ \Lambda_{\sigma_\pm} := ( \{ \pm\sigma_\pm \} \times Z_\pm) \cap L \] 
be the intersection with the $\sigma$-level set.  There exist connection one-forms $\alpha_{\sigma_\pm}$ on $Z_\pm$ for which $\Lambda_{\sigma_\pm}$ are Legendrian and so that 
\[ L_{[e^{\sigma_-},e^{\sigma_+}]} := X_{[e^{\sigma_-},e^{\sigma_+}]} \cap L \] 
can be isotoped through Lagrangians to a cylindrical near infinity Lagrangian cobordism
pair $(X,L')$ with concave end $(Z_-, \Lambda_{\sigma_-})$ and convex end $(Z_+, \Lambda_{\sigma_+})$ for some stable Hamiltonian structures $(Z_\pm,\alpha_{\sigma_\pm}, \omega_\pm)$.
Furthermore, we may choose $\alpha_{\sigma_\pm}$ arbitrarily close to $\alpha_\pm$, so that in particular $\alpha_{\sigma_\pm}$ define contact structures on $Z_\pm$, and so the submanifold $\Lambda_{\sigma_+}$ is  Legendrian for the contact manifolds $(Z_\pm,\alpha_{\sigma_\pm})$.
\end{lemma}

\begin{proof} Moser isotopy can be used to create a cylindrical end  
for which the Lagrangian is cylindrical near infinity.  It suffices to consider the  case that $X$ has only a convex end $Z_+$, as the case of a concave end is similar.  By assumption, the submanifold $L$ is $C^1$ close to $\R \times \Lambda_+$ on the end. It follows that $\Lambda_{\sigma_+}$ is a smooth codimension one submanifold of $L$.  The projection 
\[ \Pi_{\sigma_+} = p(\Lambda_{\sigma_+}) \subset \{ \sigma_+ \} 
\times Y \]
\label{rep:noZsig}
is an (in general immersed) isotropic submanifold,  since $ \omega = \d (e^s \alpha_+)$ 
vanishes on $\Lambda_{\sigma_+}$.  Because $\Lambda_+$ is an $\Upsilon$-cover and $\Lambda_{\sigma_+}$ is $C^1$-close to $\Lambda_+$, the map $\Lambda_{\sigma_+} \to \Pi_{\sigma_+}$ is an $\Upsilon$-cover.

\vskip .1in \noindent {\em Step 1:  There exist a connection one-form on 
\[ Z_{\sigma_+} := \{ \sigma_+ \} \times Z_+ \cong Z_+ \]
for which $\Lambda_{\sigma_+}$ is horizontal.}  
Indeed, by the equivalence between connection one-forms and horizontal distributions, at any point $z \in \Lambda_{\sigma_+}$ there is a unique one-form 
\[ \alpha_{\sigma_+} \in \Hom(T_z (Z_{\sigma_+} |_{\Pi_{\sigma_+}}),\R)  \] 
where $Z_{\sigma_+} |_{\Pi_{\sigma_+}}$ is the restriction of the circle bundle $Z_{\sigma_+}$ 
to the image of the Lagrangian $\Pi_{\sigma_+}$ for which $T_z \Lambda_{\sigma_+}$ is horizontal and for which
$\alpha_{\sigma_+}(2 \pi \ppth) = 1$.  By $S^1$-invariance, 
$\alpha_{\sigma_+}$ extends to a form on the orbit of $z$.   Since $\Lambda_{\sigma_+} \subset L$
is by assumption $\Upsilon$-invariant, $T_{hz} \Lambda_{\sigma_+}$ is horizontal for any $h \in \Upsilon$.  This construction defines a one-form 
\[ \alpha_{\sigma_+} \in \Omega^1(Z_{\sigma_+} |_{\Pi_{\sigma_+}}) \]
of $Z_+$ to $\Pi_{\sigma_+}$.  Using an $S^1$-equivariant tubular neighborhood,  one may extend $\alpha_{\sigma_+}$ first to an open neighborhood of $Z_{\sigma_+} |_{\Pi_{\sigma_+}}$, then to all of $Z_{\sigma_+}$ using a patching argument and convexity 
of the space of connection one-forms.  
For $\sigma_+$ is sufficiently large,  $\alpha_{\sigma_+}$ is $C^1$-close to $\alpha_+$,  and so the form $\d \alpha_{\sigma_+}$ is
$C^0$ close to $\d \alpha_+ = \pi^* \omega_{Y_+}$.

\vskip .1in \noindent {\em Step 2:  There exists a cylindrical symplectic cobordism containing a cylindrical Lagrangian cobordism containing $(Z_{\sigma_+},\Lambda_{\sigma_+})$}. The cylindrical submanifold %
\[ L_{[\sigma_0,\sigma_1]} := [\sigma_0,\sigma_1] \times \Lambda_{\sigma_+} \] 
(shown in Figure \ref{fig:straighten})  is isotropic with respect to the closed two-form 
%
\[ \omega_{[\sigma_0,\sigma_1]} := \pi_{\sigma_+}^* (\omega | \{ \sigma_+\} \times Z_+)  + \d ( ( e^s - e^{\sigma_+}) \alpha_{\sigma_+}) \in \Omega^2([\sigma_0, \sigma_1] \times Z_+)  \]
where $\pi_\sigma: \R \times Z_\pm \to X$ is the map obtained
by composing projection onto $Z_\pm$ with inclusion of $Z_\pm$ into $X$.  Furthermore, after 
shrinking the interval $[\sigma_0,\sigma_1]$, the closed form $\omega_{\sigma_+}$ is symplectic
on $[\sigma_0,\sigma_1] \times Z_+ $ of $\{ \sigma_+ \} \times Z_+$ and $L_{[\sigma_0,\sigma_1]}$ is Lagrangian.  By construction, the map $(s,z) \mapsto e^s$ is a moment map for the $S^1$-action.

\vskip .1in \noindent {\em Step 3: The cobordism  $([\sigma_0,\sigma_1] \times Z_+ ,[\sigma_0 \times \sigma_1] \times \Lambda_{\sigma_+})$ is symplectomorphic to 
$[\sigma_0,\sigma_1] \times Z_+,L_{[\sigma_0,\sigma_1]})$ after shrinking the tubular neighborhood.}   The equivariant coisotropic embedding theorem (\cite[p. 315]{gs:stp}, \cite[Remark 2.4]{sl:strat})  \label{rep:moser} implies that the forms $ \omega$ and $\omega_{[\sigma_1, \sigma_2]}$ are symplectomorphic \label{rep:equiv} by a diffeomorphism of symplectic-manifolds-with-boundary
\begin{equation} \label{eq:psi} \psi: (X_{[\sigma_0,\sigma_1])}, \omega ) \to ( X_{[\sigma_0,\sigma_1]},\omega_{[\sigma_0,\sigma_1]} ) \end{equation}
that preserves the moment maps $(s,z) \mapsto e^s$ for the $S^1$-actions.    Furthermore, $\psi$ may be taken to  equal  the identity on $\Lambda_{\sigma_+}$, since the two-forms $\omega,\omega_{[\sigma_0,\sigma_1]}$ agree  on $TZ |_{\Lambda_{\sigma_+}}$.

\vskip .1in \noindent {\em Step 4:  The image of $L$ is isotopic to $L_{[\sigma_0, \sigma_1]}$.}   The image $\psi(L)$ of $L$ under $\psi$ is a Lagrangian with respect to $\omega_{\sigma_+}$  and agrees with $L_{\sigma_+}$ at
$\Lambda_{\sigma_+}$.  It follows that $\psi(L)$ is an exact deformation of $L_{\sigma_+}$ in
an open neighborhood of $\Lambda_{\sigma_+}$, given by the Hamiltonian flow 
of $L_{\sigma_+}$  of some $S^1$-invariant function 
\[ H: X_{[\sigma_0,\sigma_1]} \to \R . \] 
A cylindrical-near-infinity Lagrangian $ \psi'(L) \subset X $
is given as Hamiltonian flow $\psi'$ of $L$ for $\rho H$ where $\rho$ is a function
whose derivative is supported near $\{ \sigma_+ \} \times \Lambda_{\sigma_+}$,
equal to $0$ in an open neighborhood of $\{ \sigma_1 \} \times \Lambda_{\sigma_+}$
and $1$ in an open neighborhood of $\{ \sigma_0 \} \times \Lambda_{\sigma_-}$.
  Gluing together 
$\psi'(L)$ with $L$ (the map $\psi$ is the identity 
outside of a small open neighborhood of $\{ \sigma_+ \} \times Z_+$) gives
the required Lagrangian $L'$ in the statement of the Lemma. 
\end{proof}

\begin{example}  We specialize to the case of the Harvey-Lawson filling.
The straightening Lemma \ref{asymlem} implies that 
after a Hamiltonian isotopy we may assume that
$L_{(1)} \cong S^1 \times \R^{n-1} \subset \C^n$ is invariant under dilation in a 
neighborhood of some sphere $e^\sigma S^{2n-1} \cap L_{(1)}$. 
\end{example}

\begin{figure}
    \centering
    \scalebox{.7}{
\begingroup%
  \makeatletter%
  \providecommand\color[2][]{%
    \errmessage{(Inkscape) Color is used for the text in Inkscape, but the package 'color.sty' is not loaded}%
    \renewcommand\color[2][]{}%
  }%
  \providecommand\transparent[1]{%
    \errmessage{(Inkscape) Transparency is used (non-zero) for the text in Inkscape, but the package 'transparent.sty' is not loaded}%
    \renewcommand\transparent[1]{}%
  }%
  \providecommand\rotatebox[2]{#2}%
  \newcommand*\fsize{\dimexpr\f@size pt\relax}%
  \newcommand*\lineheight[1]{\fontsize{\fsize}{#1\fsize}\selectfont}%
  \ifx\svgwidth\undefined%
    \setlength{\unitlength}{372.71598275bp}%
    \ifx\svgscale\undefined%
      \relax%
    \else%
      \setlength{\unitlength}{\unitlength * \real{\svgscale}}%
    \fi%
  \else%
    \setlength{\unitlength}{\svgwidth}%
  \fi%
  \global\let\svgwidth\undefined%
  \global\let\svgscale\undefined%
  \makeatother%
  \begin{picture}(1,0.56383232)%
    \lineheight{1}%
    \setlength\tabcolsep{0pt}%
    \put(0,0){\includegraphics[width=\unitlength,page=1]{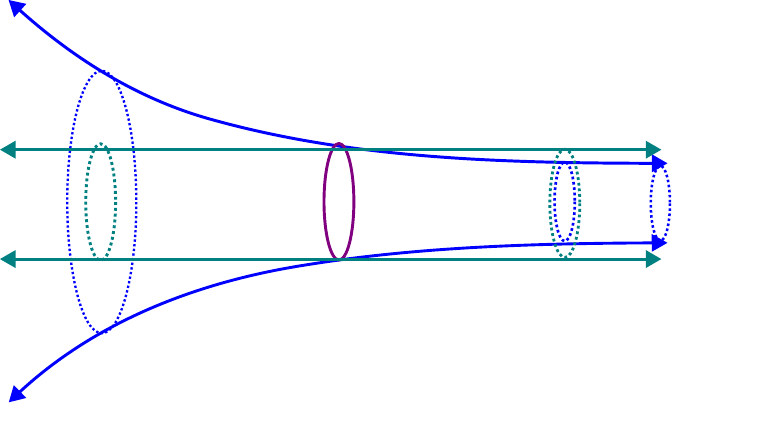}}%
    \put(0.61489264,0.00000004){\color[rgb]{0,0,1}\makebox(0,0)[lt]{\lineheight{1.25}\smash{\begin{tabular}[t]{l}L\end{tabular}}}}%
    \put(0.18017046,0.34259493){\color[rgb]{0,0.50196078,0.50196078}\makebox(0,0)[lt]{\lineheight{1.25}\smash{\begin{tabular}[t]{l}$\R \times \Lambda_\sigma$\end{tabular}}}}%
    \put(0.46036362,0.28889135){\color[rgb]{0.50196078,0,0.50196078}\makebox(0,0)[lt]{\lineheight{1.25}\smash{\begin{tabular}[t]{l}$\Lambda_\sigma$\end{tabular}}}}%
    \put(0.87357492,0.29750531){\color[rgb]{0.50196078,0,0.50196078}\makebox(0,0)[lt]{\lineheight{1.25}\smash{\begin{tabular}[t]{l}$\Lambda$\end{tabular}}}}%
    \put(0.04065155,0.45526757){\color[rgb]{0.50196078,0,0.50196078}\makebox(0,0)[lt]{\lineheight{1.25}\smash{\begin{tabular}[t]{l}$L'$\end{tabular}}}}%
    \put(0.75147436,0.3142117){\color[rgb]{0,0,1}\makebox(0,0)[lt]{\lineheight{1.25}\smash{\begin{tabular}[t]{l}$L$\end{tabular}}}}%
    \put(0,0){\includegraphics[width=\unitlength,page=2]{straighten.pdf}}%
  \end{picture}%
\endgroup%
}
    \caption{Straightening an asymptotically cylindrical Lagrangian}
    \label{fig:straighten}
\end{figure}

\begin{remark} We remark on the possibility of obtaining a cobordism between contact manifolds, rather than stable Hamiltonian manifolds.   If $\sigma_+, \sigma_-$ are sufficiently large then 
the forms $\alpha_{\sigma_\pm}$ are contact forms.  However, we may not be able
to choose a cobordism between $(Z_-,\alpha_{\sigma_-})$ and $(Z_+,\alpha_{\sigma_+})$ in the sense that the forms $\d \alpha_{\sigma_\pm}$ extend as symplectic forms over the cobordism and $L$ is Lagrangian with respect to the extension.  Indeed, the action of loops in the Legendrian $\Lambda$ as $\d \alpha_{\sigma_\pm}$ may not equal those of $ \omega_{\sigma_\pm}$.
\end{remark}

\section{Pseudoholomorphic buildings}
\label{tocob}

In this section,  we construct the moduli spaces of holomorphic buildings 
used for both the differentials in the Chekanov-Eliasberg algebra and the 
chain maps used associated to Lagrangian cobordisms.   Because our stable Hamiltonian manifolds
are circle-fibered over symplectic manifolds and we consider only holomorphic disks
rather than curves of higher genus, we may use Cieliebak-Mohnke \cite{cm:trans} perturbations to regularize. 

\subsection{Punctured surfaces} 

We recall basic terminology for surfaces with strip-like ends
at punctures.

\begin{definition} \label{def:surfwstrip}
A {\em  surface
with strip-like and cylindrical ends}  $S$ is obtained from a closed
oriented surface-with-boundary $\ol{S}$ by removing a finite collection of
boundary points $z_{e,\white}, e = 1,\ldots, e(\white)$ 
and interior points 
$z_{e,\black}, e = 1,\ldots, e(\black) $:
\[ S = \ol{S} - \{ z_{e,\white}, e = 1,\ldots, e(\white), \quad 
z_{e,\black}, e = 1,\ldots, e(\black) . \} \]
We call $S$ a {\em  punctured surface} for short, and the removed
points {\em  punctures}.     

Suppose $S$ is equipped with a conformal structure, giving rise to an
almost complex structure 
\[ j: TS \to TS, \quad j^2 = -\on{Id}_{TS} . \] 
An open neighborhood $U_{e,\white} \subset S$
of each puncture $z_{e,\white}$ is assumed to be equipped with a local coordinate, and
similarly for open neighborhoods $U_{e,\black}$ of $z_{e,\black}$.  Each such 
coordinate gives 
 holomorphic embeddings called {\em  strip-like} resp. {\em  cylindrical ends}
 \[ \kappa_{e,\white}: U_{e,\white} - \{ z_{e,\white} \} \cong \pm \R_{> 0} \times [0,1 ],
 \quad \kappa_{e,\black}: U_{e,\black} - \{ z_{e, \black} \} \cong \pm \R_{> 0} \times S^1
 . \]
The set of ends is denoted $\mE(S)$ and is equipped with a partition into incoming resp. outgoing ends  
\[  \mE(S) = \mE_+(S) \cup \mE_-(S)  \]   
based on whether the embedding is into $+\R_{>0}$ or $-\R_{>0}$. In the later definition of punctured holomorphic maps, the incoming resp. outgoing ends will be mapped to the convex resp. concave end of a symplectic cobordism, asymptotic to Reeb chords or Reeb orbits. The Reeb chords at incoming resp. outgoing punctures will be used as input resp. output for the differential and cobordism maps of Legendrian contact homology \cite{BCSW2}.
\end{definition}

In the Gromov compactification of the moduli space of holomorphic maps, nodal surfaces will appear as degenerations.

\begin{definition} \label{def:nodsurf}
A {\em  nodal surface } $S$ is obtained from a surface with boundary
$\ti{S}$ with strip-like and cylindrical ends by gluing along disjoint
pairs of points $w_\pm(e) \in \ti{S}$, where $e$ ranges over some
index set, which may be either a pair of boundary points
$w_\pm(e) \in \partial \ti{S}$ or a pair of interior points
$w_\pm(e) \in \on{int}(\ti{S})$. \label{rep:orapair}

\vskip .1in \noindent
A nodal surface $S$ is a {\em  nodal
  disk} if each component $S_v$ is a disk, the boundary $\partial S$
is connected, and the combinatorial type of the curve (defined below)  is a tree.   \label{rep:tree}

\vskip .1in \noindent  A {\em 
  marking} is a pair $\ul{z} = (\ul{z}_\black, \ul{z}_\white)$
consisting of a tuple of points $\ul{z}_\white$ on the boundary
$\partial S$ disjoint from the nodes and a tuple of points
$\ul{z}_\black \in \on{int}(S)$ in the interior, with the property
that the ordering of points on the boundary
$\partial S_v \cap \ul{z}_\white$ is cyclic with respect to the induced (counter-clockwise) orientation of the boundary.

\vskip .1in \noindent
A marked  nodal surface $(S,\ul{z})$ is {\em  stable} if the group of automorphisms
%
$\Aut(S) $
%
is finite.   \label{rep:stable} Equivalently, for each sphere component there are at least three special (nodal or marked points), while for each disk component there are at least three special boundary 
points or at least one special boundary point and one special interior point.

\vskip .1in \noindent
A marked nodal surface $(S,\ul{z})$ has a combinatorial type
$\Gamma = \Gamma(S)$, which is a graph 
\label{rep:whichis} 
whose vertices $v \in \Ver(\Gamma)$ correspond to
components $S_v$ and edges $e \in \Edge(\Gamma)$ correspond to nodes
$w_\pm(e) \in S$ or markings $z(e) \in S$; we define the subset
$\Edge_\rightarrow(\Gamma)$ to be the set of edges corresponding to
markings and call them {\em  leaves}.   

\vskip .1in \noindent 
A \textit{matching isomorphism} at a node $w_\pm(e)$ where the edge $e$ is between vertices $v_\pm$ is the data of an $S^1$-equivariant diffeomorphism between the circles 
\begin{equation} \label{eq:me}
    m_e : T_{w_+(e)}S_{v_+}\setminus{\{0\}} /\R_+ \to T_{w_-(e)}S_{v_-}\setminus{\{0\}} /\R_- .
\end{equation}
\vskip .1in \noindent 
An \textit{asymptotic marker} at an interior marked point $z(e)$
corresponding to an interior edge $e \in \Edge_\black(\Gamma)$ is an element  
\[ m_e \in T_{z_\black}S\setminus{\{0 \}}/\R . \]
\vskip .1in \noindent 
For nodal disks $(S,\ul{z})$, the combinatorial type $\Gamma$ is required to be a tree.  By a {\em decorated nodal disk} we mean a nodal disk with matching isomorphisms and asymptotic markers.
\label{rep:asympmarker}
\end{definition}

\subsection{Pseudoholomorphic buildings} 

We define pseudoholomorphic buildings in symplectic cobordisms with ends modelled on
cylinders over stable Hamiltonian manifolds. 

\begin{definition} \label{def:cyl}
Let $(Z,\alpha,\omega) \to Y$ be a circle-fibered stable Hamiltonian manifold as above and $X = \R \times Z$ be the trivial cobordism.    An almost complex structure $ J : TX \to TX$ is
  {\em  cylindrical} if there exists an almost complex structure
  $\ol{J} : TY \to TY $ so that the projection
  $p_X: X \to Y$ is almost complex and $J$ is
  the standard almost complex structure on any fiber.  More precisely, let 
\[ \partial_s \in \Vect(\R \times Z), \quad \partial_{\theta} \in \Vect(\R \times Z) \]
denote the translational vector field on $\R$ resp. rotational vector field on $Z$. 
The almost complex structure $J$ is determined
on the vertical part of $TX$ by 
\[ \quad J  \partial_s = \partial_\theta, \quad J \partial_\theta = - \partial_s  \  \text{in}  \ \Vect(\R \times Z) .\] 
On the other hand, the projection to $Y$ is required to be almost complex:
\[   D {p}_X J = \ol{J} D {p}_X \  \text{in}  \ \Map(TX,TY) . \]
Suppose now $X$ is an arbitrary cobordism with concave end $Z_-$ and convex end $Z_+$.
An almost complex structure $J$ on $X$ is called {\em  cylindrical}
if it is the restriction of cylindrical almost complex structures
on $\R \times Z_\pm$ on the cylindrical ends $\pm (0,\infty) \times Z_\pm \to X$.    The space of cylindrical almost complex structures is denoted $\J_{\cyl}(X)$.
This ends the Definition.
\end{definition}

We recall some notions of energy for maps from a punctured surface
$S$.

\begin{definition} \label{henergy}
\begin{enumerate} 
\item {\rm (Horizontal energy)} The {\em  horizontal energy} of a holomorphic map $ u = (\phi,v): (S,j) \to (\R \times Z,J)$ bounding
  $\R \times \Lambda$ is (\cite[5.3]{sft})
\[ E^{\bh}(u) = \int_S v^* \omega .\]
\item {\rm (Vertical energy)} The {\em  vertical energy} of a
  holomorphic map $ u = (\phi,v): (S,j) \to (\R \times Z,J)$ is
  (\cite[5.3]{sft})
\begin{equation} \label{alphaen}
E^{\bv}(u) = \sup_{\zeta} \int_S (\zeta \circ \phi) \d \phi \wedge v^*
\alpha \end{equation}
where the supremum is taken over the set of all non-negative
$C^\infty$ functions 
\[\zeta: \R \to \R, \quad  \int_\R \zeta(s) \d s = 1 \]
with compact support.
\item {\rm (Hofer energy)} The {\em  Hofer energy} of a holomorphic map
\[ u = (\phi,v): (S,j) \to (\R \times Z,J) \] 
is (\cite[5.3]{sft})
  is the sum
\[ E(u) = E^{\bh}(u) + E^{\bv}(u) .\]
\item {\rm (Generalization to manifolds with cylindrical ends)}
  Suppose that $X$ is a  manifold with cylindrical ends
  modelled on $\pm(0,\infty) \times Z_\pm$.  \label{Zfix} \label{Zfixp} The
  vertical energy $E^{\bv}(u)$ is defined as before in \eqref{alphaen}.
  The Hofer energy $E(u)$ of a map $u: S \to X$ from a
  surface $S$ with cylindrical ends to $X$ is defined by
  dividing $X$ into a compact piece $X^{\on{com}}$ and 
  cylindrical ends \label{Rfix} \label{Rfixp} diffeomorphic to
  $\pm(0,\infty) \times Z_\pm$.  Then we set
\[ E(u) = E(u | X^{\on{com}} ) + 
E(u | (0,\infty) \times Z_+) + 
E(u | (-\infty,0) \times Z_-)  
.\]
\end{enumerate} 
\end{definition}

In symplectic field theory, the compactification of punctured
holomorphic curves is given by {\em  holomorphic buildings} in which bubbling
occurs on the cylindrical ends.  
Let $(S, \partial S, j)$ be a bordered Riemann surface, with punctures in the interior and on the boundary.  Let $(X,J)$ be an almost complex manifold.    A smooth map $u : S \to X $ is {\em 
  pseudoholomorphic} 
  (or {\em holomorphic} for short)
  if the antiholomorphic part of the derivative
vanishes:
\[ \olp_J u := \frac{1}{2} ( \d u + J \d u j ) = 0  .\]

\begin{definition}[Compactifying symplectizations]\label{disc:compactifying targets}
    Let $Z$ be a compact circle-fibered contact manifold.  Let $\ol{\R \times Z}$ be the almost complex manifold defined by adding two divisors $Y_{\infty}, Y_0$ to compactify $\R \times Z$.  Explicitly, $\ol{\R \times Z}$ is the associated fiber bundle as in Example \ref{ex:symp2} 
    \[ \ol{\R \times Z} = (Z \times \CP^1)/S^1 \],
    where the action of $S^1$ is diagonal. This has the effect of replacing each $S^1$-fiber in $Z$ with a 
    $\CP^1$.   The bundle comes equipped with a zero divisor and an infinity divisor 
    \[ Y_0 = (Z \times \{ 0 \})/S^1, \quad Y_\infty = (Z \times \{ \infty \})/S^1 .\]
    The almost complex structure on $\ol{\R \times Z}$ is determined uniquely by that of $Y$ in the horizontal directions, using the connection,  and that of $\C P^1$ in the vertical directions. We can define a symplectic form $\omega_{\sigma_1,\sigma_2}$ on $\ol{\R \times Z}$ by identifying $\ol {\R \times Z}$ with the  symplectic cut as in Lemma \ref{lem:trunc}.

    We can similarly compactify a cobordism $X$ by adding divisors $Y_{+,\infty}$ at the convex end and $Y_{-,0}$ at the concave end, and define symplectic forms on the compactification. We denote the symplectic form on the compactification $\ol{X}$ with $\ol{\omega}_{\sigma_+,\sigma_-}$ induced from symplectic cut at $\{\sigma_\pm\} \times Z_\pm$. We will suppress the subscript at times and use the notation $\ol{\omega}$ for the symplectic form on $\ol{X}.$      This ends the Definition.
   
\end{definition}

\begin{definition} {\rm (Holomorphic Buildings)}
Let $X$ is be a cobordism with concave end $Z_-$ and convex end $Z_+$, and
$J$ is an almost complex structure on $X$ that is cylindrical on the ends.
For integers $k_-, k_+ \ge 0$, define topological spaces
$ \XX[k_-,k_+]$ by adding on $k_-$ resp. $k_+$ outgoing resp. incoming neck pieces: 
\begin{multline} 
\XX[k_-,k_+] :=(\ol{\R \times Z_-}) \cup_{Y_-}  \ldots 
\cup_{Y_-} (\ol{\R \times Z_-}) \cup_{Y_-} 
\ol{X} \\ \cup_{Y_+} (\ol{\R \times Z_+} )
\cup_{Y_+} \ldots \cup_{Y_+} (\ol{\R \times Z_+} ) .\end{multline}
Denote by $\ol{\XX}[k_-,k_+]_k$ the $k$-th topological space
in the union above, 
and $\XX[k_-,k_+]_k$ the complement of the copies of $Y_\pm$.
Thus 
\[  \XX[k_-,k_+]_k = \begin{cases} {\R \times Z_-} & k< 0 \\ 
X & k = 0 \\ 
{\R \times Z_+} & k> 0 \end{cases} .\]   
Let $L$ be a Lagrangian cobordism with concave end $\Lambda_-$ and convex end $\Lambda_+$.
Let
\begin{multline} 
\LL[k_-,k_+] :=\ol{\R \times \Lambda_-} \cup_{\Pi_-}  \ldots 
\cup_{\Pi_-} \ol{\R \times \Lambda_-} \cup_{\Pi_-} 
\ol{L} \\ \cup_{\Pi_+} \ol{\R \times \Lambda_+} 
\cup_{\Pi_+} \ldots \cup_{\Pi_+} \ol{\R \times \Lambda_+} \end{multline}
be the union of  the closure $\ol{L}$ with 
$k_-,k_+$ cylindrical pieces attached to the ends $\Pi_\pm$. 

\vskip .1in \noindent 
A {\em  holomorphic $(k_-,k_+)$-building} in $X$ bounding $L$ is a continuous map of a nodal curve-with-boundary $S$ 
to 
$\XX[k_-,k_+]$ with the following properties:  
For  each irreducible component $S_v$, let $S_v^\circ$ denote the pre-image of $\XX[k_-,k_+]_{k(v)}$; we call the elements
of $S_v - S_v^\circ$ the {\em punctures}.
The restriction $u_v$ of $u$
is a holomorphic map to some component $\XX[k_-,k_+]_{k(v)}$:
\[ u_v: S_v \to \XX[k_-,k_+]_{k(v)}, \ v \in \Ver(\Gamma)  \] 
with finite Hofer energy. 

Since our contact form is of Morse-Bott type, the finiteness of Hofer energy guarantees the punctures of the curve are asymptotic to either Reeb orbits or Reeb chords by the results of \cite{sft}. For each puncture $w \in \ol{S}_v$, let $U_w$ be an open neighborhood of $w$ and $z$ a coordinate so that
the puncture is at $z = 0$ and
\begin{enumerate}
\item if $w$ is an interior puncture, the asymptotic marker at $w$ is the equivalence class of the vector $(1,0)$ in the local coordinate;
\item if $w$ is a boundary puncture, then $U_w$ is an upper half-isk in the local coordinate, with boundary on the real line. 
\end{enumerate}
Then, in the given local coordinates
\[ \gamma(\theta) := \lim_{r \to 0} \pi_Z (u(re^{i\theta})) \] 
is a Reeb orbit or chord.

We require that the incoming punctures are asymptotic to the convex end and the outgoing punctures are asymptotic to the concave end.  For each node $w$ of $S$ connecting components $S_{v_-}$ and $S_{v_+}$ and mapping to some copy of $Y$, the limiting Reeb chord or orbit $\gamma_-$ of $u_{v_-}$ at $w$ is equal to limiting Reeb chord or orbit $\gamma_+$ of $u_{v_+}$ at $w$.   We write $u: S \to \XX$ for short. 

\vskip .1in \noindent 
An {\em  isomorphism} between buildings $u_0,u_1: S_0,S_1 \to \XX[k_-,k_+]$ is an isomorphism of nodal 
curves $\phi: S_0 \to S_1$, 
preserving the matching isomorphisms and asymptotic markers, and a combination of translations $\psi_k: \XX[k_-,k_+]_k \to \XX[k_-,k_+]_k $ for $k\neq 0$ on the cylindrical pieces
so that 
\[ u_1 \circ \phi := \psi_k \circ u_0 \] 
where each $\psi_k$ is a diffeomorphism of the form
\[ \psi_k: \R \times Z_\pm \to \R \times Z_\pm, \quad (s,z) \mapsto (s+a_k,z)  \] 
for some constant $a_k$.
\label{rep:isobuilding}

\label{rep:trivcyl} \vskip .1in \noindent A {\em  trivial cylinder or strip} is a map $u: S \to \R \times Z$ mapping to a fiber $Z_y \cong \C^\times$ of the projection to $Y$, with domain $S$ either an $\R \times S^1$
or $\R \times [0,1]$, and mapping to $Z_y $ by an exponential map $(s,t) \mapsto e^{\mu(s + it + c)}$ for some exponential growth constant $\mu$ and constant $c \in \C$.  Any trivial cylinder $u$ is contained in a fiber of $p_Y$ and has exactly 
two punctures.  Indeed, the  action sums at the incoming and outgoing punctures are equal as in Lemma \ref{anglechange} below.

\vskip .1in \noindent  A
holomorphic building $u$ is {\em  stable}  if $u$ has finitely 
many automorphisms, or equivalently the map $u_0$ to 
the component $\XX[k_-,k_+]_0 \cong \ol{X}$ is stable
in the usual sense, and  the map $u_j, j \neq 0$ to each component $\XX[k_-,k_+]_j = \ol{\R \times Z_\pm}$ has at least one connected component $u_v$ that is not a trivial cylinder or strip.
\end{definition} 

In the symplectic field theory literature, it is more common to use punctured holomorphic curves.   Since our Reeb flow induces a circle action on the contact manifold, we can collapse the Reeb trajectories to get closed holomorphic curves with nodes. The following gives a precise correspondence between maps with punctures and maps with nodes. The advantage for us to use the compactified ambient space is that we can use counting formulas from the compactified setting to study rigid buildings. Moreover, we use a transversality scheme similar to Cieliebak-Mohnke \cite{cm:trans} to regularize the moduli spaces which show up in symplectizations and symplectic cobordisms.

\begin{discussion}[Compactifying Holomorphic maps with punctures]\label{discussion:compactifying map}
    Throughout this article and in the sequel \cite{BCSW2,BCSW3}, we often use removal of singularities to extend holomorphic maps with punctures to maps from closed surfaces.  
For holomorphic buildings in manifolds with cylindrical ends modeled on $\R \times Z$, where $Z$ is a circle-fibered contact manifold, as in Definition \ref{disc:compactifying targets} we may compactify $\R \times Z$ to $\ol{\R \times Z}$.  The compactification $\ol{\R \times Z}$ has a symplectic form 
given by viewing it as a symplectic cut of $\R \times Z$ which makes the almost complex structure on $\ol{\R \times Z}$ compatible.  
    
    Any finite-energy holomorphic map may be extended over the punctures.  For an incoming puncture $p$, assume that under a choice of cylindrical coordinates $(s,t)$ of the domain near $p$, a map $u(s,t)$ is asymptotic to a trivial strip or a cylinder $\gamma$ as $s \to \infty$. Note that $\gamma$ descends to a point $q\in Y$.  The map $u$ extends over $p$ by compactifying the target to $\ol{\R \times Z}$ and defining $u(p) = q \in Y_\infty$. Following \cite[Lemma 3.10]{chanda} one can check that if $u$ is a map in $\R \times Z$ of finite Hofer energy, the compactification of $u$ has finite $\omega_{\sigma_1,\sigma_2}$-area,
    where $\omega_{\sigma_1,\sigma_2}$ is the form from \eqref{disc:compactifying targets}
\end{discussion}

\subsection{Treed holomorphic buildings}
\label{sec:treedbuildings}

Treed disks, also known as cluster configurations, were introduced
in the context of Lagrangian Floer homology by Cornea-Lalonde \cite{cornea:cluster}.
The domains for treed disks are combinations of disks with
strip-like or cylindrical ends and line segments.   

\begin{definition} \label{def:brokeninterval}
A {\em  closed interval} is a closed connected subset of 
$ \{ - \infty \} \cup \R \cup \{ \infty \}$.  
Let $I_1,\ldots, I_k$ be a sequence of closed intervals so that ${I}_1,\ldots, I_{k-1}$ contain $\infty$,
and ${I}_2,\ldots, {I}_k$ contain $- \infty$.
A {\em  broken interval} is a topological space $I$ obtained from the union of 
closed intervals ${I}_1,\ldots, {I}_k$ by identifying 
$\infty$ in one interval with $-\infty$ in the next:
\[ I = I_1 \cup_{\infty \sim - \infty} I_2 \cup_{\infty \sim - \infty} \ldots \cup_{\infty \sim - \infty} I_k .  \]
The images of the identified points $\pm \infty$ in each subinterval are called {\em  breakings}, while the 
boundary points of $I$ are {\em  ends} and may be either
infinite or non-infinite.    The standard metric of the copy of real line $\R$ in each $I_j$  induces a metric on the complement in $I$ of the breakings. 

\vskip .1in \noindent
A {\em  treed disk} $C$ is obtained from a nodal disk $S$ by replacing
each node $w_\pm(e) \in  S$ or marking $w(e) \in \partial S$
with a broken interval $T_e$; in the case of 
a marking the corresponding broken interval $T_e$
is required to have one infinite end and be joined to $S$ at
the finite end; while in the case of a node both ends of the broken interval are required to be finite (although the length of the broken interval may be infinite).
\label{rep:brokenintervals}

\vskip .1in \noindent 
The {\em  combinatorial type} $\Gamma$ of any treed disk is the graph 
\[ (\Ver(\Gamma) = \Ver_\white(\Gamma) \cup \Ver_\black(\Gamma), \quad \Edge(\Gamma) = \Edge_\white(\Gamma) \cup \Edge_\black(\Gamma))
\]
whose vertices correspond to disk and sphere components, respectively, and whose
edges correspond
to boundary resp. interior edges.  The set of semi-infinite edges is denoted
\[ \Edge_{\rightarrow}(\Gamma) \subset \Edge(\Gamma) .\]
The type $\Gamma$ has the  additional data of subsets 
\[ \Edge_0(\Gamma), \Edge_\infty(\Gamma) \subset \Edge(\Gamma) -
\Edge_{\rightarrow}(\Gamma) \]
describing which edges have zero or infinite length.  \label{rep:zeroor}  That is, $\Edge_\infty(\Gamma)$ consists
of those edges that are broken intervals, while $\Edge_0(\Gamma)$ consists of edges of zero length 
that are collapsed under the equivalence relation $\sim$ in \eqref{edgecollapse}. See \cite[Section 4.2]{flips} for more discussion of the combinatorial type.  Any treed disk 
\[ C = \bigcup_{v  \in \Ver(\Gamma)} S_v \cup 
\bigcup_{e \in \Edge(\Gamma)} T_e \] 
is the union of disks and spheres $C_v$ 
for  the vertices $v \in \Ver(\Gamma)$ and interior and boundary edges 
$T_e$ for the edges $e \in \Edge(\Gamma)$.

A treed disk $C$ is {\em  stable} if the surface components $S_v$ with markings given by $S_v \cap T$ is the union of stable disks 
and spheres  and  each broken interval $T_e$ has at most one breaking.  This ends the Definition.
\end{definition}

 \begin{figure}[ht]
     \centering
     \scalebox{.7}{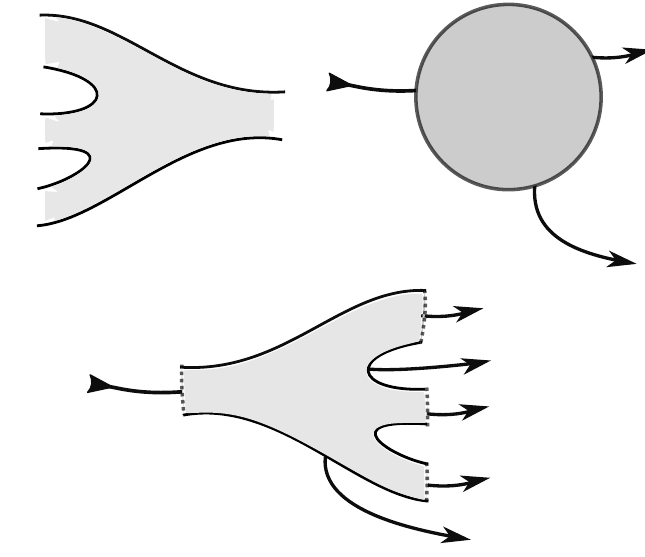}
     \caption{A punctured disk and two treed punctured disks}
     \label{treevs}
 \end{figure}

The moduli space of stable treed disks with a fixed number of punctures has
the structure of a compact Hausdorff space as in Charest-Woodward \cite{flips}.
Let $\M_\Gamma$ denote the moduli space of treed bordered surfaces of
type $\Gamma$.   If $\Gamma $ is a type with an edge $e \in \Edge_0(\Gamma)$, we declare $\Gamma$  to be equivalent to the type  $\Gamma'$ obtained by removing the edge $e$ and identifying the adjacent vertices $v_\pm \in \Ver(\Gamma)$.   Denote by 
\begin{equation} \label{edgecollapse} \ol{\M} = \sqcup_\Gamma \M_\Gamma / \sim \end{equation}
the union over 
stable
types $\Gamma$, possibly disconnected, modulo the edge-collapse equivalence relation.   The space 
$\ol{\M}$ has a natural Hausdorff topology, induced by convergence of the surface components in the moduli space of marked disks and convergence of the edge lengths, as in Cornea-Lalonde \cite{cornea:cluster}.  Denote by
$\M \subset \ol{\M}$ the subset of strata of top dimension, 
consisting of configurations $C$ so that 
each edge $T_e$ has non-zero length $\ell(e)$ in $(0,\infty)$.
 The set of types has a natural partial order with $\Gamma'
 \preceq \Gamma$ if and only if $\M_{\Gamma'}$ is contained in the
 closure of $\M_\Gamma$. 

We will also consider domains with components labelled by integers, used to describe holomorphic buildings in the discussion below. 

\begin{definition} 
A genus zero {\em  building} is a nodal disk $S$ equipped with a
partition 
\begin{equation} \label{treedbuilding} 
 S = S_1 \cup \ldots \cup S_k \end{equation} 
into (possibly disconnected) components $S_1,\ldots, S_k \subset S$ called {\em  levels} so that  intersection
between two levels $S_i \cap S_j$ is empty unless $i = j \pm 1$. 
Removing
the intersections $S_i \cap S_{i \pm 1}$ gives a surface with
cylindrical and strip-like ends

\vskip .1in \noindent   A  
 {\em  treed building} is a treed disk $C$ equipped with a decomposition 
 \[ C = C_1 \cup \ldots \cup C_k \] 
 so that $C_i,C_j$ intersect only if $j \in \{ i-1,i,i+1 \}$.    See Figure \ref{tbuilding}.

\begin{figure}
    \centering
\scalebox{.6}{  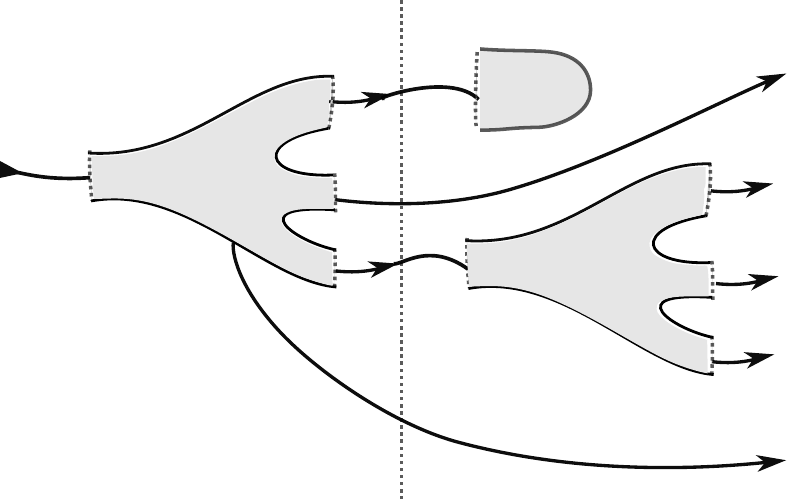}
    \caption{Treed building}
    \label{tbuilding}
\end{figure}

\vskip .1in \noindent 
The {\em  combinatorial type} of a
treed building is a graph 
\[ \Gamma = (\Ver(\Gamma), \Edge(\Gamma)) \]
equipped with a decomposition into
not-necessarily-connected subgraphs
\[ \Gamma = \Gamma_1 \cup \ldots \cup \Gamma_k \] 
so that $\Gamma_i \cap \Gamma_j$ is empty unless $j \in \{i -1, i , i+1 \} $.
 \end{definition}

Treed holomorphic buildings 
 in symplectic cobordisms are obtained by adding trajectories of the following kind.
 Let $Z$ be a fibered contact or stable Hamiltonian manifold, and $\Lambda \subset Z$ a Legendrian.  We assume that a metric is fixed on $\Lambda$,
 hence on $\cR(\Lambda)$ using the identification of components of $\cR(\Lambda)$ with $\Lambda$ in 
 Lemma \ref{startend}.
 
\label{treeddisk}

\begin{definition} \label{morsedatum} 
A {\em  Morse datum} for $(Z,\Lambda)$ consists of a pair  of vector fields on the space of Reeb chords and on the Legendrian 
\[ \zeta_\white \in \Vect({\cR}(\Lambda)),  \quad  \zeta_{\black,\Lambda} \in \Vect(\R \times \Lambda)^{\R} \] 
arising as follows.
\begin{enumerate} 
\item  There exists a  Morse function on the space of Reeb chords
\[  f_\diam:  {\cR}(\Lambda) \to \R ;\] 
so that $\zeta_\diam$ is the gradient vector field:
\begin{equation} \label{zdiam} 
\zeta_\white := \grad(f_\diam) \in \Vect({\cR}(\Lambda)) .\end{equation}
\item There exists a Morse function 
\[ f_\black :   \Lambda \to \R ;\] 
with gradient vector field 
\[ \grad(f_\black) \in \Vect(\Lambda) \] 
so that $\zeta_{\black,\Lambda}$  is a translation-invariant lift of $\grad(f_\black)$\label{zwhite}.
\end{enumerate}
\end{definition}

\begin{definition}  A vector field $\zeta_{\black,\Lambda} \in \Vect(\R \times \Lambda)$ is {\em  positive} if in coordinates
$(s,\lambda)$ on $\R \times \Lambda$ there exists a function
\[ a: \Lambda \to \R_{> 0} \]
so that 
\begin{equation} \label{zetablack} \zeta_{\black,\Lambda} = a(\lambda) \partial_s + {p}^* \grad(f_\black) 
\end{equation}
where 
\[ {p}^* : \Vect(\Lambda) \to \Vect(\R \times \Lambda)^\R  \] 
is the obvious identification of
translationally-invariant vector fields trivial in the $\R$-direction
with 
vector fields on $\Lambda$.
\end{definition}

The limit of any Morse trajectory on a compact manifold \label{rep:compact} along any infinite length trajectory is a zero of the gradient vector field.   We introduce labels for the possible limits of the trajectories above as follows. 
Denote by \label{rep:ulR}
\[ \ul{\R} =  (T\R) \times \Lambda \] 
the translational factor in
$T(\R \times \Lambda) = T\R \oplus T \Lambda$.  The  zeroes of the vector field
$p_*(\zeta_{\black,\Lambda})$ correspond to tangencies of $\zeta_{\black,\Lambda}$ with the translational factor:
\[  {p}_*(\zeta_{\black,\Lambda})^{-1}(0) = \zeta_{\black,\Lambda}^{-1}(\ul{\R}) \subset
  \R \times \Lambda . \] 
Let 
\begin{equation} \label{gens} \cI(\Lambda) := 
\cI_\white(\Lambda) \cup \cI_\black(\Lambda), 
\quad  \cI_\white(\Lambda) := \zeta_\white^{-1}(0), \quad \cI_\black(\Lambda) := \grad(f_\black)^{-1}(0)
\end{equation}
be the set of zeroes of these vector fields; these will be the
generators of our Chekanov-Eliashberg algebras.    The inclusion 
of the generators $\cI_\black(\Lambda)$ is similar to the inclusion of the chains on the Lagrangian in the definition of immersed Lagrangian Floer theory, while the generators 
$\cI_\white(\Lambda)$ correspond to the self-intersections.  

Treed holomorphic disks in cobordisms are combinations of holomorphic maps and trajectories of the vector fields above.
Let $(X,L)$ be a cobordism pair with concave end $(Z_-,\Lambda_-)$ and convex end $(Z_+,\Lambda_+)$. 
Suppose $\R \times Z_\pm$ are equipped with cylindrical almost complex
structures $J_\pm$ and $\R \times \Lambda_\pm$ are equipped with vector fields $\zeta_\pm \in \Vect(\R \times \Lambda_\pm)$ as above.  

\begin{definition} \label{def:cylnearinf} Let $X$ be a symplectic
cobordism with concave end $Z_-$ and convex end $Z_+$.
An almost complex structure
\[ J : TX \to TX \]
is {\em   cylindrical-near-infinity} if $J$ restricts
to cylindrical almost complex structure $J_\pm$ on the ends 
\[ \kappa_\pm(\pm (0,\infty) \times Z_\pm) \subset
X.\]

\vskip .1in \noindent
A function resp. gradient vector field
\[ f_L: L \to \R, \quad \zeta_{L} = \grad(f_L) \in \Vect(L) \] 
is {\em  cylindrical near infinity}   if the restriction of the vector field $\zeta_L$ to the cylindrical ends is given by 
\[ \zeta_{L} |_{\kappa_\pm(\pm (0,\infty) \times Z_\pm)} = \zeta_{\black,\Lambda_\pm} \] 
for some vector fields $\zeta_{\black,\pm}$
of the form \eqref{zetablack}.

\vskip .1in \noindent A {\em  decorated cobordism pair}
is a cobordism pair $(X,L)$ together with a pair $ (J,f_L)$ 
as above.   This ends the Definition.
\end{definition} 

In order to use the Cieliebak-Mohnke perturbation, we need pick a Donaldson type hypersurface $D$ in $X-L$. This choice will be discussed in detail in Section \ref{sec:don}. Here we only use it to define combinatorial types of treed holomorphic disks. For notational convenience,  we set \label{rep:zetad}
\[ \zeta_D = 0 \in \Vect(D) \]

\begin{definition} \label{def:treedbuildings} Let $(X,L)$ be a decorated cobordism pair with concave end 
$(Z_-,\Lambda_-)$ and convex end $(Z_+,\Lambda_+)$ as above, so that the Lagrangians
$\Pi_\pm = p_\pm(\Lambda_\pm)$ are embedded. 

\vskip .1in \noindent  
A {\em  treed holomorphic disk} from a treed disk $C$ to $X$ is a
  collection of holomorphic maps and trajectories
\[  \begin{array}{ll}
 u_v: & S_v \to X \quad   v \in \Ver(\Gamma) \\
 u_{e}: & T_e \to \begin{cases} L &  e \in \Edge_L(\Gamma) \\ 
\cR(\Lambda_\pm) &  e \in \Edge_{\white,\pm}(\Gamma) \\ 
 D &  e \in \Edge_D(\Gamma)  \end{cases}
\end{array} \]
for the vector fields $\zeta_{L}, \zeta_{\white,\pm},\zeta_D$
for some partition of the edges $\Edge(\Gamma)$
into subsets 
\[ \Edge_D(\Gamma), \Edge_L(\Gamma),\Edge_{\white,\pm}(\Gamma) \subset \Edge(\Gamma)  \] 
so that the 
boundary conditions 
\[ u_v (\partial S_v) \subset L, \quad \forall v \in \Ver_\white(\Gamma) \] 
are satisfied and matching conditions hold at the end of each edge according to
\[ \lim_{r \to 0}(\pi_{Z_\pm} u_v(r e^{i \theta})) = u_{e}(z), \quad \forall z \in T_e \cap \ol{S}_v ,  \quad \forall e \in
\Edge(\Gamma), \ v \in \Ver(\Gamma) \]  
for punctures mapping to Reeb chords, and 
\[ \lim_{r \to 0}(\pi_{Z_\pm} u_{v_-}(r e^{i \theta})) = \lim_{r \to 0}(\pi_{Z_\pm} u_{v_+}(r e^{i \theta})) \] 
for components $u_{v_-}, u_{v_+}$ separated by a puncture corresponding to a Reeb orbit, 
where the local charts near the puncture are chosen compatibly with the matching isomorphism \eqref{eq:me}.

We further require, for our regularization scheme, that each component of $u^{-1}(D)$
is the image of an edge $e \in \Edge_D(\Gamma)$.  We denote by $T_D$ resp. $T_L$ resp. $T_{\white,\pm} \subset T$ the union of the  edges $T_e$ with $e \in \Edge_D(\Gamma), \Edge_L(\Gamma),\Edge_{\white,\pm}(\Gamma) $.

\vskip .1in \noindent  
A {\em  treed holomorphic building} in
$\XX[k_-,k_+]$ with boundary on $\LL$ is a treed disk $C$ equipped with a decomposition 
\[ C = C_{k_-} \cup \ldots \cup C_{k_+} \]
together with treed disks bounding $L$ or $\R \times \Lambda_\pm$
\begin{eqnarray*} u_j: &  C_j \to \XX[k_-,k_+]_j := \R \times Z_- &  j= k_-,\ldots, -1 \\
u_0 : &  C_0 \to \XX[k_-,k_+]_0 := X &  \\ 
u_j:  & C_j \to \XX[k_-,k_+]_j  := \R \times Z_+ & j = 1,\ldots, k_+ \end{eqnarray*}
satisfying the following.
\begin{enumerate} 
\item ({\rm Balancing conditions}) For any two adjacent  $j,j+1$ there exists
 $\lambda_j \in [0,\infty]$ such that the following holds:   Suppose 
$T_e$ is an edge connecting $C_j$
and $C_{j+1}$ and $u|_{T_e}: T_e \to \cR(\Lambda)_\pm$ is a trajectory 
on the space of Reeb chords. Then the length of $T_e$ is 
\[ \ell(T_e) = \lambda_j , \quad \forall T_e \ \text{connecting} \ C_j,C_{j+1} .\] 
\item ({\rm Matching conditions})  If $u_i,u_{i+1}$ are adjacent levels joined by an edge $T_e = T_{e,i} \cup T_{e,i+1}$
with coordinate $s$ on $T_{e,i}, T_{e,i+1}$ then 
\[ \lim_{ s \to \infty} u_i(s) = \lim_{s \to -\infty} u_{i+1}(s) .\]
This limit should be interpreted as an element of the space $\cR(\Lambda_\pm)$ of 
Reeb chords, if the trajectory represents a trajectory on $\cR(\Lambda_\pm)$, 
or an element of $\Lambda_\pm$  if the trajectory represents a trajectory on $\cR(\Lambda_\pm)$.
\end{enumerate}

\vskip .1in \noindent  An {\em  isomorphism} of treed buildings $u',u''$ 
with $k$ levels is  an isomorphism of domains $\phi: C' \to C''$
(holomorphic on the surface parts and length-preserving on the
segments) together with translations 
\[ \tau_j: \XX[k_-,k_+]_j \to  \XX[k_-,k_+]_j , \;\;j \in 
\{1,\ldots, k \}\] 
so that
\[ u''_v \circ \phi = \tau_j \circ u'_v \] 
for each vertex $v \in \Ver(\Gamma_j)$.   This ends the Definition.
\end{definition}    

We introduce the following notation for moduli spaces with fixed limits along the 
leaves.   

\begin{definition} \label{def:maptypetreed}

\vskip .1in \noindent The {\em  map type} of a treed building $u$ is the decorated graph 
\[ \bGamma = (\Gamma, h,l) \] 
where \label{rep:maptype}
\[ h: \Ver(\Gamma) \to \pi_2(\ol{X},\ol{L}) \]  
(with $\ol{X}$ resp. $\ol{L}$ the compactification of $X = \R \times Z$ resp. 
$L = \R \times \Lambda$)  is the collection of homotopy classes of maps associated to the
vertices and
\[ l : \Edge_{\rightarrow}(\Gamma) \to \cI(\Lambda) \cup \{ D \} \]
%
is the collection of labels on the leaves; that is, the constraints
$l(e) \in \cI(\Pi)$  on constrained edges; or 
$D$ for the edge $T_e, e \in \Edge_D(\Gamma)$, constrained to map to 
$D$

Denote by $\M_{\bGam}(L)$ the moduli space of stable treed holomorphic finite-energy buildings
bounding $L$ of (possibly disconnected) type $\bGam$, and by $\M(L)$ the union over types $\bGam$.  In the case $L = \R \times \Lambda$, we use the notation 
$\M(\Lambda) = \M(\R \times \Lambda).$ 
Let $\M_\Gamma(L)$ denote the union of 
$\M_\bGamma(L)$ over map types $\bGamma$ whose underlying domain type is $\Gamma$.
Denote 
\[ \cI(L) = \zeta_{L}^{-1}(0) = \crit(f_L) \]
denote the zeroes of $\zeta_{L}$. 
Taking the limit of any trajectory along a semi-infinite edge $e$
defines an {\em  evaluation map}
\[ \ev_e: \M_\bGamma(L) \to \cI(\Lambda_-)^{d_-} \cup \cI(L)^d \cup \cI(\Lambda_+)^{d_+} .\]
For any given collection
\[ \ul{\gamma} = (\ul{\gamma}_-, \ul{\gamma}_\black,\ul{\gamma}_+) \in
\cI(\Lambda_-)^{d_-} \times \cI(L)^d \times \cI(\Lambda_+)^{d_+} \]
of generators $\cI(\Lambda_\pm)$ and critical points
$\ul{\gamma}_\black$ in the interior, let 
\[ \M(L, \ul{\gamma}) \subset \M(L) \] 
the moduli space of treed disks with limits along the leaves given by
$\ul{\gamma}$.   If the domain $C$ is disconnected, then we assume
the set of components $C_1,\ldots, C_k$ of the domain is equipped with an ordering and the incoming and outgoing labels $\ul{\gamma}_\pm$ are compatible with the ordering in the sense that the incoming labels for each component are in cyclic order around the boundary  and the labels for the various components are in the 
same order; similarly for the outgoing labels.   Also, the 
incoming labels are ordered before the outgoing labels.   This ends the Definition.
\end{definition}

\begin{remark}  \label{rem:splittype} 
Suppose that $\bGamma$ is a type of holomorphic building, $e \in \Edge(\bGamma)$ is an edge of closed resp. open type, and $\bGamma_1,\bGamma_2$ are the types of buildings obtained by cutting at the edge $e$, corresponding to a Reeb chord 
or orbit in $Z_\pm$.    Thus, if $u: C \to \XX$ is the original building, then $u_1: C_1 \to \XX$ an $u_2: C_2 \to \XX$ are 
buildings with the node replaced by an (interior or boundary) puncture.   Note that 
stability of the type $\bGamma$ does not imply stability of $\bGamma_1$ and $\bGamma_2$, but this will be irrelevant in what follows.   Assuming all moduli spaces are regular then the following is almost immediate from the definitions:
\begin{enumerate}
    \item If $e$ is an open edge, then there is a diffeomorphism
\[ \M_\bGamma(L) \cong \M_{\Gamma_1}(L) \times_{\cR(\Lambda_\pm)} \M_{\Gamma_2}(L) \] 
is the fiber product of the moduli spaces $\M_{\bGamma_1}(L)$ and $\M_{\bGamma_2}(L) $ over the discrete space of zeros of $\zeta_{\Lambda,\pm}$ in $ \cR(\Lambda_\pm)$.

    \item  If $e$ is a closed edge, then there is a diffeomorphism 
    \[ \M_\bGamma(L) \cong (\M_{\bGamma_1}(L) 
    \times_{\cR(Z_\pm)} \M_{\bGamma_2}(L))/S^1     \] 
    is a quotient of the fiber product of the moduli spaces $\M_{\bGamma_1}(\Lambda)$ and $\M_{\bGamma_2}(\Lambda) $ over the space of Reeb orbits $\cR(Z_\pm)$, by the $S^1$ action which simultaneous rotates the asymptotic markers $m_{e_1}, m_{e_2}$ at the markings created by the cut.    
    In particular, if $\dim(\M_{\bGamma_1}(L)) > \dim(\cR(Z_\pm))$
    then $\M_\bGamma(L)$ cannot be expected dimension zero, since the fiber product above cannot be a finite union of $S^1$-orbits. 
\end{enumerate}
\end{remark}

\subsection{Monotonicity assumptions}

We make several assumptions on the cobordism to allow us to 
restrict to Chekanov-Eliashberg differentials counting rational curves with one incoming 
boundary puncture.    These assumptions guarantee that spherical components
do not bubble off in one-dimensional components of the moduli space.  Such bubbling would force us to include Reeb orbits in the set of generators
of the Chekanov-Eliashberg complex.    \label{rep:suchbubbling}

\begin{definition}  \label{def:tamecob}
A fibered contact manifold $(Z,\alpha)$ with base $(Y, \omega_{Y})$ with $\d\alpha = p^*(\omega_{Y})$
is {\em  tame} if for some constant $\tau_Z \ge 1$
and integral symplectic form
\[ \omega_{Y,0} \in \Omega^2(Y,\R), \quad [\omega_{Y,0}] \in H^2(Y,\Z)  \] 
we have
\[   \omega_Y = \tau_Z \omega_{Y,0} \]
and the base is  monotone in the sense that 
there exists a {\em  monotonicity constant} $\tau_Y \ge 3$ so that 
\[ c_1(Y) = \tau_Y [\omega_{Y,0}] \in H^2(Y,\R) .\] 
Note that since $\omega_Y$ is proportional to $\omega_{Y,0}$, the Chern class $c_1(Y)$ of $Y$ does not depend on the choice of the symplectic form $\omega_Y$. For any compact spin embedded Legendrian $\Lambda \subset Z$ with embedded
image $\Pi \subset Y$ we call $(Z,\Lambda)$ a {\em  tame pair}.
\end{definition}

\begin{example}
    For $n>1$ and $k>0$, consider the circle bundle of the holomorphic line bundle $\mathcal{O}(-k)$ over $\C P^n$, which is a fibered contact manifold %
    \[ (Z,\alpha)\to (Y=\C P^n, \omega_Y=k\omega_{FS}). \] 
    It is tame for the choices 
    \[ \omega_{Y,0}=\omega_{FS}, \quad \tau_Z=k, \quad \tau_Y=n+1\geq 3. \]
\end{example}

The following is immediate from the condition that $\tau_Y \ge 3$:

\begin{lemma}\label{rep:Chern}  Suppose that $(Z,\alpha)$ is tame.  For any non-constant $J$-holomorphic map 
$u: C \to Y$ where $C$ is a closed curve, the Chern number of $u$ satisfies
\[ \int_C u^* c_1(Y) \ge 3.\]
\end{lemma}

\begin{lemma} \label{anglechange} 
Let $(Z,\alpha)$ be a fibered contact/stable-Hamiltonian manifold and $\Lambda \subset Z$ a Legendrian.
 Let $u: S \to \R \times Z$ be a punctured surface with boundary mapping to $\R \times \Lambda$ and limiting to
  Reeb chords and orbits $\gamma_e$ on the ends $e \in \mE(S)$ with
  actions 
  $\theta_e \in \R$.   Then the incoming and outgoing 
  actions
  are related by 
\[ \sum_{e \in \cE_+ (S)} \theta_e - \sum_{e \in \cE_-(S)} \theta_e = 
\int_{S} u^* \d \alpha =  
\int_S u_Y^* \omega_Y . \]
\end{lemma} 

\begin{proof} 
  Let $\hat{S}$ denote the compactification of $S$ obtained by adding
  in intervals along each strip-like end, so that $u$ in the statement of the Lemma extends to a map   from $\hat{S}$ to $\R \times Z$. Denote by 
  $u^* \alpha$ the
  pull-back connection form 
  on $u^* Z$.    By Stokes' formula, the difference in incoming and outgoing 
  actions 
  is given by 
\[  \sum_{e \in \cE_+ (S)} \theta_e - \sum_{e \in \cE_-(S)} \theta_e 
 + 
 \int_{\partial S} u^* \alpha   = 
  \int_{\hat{S}} \d u^* \alpha
= \int_{\hat{S}} u_Y^* \omega_Y .\]
The signs follow from our assumption that incoming ends are asymptotic to the convex end and outgoing ends are asymptotic to the concave end. The integral over $\partial S$ vanishes since $\alpha$ restricts to zero on $\Lambda$. 
\end{proof} 

The following corollary is used to show that the count of curves with a single input defines a  differential.

\begin{corollary} \label{onepos} Suppose that $(Z,\Lambda)$ is a tame pair  and $J$ is a cylindrical
  almost complex structure on $\R \times Z$.   Let $u:S \to \R \times Z$
  be a punctured $J$-holomorphic curve with boundary on $\R \times
  \Lambda$.  The sum of the 
  actions of Reeb chords at outgoing
  punctures is smaller  than the sum of the 
  actions 
  at the incoming punctures: 
  \begin{equation} \label{angledecreasing}
  \sum_{e \in \cE_+ (S)} \theta_e \geq  \sum_{e \in \cE_-(S)} \theta_e  \end{equation}
  with equality only if the map $u$ projects to a constant map 
  $u_Y$ to $Y$.  In particular, 
  any non-constant punctured holomorphic curve  $u: S \to  \R \times Z$  has at least one 
  incoming puncture.
\end{corollary} 

\begin{proof} The claim \eqref{angledecreasing} follows from Lemma \ref{anglechange}.
To prove the last claim, let $u$ be a map as in the statement of the Lemma.
Suppose that $u$ projects to a constant map $u_Y$ but is non-constant in $\R \times Z$.
The equation \eqref{angledecreasing} implies that there are both incoming and outgoing punctures.  On the other hand, if $u_Y$ is non-constant then Lemma \ref{anglechange} implies that there is at least one incoming puncture.
\end{proof}

Now we consider more general cobordisms. The following conditions will be used to guarantee the chain maps induced from cobordisms are well-defined. Recall that $(Z_\pm,\alpha_\pm)$ are fibered contact/stable-Hamiltonian manifolds with base the symplectic manifolds $(Y_\pm, \omega_{Y_\pm})$.  Let $(X,L)$ be a (cylindrical near infinity) cobordism pair with concave end $(Z_-,\Lambda_-)$ and convex end $ (Z_+, \Lambda_+)$ in the sense of Definition \ref{def:lagcob}. The normal bundle to $Y_\pm$ in $\ol{X}$ is the line bundle associated to $Z_\pm$.

\begin{definition}\label{relthomclass} 
Let $\ol{X}, Y_\pm$ be as above.    A {\em  Thom form} for the inclusion $Y_\pm \to \ol{X}$ is a
two-form on $\ol{X}$ given as follows.  Let $\rho:[ 0,\infty) \to [0,1]$ be a compactly supported  bump function  equal to $1$ in an open neighborhood of $0$.  We use the normal bundle neighborhood of $Y_\pm$ to define the corresponding Thom form as 
\[ \on{Thom}_{Y_\pm} =  
\d (\rho \alpha_\pm) \in \Omega^2(\ol{X},\ol{L}) \]
where 
$\Omega^2(\ol{X},\ol{L})$ denotes the space of two-forms on $X$ vanishing on each stratum of $\ol{L}$.
The Thom form is closed and the {\em  Thom class} 
%
\[ [Y_\pm]^\dual := [\on{Thom}_{Y_\pm}] \in H^2(\ol{X},\ol{L}) \] 
is independent of the choice of bump function $\rho$. By the construction, the Thom class can be also viewed as a class in $H^2(\ol{X}-Y_\mp)$. It represents the Poincar\'e dual of the divisor class $[Y_\pm]$.  
\vskip .1in  \noindent 
The {\em  logarithmic Chern class} of $ \ol{X} - Y_\pm$ is 
 \[  c_1^{\on{log}}(\ol{X} - Y_\pm) := c_1(\ol{X} - Y_\pm) - [Y_\mp]^\dual \in H^2(\ol{X} - Y_\pm). \] 
This ends the Definition.
\end{definition}

\begin{figure}

    \centering
    \scalebox{0.8}{
\begingroup%
  \makeatletter%
  \providecommand\color[2][]{%
    \errmessage{(Inkscape) Color is used for the text in Inkscape, but the package 'color.sty' is not loaded}%
    \renewcommand\color[2][]{}%
  }%
  \providecommand\transparent[1]{%
    \errmessage{(Inkscape) Transparency is used (non-zero) for the text in Inkscape, but the package 'transparent.sty' is not loaded}%
    \renewcommand\transparent[1]{}%
  }%
  \providecommand\rotatebox[2]{#2}%
  \newcommand*\fsize{\dimexpr\f@size pt\relax}%
  \newcommand*\lineheight[1]{\fontsize{\fsize}{#1\fsize}\selectfont}%
  \ifx\svgwidth\undefined%
    \setlength{\unitlength}{450bp}%
    \ifx\svgscale\undefined%
      \relax%
    \else%
      \setlength{\unitlength}{\unitlength * \real{\svgscale}}%
    \fi%
  \else%
    \setlength{\unitlength}{\svgwidth}%
  \fi%
  \global\let\svgwidth\undefined%
  \global\let\svgscale\undefined%
  \makeatother%
  \begin{picture}(1,0.5)%
    \lineheight{1}%
    \setlength\tabcolsep{0pt}%
    \put(0,0){\includegraphics[width=\unitlength,page=1]{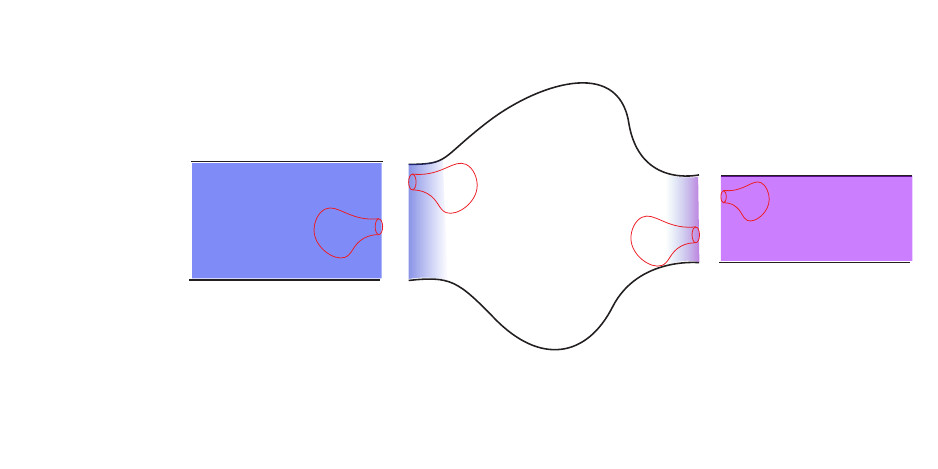}}%
    \put(0.35511813,0.24803149){\makebox(0,0)[lt]{\lineheight{1.25}\smash{\begin{tabular}[t]{l}$I$\end{tabular}}}}%
    \put(0.46141734,0.29372704){\makebox(0,0)[lt]{\lineheight{1.25}\smash{\begin{tabular}[t]{l}$II$\end{tabular}}}}%
    \put(0.68370076,0.24073489){\makebox(0,0)[lt]{\lineheight{1.25}\smash{\begin{tabular}[t]{l}$III$\end{tabular}}}}%
    \put(0.79863515,0.28089238){\makebox(0,0)[lt]{\lineheight{1.25}\smash{\begin{tabular}[t]{l}$IV$\end{tabular}}}}%
    \put(0.55629924,0.08070867){\makebox(0,0)[lt]{\lineheight{1.25}\smash{\begin{tabular}[t]{l}$X$\end{tabular}}}}%
    \put(0.29370079,0.14094488){\makebox(0,0)[lt]{\lineheight{1.25}\smash{\begin{tabular}[t]{l}$\R \times Z_+$\end{tabular}}}}%
    \put(0.81377955,0.15511813){\makebox(0,0)[lt]{\lineheight{1.25}\smash{\begin{tabular}[t]{l}$\R \times Z_-$\end{tabular}}}}%
  \end{picture}%
\endgroup%
}
    \caption{Four types of sphere breakings in cobordisms}
    \label{fig:montonicity_sphere}
\end{figure}

\begin{definition} \label{nonnegcob} \label{def:tamepair}   The pair $(X,L)$ with concave end $(Z_-,\Lambda_-)$ and convex end $(Z_+,\Lambda_+)$ is a {\em  tame cobordism pair} if and only if both ends $Z_\pm$ are tame, and the following conditions hold.

\begin{enumerate}[label={(\bfseries P\arabic*)}]
  \item \label{p1} {\rm (Rationality)} The symplectic classes
   $[\omega_{Y_\pm}] \in H^2(Y_\pm)$ and the class
  $[\ol{\omega}] \in H^2(\ol{X})$ are rational classes and the class
  $[\ol{\omega} |_{\ol{X} - Y_-}] \in H^2(\ol{X} - Y_-)$ is an integral class. 
  Also, the Lagrangian $\ol{L}$ is rational in the sense of
    Definition \ref{rational} below.
\item \label{p2} {\rm (No-cap condition)} The logarithmic first Chern class of $\ol{X} - Y_-$ is a very positive
multiple of the symplectic class on the complement of $Y_-$ in the sense that there exists a constant $\lambda_- > 0$ so that 
\[  c_1^{\on{log}}(\ol{X} - Y_-) = (1 + \lambda_-)[\ol{\omega} |_{\ol{X} - Y_-}] \in H^2(\ol{X} - Y_-) .\]
\item \label{p3} {\rm (Monotonicity for the concave end)}  If the concave end $\Lambda_-$ of the cobordism $L$ is non-empty then the relative Thom class of the divisor $[Y_-]^\dual \in H^2(\ol{X} - Y_+,\ol{L} - Y_+)$ is non-positive in the sense that there exists a constant $\lambda_+  \ge 0$ so that 
 \[  [Y_-]^\dual  =   -\lambda_+ [ \ol{\omega}|_{\ol{X} - Y_+}]  \in H^2(\ol{X} - Y_+,\ol{L} - Y_+) .\] 
\end{enumerate}
\end{definition}

\begin{discussion}[Motivation for Definition \ref{def:tamepair}]

In general, the compactness theorem in symplectic field theory shows that there can be four distinct types of sphere breaking which can occur in compactified moduli space of disks in the cobordism; they are as follows (see Figure \ref{fig:montonicity_sphere}).
\begin{itemize} 
\item A type {I} sphere cannot occur due to maximum-principle or Corollary \ref{onepos}. 
\item Type {III} spheres do not occur by item \ref{p3}; see Lemma  \ref{1cons}.  
\item A dimension counting argument and \ref{p2} will rule out spheres of type {II} and {IV}; 
see  Lemma \ref{nospheres}.
\end{itemize}
A similar set of conditions to prevent sphere bubbles in exact cobordisms is called \textit{good ends} in \cite[Appendix B.2]{Ekholm:rationalsft}. 
\end{discussion}

\begin{example} \label{rep:trivcob} Consider the cobordism $X = \C^n - \{ 0 \} \cong \R \times S^{2n-1}$ from $S^{2n-1}$ to itself.  Let 
 $L$ be the cylinder $ \R \times \Lambda$.  The compactification 
$\ol{X}$ is the blow-up of $\CP^n$ at $0$ with $Y_\pm \cong \CP^{n-1}$.  
We assume that $[\ol{\om}] \in H^2(\ol{X})$ is chosen so that, with $H \in H^2(\CP^{n-1})$ the hyperplane class  (and the same notation in relative cohomology)
\[  [ \ol{\om} ] |_{Y_-} = e^{\sigma_-} H, \quad 
 [ \ol{\om} ] |_{Y_+} = e^{\sigma_+} H \] 
for some $0 < e^{\sigma_-} < e^{\sigma_+}.$
The logarithmic first Chern class is 
\[ c_1^{\on{log}}(\ol{X} - Y_-) = n H \] 
and the outgoing divisor has dual class
\[  [Y_-]^\dual =  -H  \in  H^2(\ol{X} - Y_+,\ol{L} - Y_+).\]
Therefore the constants are 
\[ \lambda_- = \frac{n}{e^{\sigma_+}} - 1, \quad \lambda_+ = \frac{1}{e^{\sigma_-}}. \]
The cobordism pair is tame if $n > 2, e^{\sigma_+} = 1$ and $ e^{\sigma_-} < 1$, and 
for these choices the symplectic class is integral on the complement of $Y_-$.
\end{example}

More generally, a trivial cobordism of a tame contact manifold satisfies \ref{p2}.

\begin{example}\label{ex:trivco}
    As in Definition \ref{def:tamecob}, let $(Z,\alpha)$ be a tame circle fibered contact manifold over a base $(Y,\omega_Y)$. If we cut $\R\times Z$ at level $e^{\sigma_+}= 1/\tau_Z$, then $\ol{\omega}(A)= \omega_{0,Y}(A)$ for any sphere class $A$. On the other hand, we have $c_1^{\on{log}}(\ol{X}- Y_-)(A)=c_1(Y)(A)$. By the tame condition, \ref{p2} is satisfied.
\end{example}



We now discuss our motivating example given by the Harvey-Lawson 
filling.  According to Lemma \ref{asymlem}, the Harvey-Lawson
filling defines a  cylindrical-near-infinity filling $L \subset X$ of  a Legendrian $\Lambda_\eps$ with respect to some stable Hamiltonian triple $(Z,\alpha_\eps,p^*\omega_Y)$ with $Z \cong S^{2n-1}$
and $\Lambda_\eps$ a lift of a perturbation of the Clifford torus.  Although the form $\alpha_\eps$ for which $\Lambda_\eps$ is Legendrian is contact, 
the symplectic form $\omega$ on $X$ is an extension of $\alpha_0$
so that $\d \alpha_0 = p^*\omega_Y$, rather than the perturbed form 
$\alpha_\eps$.  From now on, by the {\em  Harvey-Lawson filling} 
we mean the filling $L$ of $\Lambda_\eps \subset S^{2n-1}$ produced by Lemma \ref{asymlem} applied
to the Harvey-Lawson Lagrangian \eqref{hl}.

\begin{lemma} The Harvey-Lawson filling $L \cong T^{n-2} \times \R^2$
of $\Lambda_\eps \cong T^{n-1} \subset S^{2n-1}$ is tame.
\end{lemma}

\begin{proof}  The compactified filling is $\ol{X} = \CP^n$ which has Chern class 
\[ c_1(\ol{X}) = (n+1) [\CP^{n-1}]^\dual = (n+1) [Y_-]^\dual .\]
Hence
\[ c_1(\ol{X}) - [Y_-]^\dual = n [Y_-]^\dual = (1 + \lambda_-) [\ol{\omega}] \] 
for $n \ge 2$.    The concave end $\Lambda_-$ is
empty, so the condition \ref{p3} is vacuously true. Hence $(X,L)$ is a tame cobordism pair.
\end{proof}

\begin{example}
We now give an example of a cobordism that is not tame.
Let $X$ be the blow-up of the unit ball $B_1(0)$ in $\C^n$.
Since the Harvey-Lawson filling $L$ of $\Lambda_\eps$ is disjoint from $0 \in B_1(0)$, it defines a Lagrangian filling
in $X$ as well, denoted with the same notation.
 
\begin{lemma} \label{nottame} 
Let $X,L$ be as above so that $X$ is the blow-up of $B_1(0)$ at $0$.  
The Harvey-Lawson filling $L 
\subset X$ is not tame as a filling.
\end{lemma}

\begin{proof} We check that the condition \ref{p2} fails.  The compactification $\ol{X}$ is the blow-up $\Bl_0 \CP^n$
of $\CP^n $ at $0 \in \C^n \subset \CP^n$.  
The space $\ol{X}$ fibers over 
$\CP^{n-1}$ via a map 
\[ \pi: \Bl_0 \CP^n \to \CP^{n-1} .\] 
The blow-up $\Bl_0 \CP^n \cong \P(\mO(0) \oplus \mO(1))$ is isomorphic to the projectivization of the  sum of the hyperplane and trivial bundles $\mO(1), \mO(0)$.
The Chern numbers of the fibers of the projection  are 
\[ \lan c_1(\ol{X}), 
[\pi^{-1}(\ell)] \ran = 2, \quad 
\forall \ell \in \CP^{n-1} . \] 
The evaluation of the logarithmic Chern class on a
fiber is
\[ \lan  c_1^{\on{log}}(\ol{X} - Y_-), [\pi^{-1}(\ell)] \ran  =
\lan c_1(\ol{X}), [\pi^{-1}(\ell)] \ran -[Y_+]^\dual\cdot [\pi^{-1}(\ell)]= 2 -1 = 1 \] 
which is less than the required bound in \ref{p2}. Note that we view the blow-up as a non-exact filling of the standard contact sphere and $Y_-$ is empty in this example.
\end{proof}

The moduli spaces of holomorphic disks in the example in Lemma \ref{nottame} have boundary components 
where a sphere representing a fiber has bubbled off.  Thus, for a well-defined theory of Legendrian
contact homology in which such fillings define augmentations, one would have to include the Reeb orbits
in the set of generators.  This ends the Example. \end{example}

\begin{lemma} \label{qtop} Let $\Lambda$ be a Legendrian in a tame fibered contact manifold
$Z$.  Suppose that $X = \R \times Z$ is the symplectization of $Z$ and 
%
\[\ol{X} = (\R \times Z)_{[e^{\sigma_-},e^{\sigma_+}]} \cong 
Y \sqcup ( (\sigma_-,\sigma_+) \times Z ) \sqcup Y \] \[ \ol{L} = (\R \times \Lambda)_{[e^{\sigma_-},e^{\sigma_+}]} \cong 
\Pi \sqcup ( (\sigma_-,\sigma_+) \times \Lambda ) \sqcup \Pi \]
(where $Y \cong Y_- \cong Y_+, \Pi \cong \Pi_- \cong \Pi_+)$ is the symplectic cut at $\sigma_-,\sigma_+$.   For choices of $\sigma_-,\sigma_+ \in \R$ so that 
$[e^{\sigma_-}\omega_Y] \in H_2(Y)$ is rational and $[e^{\sigma_+}\omega_Y] \in H_2(Y)$ is integral,
and $e^{\sigma_+} - e^{\sigma_-} \in \Q$, $(X,L)$ is a tame cobordism pair. 
\end{lemma} 
\begin{proof}   We check the rationality
of the symplectic class.  The degree two homology of $\ol{X}$ is generated by the images of classes in $H_2(Y)$ under either embedding $Y \to \ol{X}$, or classes in the image of the map $H(\pi^{-1}(y)) \to H_2(\ol{X})$ induced by inclusion of a fiber $\pi^{-1}(y)$ of $\pi: \ol{X} \to Y$.  This follows by the Leray-Serre spectral sequence (or, more simply, the Gysin sequence). \label{rep:gysin}   The fiber $\pi^{-1}(y)$ over any $y \in Y$ has area $e^{\sigma_+} - e^{\sigma_-}$ which is a rational number by assumption. On the other hand, the degree two homology classes in $\ol{X}-Y_-$ have integral areas by assumption. Hence, \ref{p1} is satisfied. 
Note that \ref{p2} does not involve the Legendrian. A similar argument as in Example \ref{ex:trivco} shows that \ref{p2} is satisfied in this case.

Finally, \label{rep:finally} we check the outgoing end.  We compute
 \[ [Y_-]^\dual  =  p^* c_1(Z)  =  -p^*  [\omega_Y] =   - e^{\sigma_-} [\ol{\omega}] \in 
H_2(\ol{X} - Y_+,\ol{L} - Y_+) .\] 
So \ref{p3} holds with  proportionality constant $\lambda_+ = e^{\sigma_-}$.
\end{proof} 

\begin{remark}
The conditions  in Definition \ref{def:tamepair} 
could presumably be relaxed by dealing with more complicated curve counts; for example, including Reeb orbits in the set of generators would allow relaxation of the second inequality while  presumably allowing all-genus counts would allow relaxation of the first.  However, allowing these possibilities introduces further technical difficulties. 
Note that the conditions here are weaker than the exact conditions in, for example, Chantraine \cite{chantraine:note}.   
However, the condition \ref{p3} rules out the kinds of disks shown in Figure 1 in 
\cite{chantraine:note} which prevent the cobordism maps from being well-defined, since 
there can be no positive area disks with only outgoing punctures. 
\label{rep:chantraine}
\end{remark}

\begin{lemma} \label{1cons}
Let $(X,L)$ be a cobordism satisfying \ref{p3}. Any holomorphic 
punctured 
disk $u: S \to X$ bounding $L$ with an outgoing strip-like end has at least one incoming end.
A similar result holds for holomorphic punctured spheres as well.\end{lemma}

\begin{proof}  We argue by contradiction using positivity of area. We only deal with the case of disk since the same argument works for the case of spheres as well.
Let $u: S \to X$ be a punctured disk bounding $L$ with only outgoing ends.  After compactifying $X$ to $\ol{X}$, the intersection number of the map $u$ with the outgoing divisor $Y_-$ is 
\[ [\ol{u}] \cdot [Y_-]  = -[\ol{u}] \cdot \lambda_+ [\ol{\omega}]  = -  \lambda_+ A(\ol{u}) \leq 0 .\]
Since $u$ is holomorphic with non-positive area, $u$ must be constant. By definition $u$ maps to the interior $X$, the map $u$ has no outgoing ends, which is a contradiction.  
\end{proof} 

We now explain the meaning of the no-cap condition \ref{p2}.
By a {\em  rigid holomorphic sphere} we mean a map 
$u: S \to X$ from a (possibly nodal) sphere $S$ of some type $\bGamma$ 
which lies in a moduli space of holomorphic maps
of expected dimension zero.  In section \ref{fredtheory}
we show that the moduli spaces could be made regular by suitable perturbations, so that in fact such a map is rigid 
if and only if it is isolated.  \label{rep:rigid}

\begin{lemma}  \label{nospheres} Suppose $(X,L)$ is a tame cobordism.  There are no rigid holomorphic spheres $u: S \to X$ with a single puncture $z \in \ol{S}$ constrained to map to a fixed Reeb orbit over $Y_\pm$.
\end{lemma}

\begin{proof}  {\em Case:  There is a single incoming puncture.}
Consider a punctured sphere $u:S \to X$ of type $\bGamma$ whose closure $\ol{u}$ meets the divisor $Y_+$ but not the divisor $Y_-$, and $\ol{u}$ its extension to a map to $\ol{X}$.  Recall that if the map $u$ was asymptotically close to a trivial cylinder over an $m-$fold cover of a simple Reeb orbit, then $\ol{u}$ intersects the divisor $Y_+$ with multiplicity $m$.  The dimension of the moduli space $\M_{\bGamma}(X)$ of spheres constrained to intersect $Y_+$ at a single puncture with multiplicity 
  \[ m = ( \ol{u}_* [\P^1] , [Y_+]^\dual)= \ol{u}\cdot [Y_+] \] 
is (similar to Cieliebak-Mohnke \cite[Proposition 6.9]{cm:trans}) 
\begin{eqnarray*} \dim \M_{\bGamma}(X)&=& 2( c_1(\ol{X}-Y_-), \ol{u}_* [\P^1] ) + 2n - 4 - \underbrace{2n -  2(m-1)}_{\text{ fixed Reeb orbit constraint}} \\ &=& 2( c_1(\ol{X}-Y_-), \ol{u}_* [\P^1] )  - 2 - 2m \\
 \\ &=& 2( c_1^{\on{log}}(\ol{X}-Y_-), \ol{u}_* [\P^1] )  - 2  \\
  &=& 2((1 + \lambda_-) [\ol{\omega}], \ol{u}_* [\P^1] )  - 2  > 0 . \end{eqnarray*}
Here we used that $X$ satisfies \ref{p2}. Hence such spheres cannot be rigid. 

\vskip .1in \noindent  {\em Case:  There is a single outgoing puncture.}
Let $u: S \to X$ have a single outgoing puncture asymptotic to the concave end of $X$, and with no incoming punctures then by \ref{p3}
\[  0 > -  [Y_-] \cdot [\ol{u}] = \lambda_+ [ \ol{\omega}] \cdot [\ol{u}]  .\] 
\noindent This contradicts the fact that $\ol{u}$ is holomorphic and non-constant.
\end{proof}

\begin{remark}
    The same statement holds, with the same proof, if we replace $X$ by $\R\times Z_\pm$, since our contact manifolds are tame. Therefore, we can exclude the occurrence of spheres of type {II}, {III}, {IV} in Figure \ref{fig:montonicity_sphere} in rigid holomorphic buildings.
\end{remark}

\begin{discussion} \label{disc:nosphere}
    We use Lemma \ref{1cons} to stop sphere breaking to happen in the SFT limit of a one-dimensional moduli space of treed holomorphic disks in the proofs of Theorems \ref{twolevels2} and \ref{twolevels3}. We sketch the idea of the proof here. The Lemma gives that there is no rigid treed holomorphic building of a fixed domain type which contains a sphere bubble. Indeed, any treed disk with a sphere bubble will have a terminal `leaf sphere' which is a sphere with a single puncture which is asymptotic to $Z_\pm$. Lemma \ref{1cons} shows that any treed disk with such a terminal sphere can not be rigid since the moduli space of sphere with a puncture with a point constraint on the divisor $Y_+$ (along with multiplicity) is positive dimensional. 

\end{discussion}

\subsection{Homological invariants}

In the definition of contact homology we will work with coefficient rings
that are completed group rings for various homology groups. 
The standard in the field is to use a group ring on first homology and assign homology classes to punctured disks by closing up the Reeb chords
at infinity.   
Denote by 
\[ \Lambda_1,\ldots, \Lambda_k \subset \Lambda \]
the connected components of the Legendrian $\Lambda$.
For each connected component $\Lambda_i$ of $\Lambda$ choose
  a base point 
  \[ {\lambda}_i \in \Lambda_i, \quad i = 1,\ldots, k .\]   

\begin{definition} \label{boundclass}  
\begin{enumerate}
    \item {\rm (Capping paths)} Let
  $\gamma: [0,1] \to Z$ be a Reeb chord. A {\em  capping path} for $\gamma$ is a path
  \begin{equation} \label{cappath}
 \hat{\gamma}_b: [0,1] \to \Lambda, \quad \hat{\gamma}_b(0) =
 \gamma(b), \quad \hat{\gamma}_b(1) = {\lambda}_i\end{equation} 
 for any $b \in \{ 0, 1 \}$.
\item {\rm (Boundary cycles)}
Let  $u: S \to \R \times Z$ be a holomorphic map bounding $\R \times \Lambda$ with limits on the strip-like ends given by the
collection of Reeb chords 
\[ \ul{\gamma} = (\gamma_0,\ldots, \gamma_k) \in \cR(\Lambda)^{k+1} \] 
with capping paths $\hat{\gamma}_{i,b}, i = 0,\ldots, k+1.$
The boundary $\partial S$ is divided into components 
\[ \partial S = (\partial S)_0 \cup \ldots \cup (\partial S)_k \] 
by the strip-like ends.  The restriction 
$(\partial u)_m$ of $u$ to $(\partial S)_m$ after projection to $\Lambda$ 
concatenates with the capping paths for chords in $\ul{\gamma}$ to
produce a sequence of {\em  boundary cycles}  
\[ \ol{\partial u}_m := \hat{\gamma}_{m+1,0} \circ (\partial u)_m \circ \hat{\gamma}_{m,1}^{-1}  : S^1 \to \Lambda_{i(m)}, m = 0,\ldots, k \] 
in the components $\Lambda_{i(m)} $. 
\item {\rm (Total homology)}  The total homology class of the
combined boundary cycles is denoted
\begin{equation} \label{cappedoff} [\partial u] = \sum_k
  [ \ol{\partial u}_j ] \in H_1(\Lambda) .\end{equation}
\end{enumerate} 
This ends the Definition.
\end{definition}

\begin{figure}[ht]
    \centering
    \scalebox{.5}{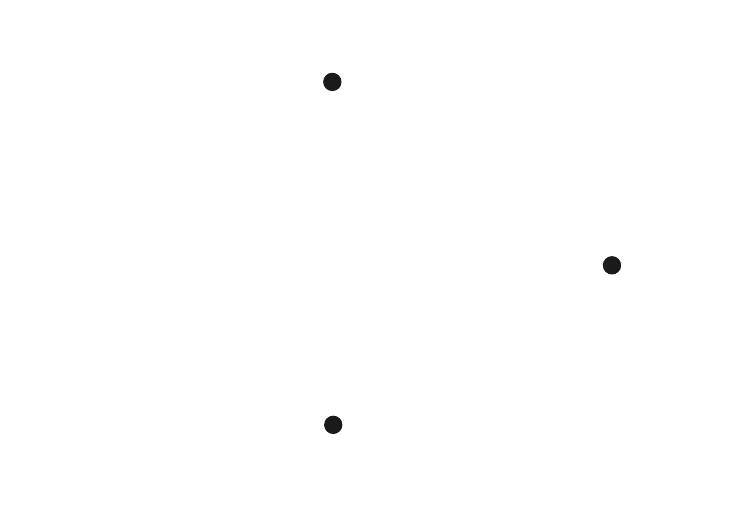}
    \caption{Capping paths to make a cycle}
    \label{capping}
\end{figure}

\begin{remark} 
Similarly, in the case that the Legendrian is connected and the contact manifold is simply-connected, second homology classes are associated to punctured disks by a capping procedure.
For each Reeb chord 
$\gamma$, the loop obtained 
by concatenating $\gamma, \hat{\gamma}_0, \hat{\gamma_1}$ is the boundary 
of some continuous map  $v_\gamma: S_\gamma \to Z$ from a disk 
$S_\gamma$ to $Z$ by the assumption that $Z$ is simply-connected.  The homology class obtained
from a punctured disk $u :S \to Z$ by gluing in the disk $v_\gamma$
at each Reeb chord or orbit is denoted 
\begin{equation} \label{cappedoff2} [u] = \left[ u_*(S) + \sum_\gamma 
v_{\gamma,*} S_\gamma \right] \in H_2(Z,\Lambda).\end{equation}
This ends the Remark.  \end{remark}

Finally,  we define homology classes of disks in cobordisms. 

\begin{definition} Let $(X,L)$ be a symplectic cobordism.  Let $\ol{L}$ denote the closure of $L$ in the manifold with boundary $\ti{X}$.
For each component $\ti{L}_i$ of the cobordism $\ti{L}$ choose a choice of base point
${\lambda}_i$
in each component and a collection of paths $\hat{\gamma}_b, b \in \{ 0,1 \}$ from the base point 
to the start and end points of any Reeb chord $\gamma$.  
Let  $u: S \to X$ bounding $L$ 
punctured holomorphic curve.  By concatenating with the disks 
$v_{\gamma_b}$ and paths 
$\hat{\gamma}_b$ for each strip-like end of $S$ one obtains homology
classes  denoted 
\[  [\partial u] \in H_1( \ti{L}) \cong H_1(L) . \]   
If $L$ is connected and $X$ is simply connected, one obtains 
second homology classes  
\[ [u ] \in H_2(\ti{X}, \ti{L}) \cong H_2(X,L) \] 
by a choice
of capping disk at each Reeb chord.  Let $u = (u_1,u_2,\ldots, u_k)$ be a
  holomorphic building with levels $u_1,u_2, \ldots, u_k$.  Define 
\[ [\partial u] := [ \partial u_1] + [\partial u_2] + \ldots +
[\partial u_k] , \quad  [ u] := [  u_1] + [ u_2] + \ldots +
[ u_k] . \]
The {\em  boundary homology class} of a treed holomorphic disk 
\[ u: C \to \R \times Z \quad \text{or} \quad u: C \to X \] 
is defined as follows:  In the first case of maps to $\R \times Z$, the boundary components  $(\partial S)_i, i = 1,\ldots, l(i)$ of each disk component $u_v: S_v \to X$ define paths
\[ (\partial u_v)_i : (\partial S)_i \to \Lambda_{k(i)}, \quad i = 1,\ldots, l(i)  \] 
where
$\Lambda_{k(i)} \subset \Lambda$ is some connected component.  Each
trajectory $u_e: T_e \to {\cR}(\Lambda)$ defines by composition a pair
of maps 
\[ u_{e,b}: T_e \to \Lambda_{k_b(e)}, \quad b \in \{ 0, 1 \}  \] 
by composing $u_e$ with
evaluation ${\cR}(\Lambda) \to \Lambda$ at the start or endpoint of the
path.  The boundary homology class $[\partial u] \in H_1(\Lambda)$ is defined as the
class of the sum of the chains above.   In the case of a treed disk in $X$ bounding a cobordism $L$, 
one obtains a class in $H_1(\ti{L}) \cong H_1(L)$ by a similar construction.  

\vskip .1in \noindent  
The {\em  
boundary class} of a treed building $u: C \to \XX$ is the 
element $[\partial u] \in H_1(L)$ obtained by concatenating 
the restrictions $  \partial u_v$, the compositions of 
the trajectories $u_e: T_e \to {\cR}(\Lambda)$ with the 
projections $\cR(\Lambda) \to \Lambda$ given by evaluation at the end point or starting point, 
and the capping paths to the base points $\lambda_i$ in the corresponding components $\Lambda_i$ of $\Lambda$.
This ends the Definition.
\end{definition}

\vskip .1in

\section{Regularization}
\label{foundsec}

In this section, we discuss the necessary Fredholm theory, compactness
results, and regularization techniques to make the moduli spaces of
treed buildings of dimension at most one compact manifolds with boundary.
The compactness results in
symplectic field theory are similar to those covered in 
\cite{sft}, \cite{abbas:sft}.  We take as a
regularization scheme the stabilizing divisors scheme of
Cieliebak-Mohnke \cite{cm:trans}; the adaptation to the case at hand requires a small extension of the results in Pascaleff-Tonkonog \cite{pt:wall}
to the case of clean intersection.  The reader could presumably substitute their favorite regularization scheme.

\subsection{Donaldson hypersurfaces}
\label{sec:don}

To achieve manifold structures on the moduli spaces of treed holomorphic maps we use Cieliebak-Mohnke perturbations \cite{cm:trans}.  Since the compactification of our Lagrangians are only cleanly-intersecting, we extend the construction to that case, generalizing the construction in Pascaleff-Tonkonog \cite{pt:wall}. 

\begin{definition}
\vskip .1in \noindent 
A {\em  Donaldson hypersurface} in a symplectic manifold $M$ with symplectic form $\omega$ is a codimension two symplectic submanifold $D \subset M$ such that the Poincar\'e dual cohomology class $[D]^\dual$ satisfies 
  \[ [D]^\dual =  k [\omega] \in H^2(M) \]
for some positive $k > 0 $ called the {\em  degree} of $D$. 
\end{definition}

\vskip .1in
Now we assume that $X$ is a symplectic cobordism with circle-fibered boundary and let $\ol{X}$ be its compactification, as in Section \ref{tocob}. Results of Cieliebak-Mohnke \cite{cm:trans} and Charest-Woodward
\cite{cw:traj} imply that there exist Donaldson hypersurfaces which intersect 
holomorphic disks and spheres in only finitely many points:

\begin{proposition} \label{prop:je}
  There exists a $k_0 \in \Z_+$ so that if $D \subset \ol{X}$ is a Donaldson
  hypersurface $D$ of degree at least $k_0$ then there exists a tame almost
  complex structure on $\ol{X}$  so that $D$ contains no rational curves $u: S^2 \to \ol{X}$.
  Furthermore, for any such $J_0$ and energy bound $E$, there exists an
  open neighborhood, 
  \[ \J^E(\ol{X}) \subset \J(\ol{X}),   \]
  in the space of tame almost
  complex structures on $\ol{X}$ with the property that if
  $J_1 \in \J^E(\ol{X})$ then any non-constant $J_1$-holomorphic
  sphere $u: S^2 \to \ol{X}$ of energy at most $E$ meets $D$ in at least three but finitely many points $u^{-1}(D)$.
  \end{proposition} 

\begin{remark}
  In the case that $X$ is a complex projective variety, a result of Clemens
  \cite{clemens} implies that we may take $D$ to be a Bertini hypersurface, 
  and leave the complex structure unperturbed. 
  \end{remark}
  
  Our regularization procedure requires Donaldson hypersurfaces that intersect each non-constant punctured holomorphic disk.  Let $\ol{L} \subset \ol{X}$ be the closure
  of a Lagrangian cobordism, so that $\ol{L}$ is the union of 
  a smooth Lagrangian $L$ and Lagrangians $\Pi_\pm$ in $Y_\pm$, with 
  angles between the branches at the self-intersections bounded from below, as in Definition \ref{def:branchedLag}.

\begin{definition} \label{rational} A cleanly self-intersecting Lagrangian $\ol{L} \subset \ol{X}$ is {\em  rational} if there exists a line-bundle-with-connection
  \[ \hat{X} \to \ol{X} \]
  whose curvature 
  is $k \ol{\omega} \in \Omega^2(\ol{X}) $ for some integer $k > 0$
such that the pull-back of $\hat{X}$ to the Lagrangian $\ol{L}$
has a flat section.  

\vskip .1in \noindent 
The Lagrangian $L$ is {\em  exact} on a subset 
$U \subset X$ with $\ol{U} \subset \ol{X}$ open if there exists
a one-form $\alpha \in \Omega^1(\ol{U})$ with $ \d \alpha = \ol{\omega} |_{\ol{U}}$
and a continuous function $f : \ol{L} \to \R$ so that $\d f|_L = \alpha |_L$.
\noindent A symplectic cobordism $\widetilde{X}$ with concave end $Z_-$ and convex end $Z_+$ is {\em rational} if the associated
symplectic manifold $\ol{X}$ containing $Y_-,Y_+$ is rational.

\noindent Similarly, a Lagrangian cobordism
$\widetilde{L} \subset \widetilde{X}$ is {\em rational} if the associated branched Lagrangian $\ol{L} \subset \ol{X}$ is rational.   This ends the Definition.  
  \end{definition} 

  
\begin{lemma} Suppose that $\ol{X}$ is simply connected and every disk $u: S \to \ol{X}$ with boundary on $\ol{L}$ has rational area $\lan [u], [\ol{\omega}] \ran\in \Q$.  Then $\ol{L}$ is rational. 
\end{lemma}

\begin{proof}   We will construct the require section over the Lagrangian by parallel transport.  Let $\hat{X}$ be a line bundle-with-connection
  with curvature $k \ol{\omega}$.  After replacing $k$ with a sufficiently large integer multiple, every disk bounding $\ol{L}$ has integer $k\ol{\omega}$-area. It follows from the assumption that $X$ is simply connected that the holonomy around any loop in $\ol{L}$ is trivial, as in the computation in the proof of Proposition \ref{discrete}. Parallel transport of an arbitrary point in $\hat{X}$ along paths in $\ol{L}$
  then defines the desired flat section $ s: \ol{L} \to \hat{X}^{\otimes k} | \ol{L}$.
\end{proof}

\begin{example}  
    The filling of $\Lambda_\epsilon \subset S^{2n-1}$ from \eqref{eq:lambdaeps}, \eqref{hl}, induced from the Harvey-Lawson filling can be assumed to be rational for a judicious choice of $\eps>0.$   We first choose the symplectic cut so that the filling has compactification $X = \CP^n$
    with symplectic class the generator of $H^2(X,\Z)$.   Now any disk bounding 
    $\ol{L}_\eps$ can be deformed to a disk bounding $\Pi_\eps = p(\Lambda_\eps) \cong (S^1)^{n-1}$,
    since $L_\eps \cong S^1 \times \R^{n-1}$.
    Then $H_2(L_\eps)$ is generated by the $H_2(X)$ and lifts of the generators of $H_1(L_\eps)$, by the long exact sequence for relative homology, and have rational area for $\eps$ rational.  
\end{example}

The idea of creating Donaldson divisors in rational symplectic manifolds as introduced in \cite{donaldson:symplsub}  was to take the zero locus of an ``almost" holomorphic section of the line bundle $\hat{X}^{\otimes k}$. We recall the construction of Donaldson divisors in the complement of an isotropic submanifold was done in \cite{aurmohgay:symplchyp}.  The idea is to find a sequence of almost holomorphic sections over tensor powers of the line bundle $\hat{X}^{\otimes k}$ which is concentrated away from zero on the isotropic submanifold. Thus,  the zero-locus of such a section will not intersect the isotropic submanifold. 
In \cite[Theorem 3.3]{pt:wall}, it is proved that one can choose a Donaldson divisor in the complement of two isotropic  which are cleanly intersecting. We extend the ideas introduced there to construct Donaldson divisors on the complement of a branched rational Lagrangian. \label{rep:46}

\begin{proposition} \label{prop:concentrated}
  Let $\ol L$ be a branched rational
  Lagrangian submanifold then for large $k$, there exists a section $\sigma_{k,L}$ of $\hat{X}^{\otimes k}$  which is concentrated over $\ol L$ in the sense of \cite{aurmohgay:symplchyp}.
\end{proposition}

In preparation for the proof, we cover the singular locus by Darboux balls as follows. \label{rep:linearLagsubspace}
Let $\ol{L}$ be a branched rational Lagrangian.  
Choose Darboux charts at $p\in L_\cap$ so that the intersection with $L$ is the union of linear Lagrangian subspaces.  We will use these charts to construct peak sections $\sigma_{k,p}$ of $\widehat{X}^{\otimes k}$ as done in Auroux-Gayet-Mohsen \cite{aurmohgay:symplchyp}.  Define 
\[d_k: \ol{L} \times \ol{L} \to \R_{ \ge 0 } \]
to be the distance with respect to the scaled metric $kg$ where $g$ is the compatible metric on 
$X$
induced from $J$.  Since $\ol{L}$ is a rational Lagrangian, there exists a flat section 
\[ \tau_k: \ol{L} \to \widehat{X}^{\otimes k} |_{\ol{L}} . \] 
We construct an approximately holomorphic section $\sigma_{k,L_\cap}$ by the process described in \cite{aurmohgay:symplchyp}.  For a Darboux ball $\psi: B_r(0) \to X$, we call $\psi(0)$ to be the center of the Darboux ball. We choose the Darboux balls $(B_i,\psi_i)$ so that $\ol{\partial}\psi (0) =0$. Near any $p\in L_{\cap}$, there is the symplectomorphism $\phi$ from Definition \ref{def:branchedLag}, under which $\ol L$ is a union of $k$ Lagrangian subspaces in the local model.  Cover $L_\cap$ with such Darboux balls 
\[ \cup B_i \supset L_\cap  \] 
%
centers are $\frac{2}{3}$ distance away from each other. Since there  is a finite number of balls, there is an uniform bound on the size of the Darboux balls, and $|\nabla\ol{\partial}\psi|$ in terms of the $d_1$ metric. Let 
\[ P_k \subset L_\cap  \] 
denote the center of the balls. \label{rep:peaksection}

We briefly recall the construction of peak sections $\sigma_{k,p}$ from \cite{aurmohgay:symplchyp}. For a Darboux chart with center $p$,  a suitable unitary gauge-transformation provides a trivialization of $\hat{X}^{\otimes k}$ where the connection 1-form 
is $\frac{k}{4}\sum ( z_jd\ol{z}_j - \ol{z}_jdz_j).$
The sections $\sigma_{k,p}$ are defined as $\beta(z)\exp(-k|z|^2/4)$ in the Darboux balls where $\beta$ is a smooth cut-off function which vanishes outside of radius $k^{-1/6}$.   

Define a sequence of approximately holomorphic sections concentrated along the Lagrangian intersection locus $L_\cap$ as follows: 
\begin{equation}\label{eq:peaksecoverintersect}
    \sigma_{k,L_\cap} := \sum_{p\in P_k} \frac{\tau_k(p)}{\sigma_{k,p}(p)}\sigma_{k,p}: X \to \hat{X}^k.
\end{equation}

\begin{lemma}\label{lemma:peakseclowerbd}
For large enough integers $k$, if $d_k(x,L_\cap) < k^{\frac{1}{30}}$ for $x \in \ol L$, then 
\[ |\sigma_{k,L_\cap}(x)|> \frac{1}{2}e^{-2(d^2_k(x, L_\cap) 
+ d_k(x,L_\cap))}. \]
\end{lemma}

\begin{proof}  We run through the construction of asymptotically holomorphic sections again, paying careful attention to the differences in phases. 
Fix some $x \in \ol L$ such that $d_k(x,L_\cap)<k^{\frac{1}{30}}$. Let $p_f\in P_k$ be the center of the ball which contains the closest point in 
$L_\cap$ to $x$.
By increasing $k$, we can further assume that 
the distance from $x$ to the center $p_f$ satisfies the bound
\[ d_k(p_f,x) < k^{\frac{1}{20}}. \] 
The above inequality follows from the fact that for large $k$ we have $k^{1/20} > k^{1/30}+1.$

We now separate the estimate for the section into contributions whose centers are close resp. far away from the given point.    By the proof of Lemma 4 of \cite{aurmohgay:symplchyp} verbatim, when $p \in L_\cap$ and $x\in \ol L$ are at most $k^{1/10}$ distance apart, their phases are close in the sense that 
\begin{equation} \label{prox} \left| \arg\left(\frac{\sigma_{k,p} (x) }{\tau_k(x)}\right) -\arg\left(\frac{\sigma_{k,p} (p) }{\tau_k(p)}\right)\right| \leq \frac{\pi}{4} . \end{equation}
By the proximity of the phases $\arg (\frac{\sigma_{k,p}}{\tau_k})$ to each other in \eqref{prox}, 
\begin{align}\label{eq:sum_bigger_than_pf}
\left|  \sum_{p\in P_k|d_k(x,p)<k^{\frac{1}{20}}}
  \frac{\tau_k(p)}{\sigma_{k,p}(p)}\sigma_{k,p}(x) \right| &\geq \left|\sigma_{k,p_f}
                                                   (x) \right| .
\end{align}
The cutoff function used to define the section $\sigma_{k,p_f}$ is equal to one at $x$. From the definition of $p_f$ and the triangle inequality
\[ d_k(x,L_\cap) + 1 > d_k(x,p_f) .\] 
Thus for any $x \in \ol L$ with $d_k(x,L_\cap) < k^{\frac{1}{30}}$,
\begin{equation}
    |\sigma_{k,p_f}(x)| \geq e^{-d^2_k(x,p_f)} \geq e^{-(d_k(x,L_\cap)+1)^2}.
\end{equation}
Thus by combining equations \eqref{eq:sum_bigger_than_pf} and \eqref{eq:lowerboundforclose} we get 

\begin{equation}\label{eq:lowerboundforclose}
    \left|  \sum_{p\in P_k|d_k(x,p)<k^{\frac{1}{20}}}
  \frac{\tau_k(p)}{\sigma_{k,p}(p)}\sigma_{k,p}(x) \right| \ge e^{-(d_k(x,L_\cap)+1)^2}
\end{equation}

To estimate the section from ``far away'' centers, recall that 
Lemma 3 from \cite{aurmohgay:symplchyp} implies that there exists $\widetilde c, \widetilde r >0$ independent of $k, P_k$ such that
\begin{equation}\label{eq:upperboundforfar}
    \left|\sum_{p\in P_k|d_k(x,p)>k^{\frac{1}{20}}} \frac{\tau_k(p)}{\sigma_{k,p}(p)}\sigma_p(x) \right| \leq \widetilde{ c} e^{-\widetilde{r} k^{\frac{1}{20}}} \; \forall x\in X.
\end{equation}

We now combine the estimates from the close and far-away centers.
From the inequalities (\ref{eq:lowerboundforclose})  and (\ref{eq:upperboundforfar}), we get that for large $k$, 
\begin{align*}
     |\sigma_{k,L_\cap}(x)| &= \left|\sum_{p\in P_k|d_k(x,p)<k^{\frac{1}{20}}} \frac{\tau_k(p)}{\sigma_{k,p}(p)}\sigma_p(x)  +  \sum_{p\in P_k|d_k(x,p)>k^{\frac{1}{20}}} \frac{\tau_k(p)}{\sigma_{k,p}(p)}\sigma_p(x) \right| \\ \text{so} \ 
    |\sigma_{k,L_\cap}(x)| &\geq \left|\sum_{p\in P_k|d_k(x,p)<k^{\frac{1}{20}}} \frac{\tau_k(p)}{\sigma_{k,p}(p)}\sigma_p(x) \right| - \left| \sum_{p\in P_k|d_k(x,p)>k^{\frac{1}{20}}} \frac{\tau_k(p)}{\sigma_{k,p}(p)}\sigma_p(x) \right| \\ \text{so} \ 
     |\sigma_{k,L_\cap}(x)| &\geq e^{-(d_k(x,L_\cap)+1)^2} - \widetilde c e^{-\widetilde{ r} k^{1/10}} \\\ \text{so} \ 
      |\sigma_{k,L_\cap}(x)|&\geq \frac{1}{2}e^{-2(d^2_k(x, L_\cap) + d_k(x,L_\cap))}.
\end{align*}
The last inequality follows from  the chain of inequalities 
\[ \frac{1}{2}e^{-d^2_k(x,L_\cap)-d_k(x,L_\cap)} >\frac{1}{2} e^{-k^{1/15} - k^{\frac{1}{30}} }> \widetilde c e^{-\widetilde r k^{1/10}}  \]  
for large $k$. This gives the desired inequality 
\begin{equation*}
    |\sigma_{k,L_\cap}(x)| \geq \frac{1}{2}e^{-2(d^2_k(x, L_\cap) + d_k(x,L_\cap))}.
\end{equation*}
\end{proof}

\begin{proof}[Proof of Proposition \ref{prop:concentrated}]
We prove the statement of the Proposition for the special case where $\ol{L}$ is the union of $m$ Lagrangians $L_i$ which cleanly intersect at $L_\cap$ to avoid notational clutter, the general case follows from an usual patching argument. The main difficulty lies in construction of an approximately holomorphic section such that it does not vanish on the different branches of $\ol{L}$ near $L_\cap$.
We aim to construct an approximately holomorphic section concentrated on the Lagrangian $\ol{L}$ by summing the locally Gaussian sections  $\sigma_{k,L_{\cap}}$ as defined in Equation (\ref{eq:peaksecoverintersect}) and Lemma \ref{lemma:peakseclowerbd}.  

We define a section concentrated on the Lagrangian branches $L_i$ by summing the sections on peak sections which match up with the peak sections on the intersection locus $L_\cap$. 
Recall that $P_k \subset L_\cap$ was defined to be the collection center of the balls of radius one which covered $L_\cap.$
We start with balls centered around $P_k$ and extend it to a finite cover of $L_i$ by radius $ \frac{1}{3}$ balls. Denote the set of centers of these balls as 
\[ P^{\on{pre}}_k \subset L_i .\] 
Iterated removal of points from $P^{\on{pre}}_k$  ensures that any two points in $P^{\on{pre}}_k$ are at least $\frac{2}{3}$ distance apart.  
That is, if $p,q \in P^{\on{pre}}_k$ are closer than $\frac{2}{3}$ we drop either $p$ or $q$ from the set $P^{\on{pre},i}_k$. We adopt the convention of not removing points from $P_k$ while doing the iterated removal. Let $P^i_k$ denote the result of the iterated removal process. We denote the section by summing over peak sections over the balls centered at $P^i_k$ as 
\[\sigma_{k,L_i} := \sum_{p\in P^i_k} \frac{\tau_k(p)}{\sigma_{k,p}(p)}\sigma_{k,p}: X \to \hat{X}^k \].
The proof of Lemma 4 of \cite{aurmohgay:symplchyp} goes over verbatim to show that for large $k$, phases at close points are close:  That is, if  
\[ d_k(p,x) \leq k^{1/10}, \quad p,x \in L_i \] 
then 
\[  \left| \arg\left(\frac{\sigma_{k,p} (x) }{\tau_k(x)}\right) -\arg\left(\frac{\sigma_{k,p} (p) }{\tau_k(p)}\right)\right| \leq \frac{\pi}{4} .\]

We show that by taking a weighted sum of the sections constructed before, we can get a sequence of approximately holomorphic sections concentrated on $\ol{L}$.  We imitate the argument of Theorem 3.3 in Pascaleff-Tonkonog \cite{pt:wall}. To match with the notation in \cite{pt:wall} do the following renaming: call $\sigma_{k,L_i} $ as $s^{(k)}_i$ and $\sigma_{k,L_\cap}$ as $s^{(k)}_\cap$. Note that the following bounds hold for some fixed parameters  $c_l,c_M,r_l,r_M$ which do not depend on $k$:
\begin{itemize}
    \item $ |s^{(k)}_\cap(x)| \geq c_l \exp(-r_l(d^2_k(x, L_\cap) + d_k(x,L_\cap)))$ for $x\in \ol L$, with $d_k(x,L_\cap)< k^{\frac{1}{30}}$ from Lemma \ref{lemma:peakseclowerbd},
    \item $ c_l \leq |s^{(k)}_i| < 1$ for all $i$, 
    and
    \item $|s^{(k)}_i(x)| \leq c_M\exp(-r_M(d^2_k(x,L)))$.    
\end{itemize}
Define $c_\theta \in \R $ such that for $p\in L_i, q \in L_j$ with $i\neq j$ which are at a distance $d$ from $L_\cap$,  then $d_k(p,q) \geq c_\theta d $ . Such a constant $c_\theta$ exists because the angles between $L_i,L_j$ have a lower bound, by assumption.  Let $\delta>0$  be large enough so that 
\[ c_M\exp(-r_Mc^2_\theta \delta^2) < \frac {c_l}{10m^2} . \] 
This choice of $\delta$ does not depend on $k$. Set 
\begin{equation} \label{cdef}
c := \frac{c_l\exp(-r_l(\delta^2+\delta))}{10m^2}. \end{equation}
and
\[ s^{(k)} := s^{(k)}_\cap + c\sum s^{(k)}_i:  \ol{X} \to \hat{X}^k . \]

We show that the sections $s^{(k)}$ are concentrated on $\ol L$ for large $k$. For $k$ large, we have $\delta  < k^{\frac{1}{30}}$. Thus, in a $\delta$ neighborhood of $L_\cap$, from Lemma (\ref{lemma:peakseclowerbd}), we have that 
\[ |s^{(k)}_\cap (x)|  > c_l\exp(-r_l(\delta^2+\delta)) \; \forall x\in \ol L\] 
on $L_i$. 
By \eqref{cdef} the sections 
\[ |cs^{(k)}_j |  < \frac{1}{10m^2} c_l\exp(-r_l(\delta^2+\delta)) .\]
Thus, their sum is bounded above in a $\delta$ neighborhood of $L_\cap$ as follows: 
\[ \left |\sum_{j=1}^m cs^{(k)}_j \right| \leq m\frac{c_l\exp(-r_l(\delta^2+\delta))}{10m^2} =\frac{c_l\exp(-r_l(\delta^2+\delta))}{10m}.  \]
Thus we have in a $\delta$ neighborhood of $L_\cap$,  
\[ |s^{(k)}| \geq |s^{(k)}_\cap| - \left|\sum_{j=1}^m cs^{(k)}_j\right| \geq
  \left(1-\frac{1}{10m} \right) c_l\exp(-r_l(\delta^2+\delta)) .\]
When $x\in L_i$ and outside an open neighborhood of $\delta$ distance from $L_\cap$ we have that $s^{(k)}_\cap +  cs^{(k)}_i$ is bounded below by $cc_l$ on $L_i$ whereas we have an upper bound 
\[ |cs^{(k)}_j| \ < cc_M\exp(-r_Mc_\theta\delta^2) < \frac{cc_l}{10m^2} . \]
Thus, after adding the $m-1$ sections $cs^{(k)}_j$, for $j\neq i$, to $s^{(k)}_\cap + cs^{(k)}_i$ we have the following lower bound for any point on $L_i$:
\[ \left| s^{(k)}_\cap + cs^{(k)}_i + \sum_{i\neq j} cs^{(k)}_j \right| \geq \left | s^{(k)}_\cap + cs^{(k)}_i \right| -  \sum_{i\neq j} |cs^{(k)}_j |\geq cc_l - cc_l\frac{m-1}{10m^2} > \frac{cc_l}{10m^2}. \]  
We get a global lower bound of $s^{(k)}$ on any point in $\ol L$:
\[ \left | s^{(k)}(x) \right | \geq \min \left\{\frac{c c_l}{10m^2} , c_l\exp(-r_l(\delta^2+\delta))\right\}. \] 

\end{proof}

  \begin{corollary}\label{corr:existenceofD} 
  \begin{enumerate} 
  \item \label{ylag} 
  Let $Y$ be a compact rational symplectic manifold with symplectic form $\omega_Y$, and $\Pi \subset Y$ a compact rational Lagrangian submanifold.  There exists a $k_0 > 0 $ so that for any $k_1 > k_0$ there exists a $k > k_1$    and a Donaldson hypersurface $D \subset Y $ with cohomology class $k [\omega_Y]$ disjoint from $\Pi$ so that $\Pi$ is exact in $Y - D$. 
   \item \label{coblag}
  Let $(\widetilde{X}, \widetilde{L})$ be a rational cobordism between  $(Z_-,\Lambda_-)$ and
   $(Z_+,\Lambda_+)$.  Suppose that $(Z_\pm,\Lambda_\pm)$ fibers over 
  $(Y_\pm,\Pi_\pm)$ as in the previous item \eqref{ylag}, and suppose that $Y_\pm$ are equipped with Donaldson hypersurfaces for which $\Pi_\pm$ is exact in $Y_\pm - D_\pm$ for each degree $k$ 
  as in the previous item.   Possibly after increasing $k$, there there exists a Donaldson hypersurface $D$  in $\ol{X}$ of degree $k$ in $\ol{X}$ with $D \cap Y = D_\pm$ such that $\ol{L}$ is exact in the complement of  $D$. 
\end{enumerate} 
%
  \end{corollary}
  
  \begin{proof} The first item is \cite[Theorem 3.6]{cw:traj}.    The proof of the second item is similar:   Since $\ol{L}$ is a branched Lagrangian in $\ol{X}$, the claim follows from Proposition \ref{prop:concentrated} and the perturbation technique described in Donaldson \cite{donaldson:symplsub} and Auroux-Gayet-Mohsen \cite{aurmohgay:symplchyp}. 
  \end{proof}

\subsection{Fredholm theory}
\label{fredtheory}
The moduli space  of holomorphic disks with boundary on a Lagrangian cobordism is cut out locally by a Fredholm map. Note that the domains of our maps are treed disks with punctures, we need to consider linearizing PDE's over domains with cylindrical ends and over intervals. The Fredholm theory for treed disks without punctures can be found in \cite[Section 4.3]{flips}. The Fredholm theory for maps with cylindrical ends can be found in \cite[Lecture 4]{wendl:sft}. Our following construction is a combination of these two approaches.

We state our results in the context of a Lagrangian cobordism in a symplectic cobordism as in Section 2.2. Recall that $(\ol{X}, \ol{\omega})$ is a closed symplectic manifold, $Y_{\pm}$ are two disjoint symplectic hypersurfaces in $\ol{X}$ and $X= \ol{X}- (Y_- \sqcup Y_+)$. \label{rep:both} Moreover $\ol{L}$ is a Lagrangian in $\ol{X}$ such that $\Pi_{\pm}= \ol{L}\cap Y_{\pm}$ are Lagrangians in $Y_{\pm}$. Then $L=\ol{L}- \Pi_{\pm}$ is a Lagrangian cobordism in $X$ between $\Lambda_+$ and $\Lambda_-$.

Our Sobolev maps are perturbations of a class of model maps
 defined as follows.  \label{rep:of}
Let $S$ be a punctured disk with ends $ \mE(S).$  

\begin{definition} 
\vskip .1in \noindent 
A {\em  model map} is a smooth map $u:(S, \partial S) \to (X, L)$ such that, with 
coordinates given by some collection of cylindrical and strip-like ends 
\[ \kappa_e: \pm(0,\infty) \times [0,1 ] \to S, \quad \kappa_e: \pm(0,\infty) \times S^1 \to S \] 
the map $u$  is linear 
on each end $e$ 
in the sense that 
\begin{equation} \label{linear} u(\kappa_e(s,t)) = (\mu s, \gamma(\mu t)) \in \R \times Z,  \quad \forall (s,t) \in \pm(0,\infty) \times [0,1] \end{equation} 
where $\gamma$ is a Reeb chord  with action $\mu$ or similarly in the case that $\gamma$ is a Reeb orbit.  The set of model maps is defined as 
\[ \Map_0(S, X, L) = \{ u : (S, \partial S) \to (X, L) \ | \ \text{ \eqref{linear} holds for all } \ e \in \mE(S) \}.\] 
\end{definition} 

Our Sobolev completion is defined as a space of maps that are obtained by geodesic exponentiation of a Sobolev-class  section of the pull-back of the tangent bundle. 
To define the class of sections, 
let $\beta$ be a {\em  good bump function} on $\R$ such that
\begin{enumerate}
    \item $\beta$ is non-decreasing;
    \item $\beta(s)=1$ when $s\geq 1$;
    \item $\beta(s)=0$ when $s\leq 0$.
\end{enumerate}
For any {\em  Sobolev decay constant} $\delta>0$ we define a {\em  Sobolev weight function} $\kappa^{\delta}$ on $S$ such that
\begin{enumerate}
    \item $\kappa^{\delta}=0$ outside cylindrical and strip-like ends; and
    \item on each cylindrical or strip-like end, there is a constant $M>0$ with $\kappa^{\delta}(s,t)= \delta \beta(\lvert s\rvert- M)$.
\end{enumerate}
Pick a Riemannian metric $g$ on $X$ which is cylindrical near infinity and a metric connection $\nabla$. We further assume that $L$ is totally geodesic with respect to $g$. Equip $S$ with a 
metric which is of standard form on the cylindrical and strip-like ends.

\begin{definition}
For Sobolev constants $k, p$ with $kp>2$ and a model map $u\in \Map_0(S, X, L)$, define an {\em  enlarged weighted Sobolev space
 } limiting to Reeb chords and Reeb orbits $\gamma_e, e \in \cE(S)$ as follows:   A {\em  model section } of $u^* TX$ is a section 
$\xi: S \to u^* TX$ taking values in $\partial u^* TL$ on the boundary 
so that 
\begin{itemize}
    \item in a neighborhood of each cylindrical end $e$, $\xi(s,t) = \cT_{s,t} \xi(e)$ is given by 
parallel transport $\cT_{s,t}$ of an element of $TL$ along the Reeb chord $\gamma_e$ denoted 
$\xi(e) \in  T_{\gamma_e(0)} \Lambda \oplus \R$; 
\item near an interior puncture, $\xi(s,t) = \cT_{s,t} \xi(e)$ is given by 
parallel transport $\cT_{s,t}$ of an element of $TX$ along the Reeb orbit $\gamma_e$ denoted 
$\xi(e) \in  T_{\gamma_e(0)} Z \oplus \R$.
\end{itemize}
For such 
a model section, define a compactly supported section as follows.  Let $\rho_e: S \to \R$ be a cutoff function equal to $1$ in a neighborhood of the end $e$, and vanishing on all other ends. 
Given 
$\xi$ with asymptotic values $\xi(e)$, the difference 
\[ \xi_0 := \xi - \sum_{e \in \cE(S)} \rho_e \xi(e) \]
is compactly supported.    Define the Sobolev $(k,p,\lambda)$-norm as
\[
\lvert \lvert \xi \rvert \rvert^p_{k,p,\lambda}:= \sum_{e\in \mathcal{E}(S)} \lvert \xi(e) \rvert^p + \int_{S} (\lvert \xi_0 \rvert^p + \lvert \nabla \xi_0 \rvert^p+ \cdots + \lvert \nabla^{k} \xi_0 \rvert^p)e^{p\kappa^{\lambda}}
   . \]
 The completion of the space of model sections in this norm is denoted  \label{rep:addboth}
\[
\Omega^0(u^*TX, \partial u^*TL)_{k,p,\lambda}' :=\{ \xi \mid \lvert \lvert \xi \rvert \rvert_{k,p,\lambda}< +\infty, \xi |_{\partial S} \in TL \}.
\]
Let $ \exp: T X \to X $
denote the geodesic exponentiation map defined using the metric  $g$ on $X$.
Define the {\em  weighted Sobolev space of maps}
\begin{eqnarray*}
\Map_{k,p,\lambda}(S, X, L) & :=& \{ \exp_u{\xi}\mid u\in \Map_{0}(S, X, L), \xi\in \Omega^0(u^*TX, \partial u^*TL)'_{k,p,\lambda}\} \\ 
&=  & \{ \exp_u{\xi}\mid u\in \Map_{0}(S, X, L), \xi\in \Omega^0(u^*TX, \partial u^*TL)_{k,p,\lambda}\}
\end{eqnarray*}
to be those maps which differ from a model map by the exponential of a section in the given 
Sobolev space (or, the ``primed space'' which adds in constants at infinity).  This ends the Definition. 
\end{definition}

There is a similar discussion for Sobolev spaces of maps from broken intervals.  Let $I$ be a broken interval with subintervals $I_j$ as in Definition \ref{def:brokeninterval}. For each sub-interval,  define a space of model maps from $I_j$ to $L$ that are constant near infinity. 
The space $\Map_{k,p}(I_j, L)$ is the space of maps obtained by geodesic exponentiation of a $W^{k,p}$ section of $u^* TL$ from a model map.   The tangent space  is denoted $\Omega^0(T, (u|_{T_L})^* T {L})_{k,p}$, by which we mean 
$W^{k,p}$ maps on each subinterval whose limits match at any breaking.
By definition, such maps have limits along the infinite ends of
any sub-interval $I_j$. \label{rep:treemiss} Let $\Map_{k,p}(I,  L)$ be the subset 
of the product of spaces $\Map_{k,p}(I_j, L)$
so that the limit  along the positive end of the $j$-th interval is equal to that along the negative end of the 
$j+1$-th interval.  

We now define similar Sobolev completion of maps from a treed disk. 
Let $C=S\cup T$ be a treed disk, where $S$ is the surface part and $T$ is the edge part. Let $\Map_{k,p,\lambda}(C, X, L)$ 
denote the subset of the product of maps from the surface part and tree part that agree on the intersection: 
\begin{multline*}
\Map_{k,p,\lambda}(C, X, L)
:= \{ (u_S,u_T) : u_S |_{S \cap T} = u_T |_{S \cap T} \}\\
\subset \Map_{k,p,\lambda}(S, X, L) \times 
\Map_{k,p}(T,  L) .
\end{multline*}

We claim that the moduli space  is cut out locally over the moduli space of domains by a Fredholm map.
Recall from Section \ref{sec:treedbuildings} that $\M_\Gamma$ is the moduli space of treed disks $C = S \cup T$ of type $\Gamma$. Let 
\[ \M_\Gamma^i, i = 1,\ldots, m \] 
be an open cover of $\M_\Gamma$
over which the universal treed disk 
\[ \U_\Gamma = \{ (C, z \in C) \} \to \M_\Gamma \]
is trivial as a stratified bundle \label{rep:stratbundle} in the following sense:
There exists a homeomorphism 
\[  \U_{\Gamma,i} := \U_\Gamma | \M_\Gamma^i \to 
\M_\Gamma^i \times C \] 
which induces a homeomorphism from  any   treed disk $C' \in \M_\Gamma^i$ with $C$ that is 
smooth on each surface component $S_v \subset C$ and on each edge $T_e \subset C$.
Thus we may view $C'$ as the smooth space $C$ equipped with 
an almost complex structure on the surface part $j(C') \in \J(S)$ and a metric $g(C') \in \G(T)$
depending on $C'$.  Define 
\[ \B_\Gamma^{i,k,p,\lambda} := \M_\Gamma^i \times \Map_{k,p,\lambda}(C,X,L). \]
A chart near a pair $(C',u')$ is given by the product of a chart for $\M_\Gamma^i$ near $C'$
and a neighborhood of $0$ under the map constructed as follows:   Suppose for simplicity that 
$S$ has a single strip-like end and no cylindrical ends.   Define a map 
\begin{multline} \Omega^{0}(S, ( {u}|_S)^* T {X})_{k,p,\lambda}'  \oplus \Omega^0(T_L, (u|_{T_L})^* TL)_{k,p}  
\oplus \Omega^0(T_{\white,-}, (u|_{T_{\white,-}})^* \cR(\Lambda_-))_{k,p}  \\
\oplus \Omega^0(T_{\white,+}, (u|_{T_{\white,+}})^* \cR(\Lambda_+) )_{k,p}
\oplus \Omega^0(T_{D}, (u|_{T_D})^* TX )_{k,p}
\longrightarrow  \Map_{k,p,\lambda}(C,X,L) \end{multline}
by geodesic exponentiation, for some metric for which $L$ is totally geodesic and we use these maps as charts to give a Banach manifold structure on $\Map_{k,p,\lambda}$. The fiber of the vector bundle $\cE_\Gamma^{i,k-1,p,\lambda}$ over any map $u$ is the bundle of one-forms
of Sobolev differentiability class one less:
\begin{multline} \label{oneforms} 
\cE_{\Gamma,u}^{k-1,p,\lambda} := \Omega^{0,1}(S, ( {u}|_S)^* T {X})_{k-1,p,\lambda}
\oplus \Omega^1(T_L, (u|_{T_L})^* T {L})_{k-1,p,\lambda} 
\\ \oplus \Omega^1(T_{\white,-}, (u|_{T_{\white,-}})^* \cR(\Lambda_-) )_{k-1,p} 
\oplus \Omega^1(T_{\white,+}, (u|_{T_{\white,+}})^* \cR(\Lambda_+) )_{k-1,p} \\
\oplus \Omega^1(T_{D}, (u|_{T_{D}})^* TX )_{k-1,p}
\end{multline} 
and in the case of broken edges, we mean the space of one-forms on the complement of the breakings. We remark that in general, the vector spaces $\cE_{\Gamma,u}^{k-1,p,\lambda}$ do not form a smooth Banach 
vector bundle over $\B_\Gamma^{k,p,\lambda} $ for reasons \label{rep:forreasons} explained in \cite[Remark 4.2]{cw:traj}:  the transition maps between local trivializations
identifying the curves with, say, fixed curves $C_1, C_2$ 
involve identifications
\[ \M_\Gamma^i  \times 
\Omega^{0,1}(S_1, ( {u}|_{S_1})^* T {X})_{k-1,p,\lambda} \to 
\M_\Gamma^i \times \Omega^{0,1}(S_2, ( {u}|_{S_2})^* T {X})_{k-1,p,\lambda}  \] 
induced from varying diffeomorphisms $S_1 \to S_2$ depending on the curve $[C'] \in \M_\Gamma^i$. 
Differentiating with respect to $[C']$ induces a map to a Sobolev class of one lower, and so the transition maps are not differentiable.   However, the union of  $\cE_{\Gamma,u}^{k-1,p,\lambda}$
for maps $u'$ in $\B_\Gamma^{i,k,p,\lambda}$ is a smooth Banach vector bundle  with local trivialization 
induced by parallel transport. 

In order to achieve regularity, we assume that some of the edges
map to a Donaldson hypersurface.  Let 
$\ol{D} \subset \ol{X}$ be a Donaldson hypersurface 
so that $\ol{L}$ is exact in the complement of $\ol D.$ Recall from Corollary \ref{corr:existenceofD} that when $\ol{L}$ is rational, such a hypersurface exists.  We define $D$ to be the restriction of $\ol{D}$ to the open manifold $X$. As in Cieliebak-Mohnke \cite{cm:trans}:  each component of $u^{-1}(D)$ contains an edge $T_e, e \in \Edge_D(\Gamma)$.

The moduli spaces of treed disks are cut out locally by 
Fredholm sections:  Define  for each $i$ a section
\begin{multline}
  \F_\Gamma^i: \B_\Gamma^{i,k,p,\lambda} \to \cE_{\Gamma}^{i,k-1,p,\lambda}
   \times X^{\# \Edge_D(\Gamma)},  \\
  \quad ([C'],u') \mapsto
  \left(\olp_{j(C')} u'|S , \  
    \cct u' |_{T} +  \grad(\zeta(u'), 
  (\ev_e(u))_{e \in \Edge_D(\Gamma)}
  \right)\end{multline}
where $\zeta$ denotes either $\zeta_\white$, $\zeta_{\black,\Lambda_\pm}, \zeta_{L}$ or $\zeta_D$ depending on the type of edge.  
We drop the Sobolev constants $k,p,\lambda$ to simplify notation.
Embed
\[ 
\ul{0}: \ \B^{i}_\Gamma \times D^{\# \Edge_D(\Gamma)} \to 
\E^{i}_\Gamma
    \times X^{\# \Edge_D(\Gamma)} \] 
by the product of the zero section and 
a number of copies of the Donaldson hypersurface $D$.
The moduli space of maps over $\M_\Gamma^i$ is 
\[ \M_\Gamma^i(L) =  (\F_\Gamma^i)^{-1}(\ul{0}) ,\]
that is, the inverse image of the zero section 
times a number of copies of the Donaldson hypersurface.  The global moduli space
$\M_{\Gamma}(L)$ is obtained by patching together the local spaces $\M_{\Gamma}^i(L)$.  

We introduce the following terminology for linearized operators and regularity. 

\begin{definition}  The {\em  linearized operator} at any solution $u$ in some trivialization of the universal bundle is denoted
\begin{multline} \label{linop} \ti{D}_u: T_{(C,u)} \B^i_\Gamma \to \cE^i_{\Gamma,u} \times (TX/TD)^{\# \Edge_D(\Gamma)},  \quad \xi \mapsto \ccrho |_{\rho = 0} \cF_\Gamma^i(\exp_u(\rho \xi)) .\end{multline}
The definition is independent of the choice of local trivialization
of the universal bundle used.

\vskip .1in \noindent A map $u$ is {\em  regular} if
$\ti{D}_u$ is surjective, in which case $\M_\bGamma(L)$ is a
smooth manifold of dimension $ \Ind(\ti{D}_u)$ in an open neighborhood of
$u$.   

\vskip .1in \noindent The {\em  virtual dimension} of the moduli space $\M_\bGamma(L)$ at $u$ 
is the Fredholm index  of the linearized operator
\[ \vdim(\M_\bGamma(L)) := \Ind(\ti{D}_u) \] 
which is the honest dimension if every element is regular.  This ends the Definition.
\end{definition} 

The linearized operator is Fredholm, as for example shown in Wendl \cite[Lecture 4]{wendl:sft}. The restriction of $\ti{D}_u$ to variations of the map $\xi \in \Omega^0(S, u^*TX)$ is the linearized operator $D_u$ studied in McDuff-Salamon \cite{ms:jh} for holomorphic maps to closed symplectic manifolds.

The moduli space of treed maps is naturally a union over combinatorial 
type.  
For any integer $d \ge 0$, denote
\begin{equation} \label{dlocus}
\ol{\M}(L)_d = \bigcup_{\bGamma} \M_{\bGamma}(L) \end{equation}
the union of types with 
\begin{equation}
\label{rigidmap}
\vdim \M_\bGamma(L) + \codim(\M_\Gamma) = d . \end{equation}
That is,  $\M_\bGamma(L)$ is in the locus of expected dimension $d$,
allowing for deformations of domain which change the combinatorial type.  

The locus of regular maps has good smoothness properties.
 Denote by $\M^{\reg,i}_\bGamma(L) \subset \M_\bGamma^i(L)$  the locus of regular maps. 
By the implicit function theorem for smooth maps between Banach manifolds,  the locus of regular maps $\M^{\reg,i}_\bGamma(L)$ is a smooth manifold of expected
dimension.   The transition maps 
between the various moduli spaces 
\[ \M^{\reg,i}_\bGamma(L), \M_\bGamma^{\reg,i'}(L),  \quad \forall i,i' \ \text{so} \ \M_\Gamma^i \cap \M_\Gamma^{i'}\neq\emptyset \]
are
smooth  as all derivatives exist by elliptic regularity. 
\label{rep:transsmooth}

We will define the Chekanov-Eliashberg differentials and cobordism maps by counting rigid maps, defined as follows:  

\begin{definition} 
\begin{enumerate} 
\item For a pair $(Z,\Lambda)$, the locus of {\em  rigid} configurations ${\M}(\Lambda)_0$ of holomorphic maps to the symplectization $(\R \times Z, \R \times \Lambda)$ is the locus of maps
$u: C \to \R \times Z$ for which the expected dimension $\on{Ind}(\ti{D}_u) = 1$ and $\M_\Gamma$ represents a stratum of top dimension in the moduli space of domains.
%
\item  For a cobordism $(X,L)$ between $(Z_-,\Lambda_-)$  $(Z_+,\Lambda_+)$
the  locus of {\em  rigid} configurations ${\M}(L)_0$ is the locus of maps
$u: C \to X$ for which the expected dimension $\on{Ind}(\ti{D}_u)$ vanishes and $\M_\Gamma$ represents a stratum of top dimension in the moduli space of domains.
\end{enumerate}
\end{definition} 

The last condition in each item is equivalent to requiring that the lengths $\ell(e)$ of all edges $e \in \Edge(\Gamma)$ are finite and non-zero.    Note that ${\M}(\Lambda)_0$ lies in the locus in $\M(\Lambda)$ that is expected dimension zero, since equivalence is up to translation.

\subsection{Transversality} 

Now we regularize the moduli spaces as constructed in the previous section.   The domains of stable
treed buildings are automatically stable,  except for trivial cylinders:  

 \begin{lemma}  Let $(X,L)$ be a cobordism with concave end $(Z_-,\Lambda_-)$ and convex end $(Z_+, \Lambda_+)$ and $\ol{D} \subset \ol{X}$
 a Donaldson hypersurface satisfying the inequality of 
 Proposition \ref{prop:je}, so that $\ol{L}$ is exact in the complement.  
 Let $u = (u_v), v \in \Ver(\Gamma)$ be a holomorphic building in $X$.  
\begin{enumerate} 
\item If $u_v: C_v \to X$ is a non-constant sphere resp. disk component 
then $u_v$ is stable, that is, $v$ has at least three adjacent edges resp. three adjacent open edges or one adjacent open edge and one adjacent closed edge; 
\item If $u_v: C_v \to \R \times Y_\pm$ is a sphere resp. disk component that is not a trivial cylinder or strip then $u_v$ is stable, that is, $v$ has at least three adjacent edges resp. three adjacent open edges or one adjacent open edge and one adjacent closed edge; 
\end{enumerate} 
\end{lemma}

\begin{proof}  The claim for sphere components $u_v: C_v \to X$ follows from the condition in
    Proposition \ref{prop:je} which guarantees at least three intersection points with the Donaldson hypersurface, and the condition in 
    Definition \ref{def:treedbuildings} that each such intersection corresponds to an edge.  The inequality for disks in $X$ follows from the 
   exactness of $\Pi$ in $Y - D$.   The justification for components mapping to $\R \times Z_\pm$ is similar. 
\end{proof}

Domain-dependent perturbations are defined as maps from 
the universal curve.  Suppose that $\Gamma$ is a stable type
of domain and denote by $\M_\Gamma$ the moduli space of treed disks of type
$\Gamma$. Over $\M_\Gamma$ we have a universal curve $\cU_\Gamma$
which decomposes into tree and surface parts
\[ \cU_\Gamma  = \{ [ C,z], z \in C \} , \quad \U_\Gamma = \cS_\Gamma \cup \cT_\Gamma \]
depending on whether the point $z$ lies in the surface $S$ or tree $T$
part of $C$.

\begin{definition} \label{def:ddp}
Let $(X,L)$ be a decorated cobordism as in Definition \ref{def:cylnearinf}. 
A {\em  domain-dependent perturbation} is a
  collection of maps
\[ \begin{array}{lllll} 
 J_\Gamma &:& \cS_\Gamma &\to & \J_{\cyl}(X) \\ 
 \zeta_{L,\Gamma} &:& \cT_{L,\Gamma}&  \to &  \Vect(L) \\ 
 \zeta_{\white,\pm,\Gamma} &:& \cT_{\white,\pm, \Gamma} & \to & 
  \Vect({\cR}(\Lambda_\pm)) 
\end{array} \]  
consist of a domain-dependent cylindrical almost complex
structure $J_\Gamma$ and domain-dependent vector fields
$\cT_{L,\Gamma}$ and $\cT_{\diam,\Gamma}$
so that the
corresponding maps 
\begin{eqnarray*}
\cS_\Gamma \times TX &\to& TX \\
  \cT_{L,\Gamma} \times L
&\to & TL \\ 
\cT_{\white,\pm,\Gamma} \times \cR(\Lambda_\pm) &\to & T \cR(\Lambda_\pm)
 \end{eqnarray*} 
are smooth. The vector fields $\zeta_{\pm,\Gamma}$ are of the form in Definition \ref{def:cylnearinf} on each edge depending on the label of the edge.  Later, we will assume
that the perturbations vanish on an open subset containing the zeros of the vector fields
and the intersections $\cS_\Gamma \cap \cT_\Gamma \subset \cU_\Gamma$.

One may also allow Hamiltonian perturbations.  These are not necessary for regularization of the moduli spaces in the case of Legendrians with embedded projection but is necessary for those with immersed Lagrangian projections. The use of Hamiltonian perturbations is also useful for our later localization computations.
Let 
\[ \Vect_h(X,U) \subset \Vect(X) \]
denote the space of Hamiltonian vector fields vanishing
on an open subset $U$ (which later will be taken to include the
Donaldson hypersurface used for regularization).  
 A domain-dependent
Hamiltonian perturbation is a one-form
\[ H_\Gamma \in \Omega^1(\cS_\Gamma, \Vect_h(X,U)) .\]
A {\em  domain-dependent perturbation} is a tuple
\[ P_\Gamma = (J_\Gamma,\zeta_{L,\Gamma}, \zeta_{\white,-,\Gamma},
\zeta_{\white,+,\Gamma}, H_\Gamma) \] 
as above. 
\end{definition}

For a domain-dependent perturbation $P_\Gamma$, the \textit{moduli space of treed broken disks} of combinatorial type $\Gamma$ is defined as 

\[ \cM_{\Gamma,P_\Gamma}(L) = \left\{ 
\begin{array}{l|l}
 & [C]\in \cM_\Gamma,u:\cU_\Gamma|_{[C]} \to X\\ (u,[C])& z\in \cU_\Gamma|_{[C]_e}, t\in \cU_\Gamma|_{[C]_e} \text{ where } v\in V(\Gamma), e\in E(\Gamma)  \\
 & du_v(z) + \cJ_\Gamma(z) \circ du_v \circ j_C + H_\Gamma(u_v(z)) = 0\\ & du_e(t) + \zeta_*(t) = 0 
\end{array}
   \right\}. \]
To keep the notation simpler we chose to denote all the trajectories on the edges by only one equation above.

Our perturbations of the almost complex structure are required to be small in the following sense.

\begin{definition} Define a (discontinuous) function recording the number of interior edges
meeting the surface component containing the given point:
\[ \label{stab} n_D:  \cS_\Gamma \to \Z_{\ge 0 } ,  \quad z \mapsto   
\max_{ z \in S_v} \# \{ e \in \Edge_D(\Gamma), T_e \cap S_v \neq \emptyset \}  .\] 
Suppose
that the Donaldson hypersurface  has degree $k \in \Z_{> 0}$ in the sense that  
\[ [\ol{D}]^\dual = k [\ol{\omega}] \in H^2(\ol{X})  \].
A domain-dependent almost 
complex structure $J_\Gamma: \cS_\Gamma \to \J(X) $ is {\em  stabilizing} if
it takes values in the set for which there are no holomorphic spheres contained in the Donaldson divisor, 
that is, 
\[ J_\Gamma(z) \in \J^{n_D/k}(X) , \quad \forall z \in \cS_\Gamma \]
with notation from Proposition \ref{prop:je}.
\end{definition}

The set  of perturbations satisfying the stabilizing condition contains a non-empty open subset of domain-dependent almost complex structures, by definition.

We apply a Sard-Smale argument to show that the moduli space of
expected dimension at most one is cut out transversally for generic perturbations.
For this, it is convenient to adopt a space of perturbations introduced
by Floer.   
\begin{definition} Let $  \ul{\eps} = (\eps_0,\eps_1,\eps_2,\ldots)$
be a sequence of positive real numbers.   For a smooth function $f: \R \to \R$ define the $\ul{\eps}$-norm 
\[ \Vert f \Vert_{\ul{\eps}} = \sum_{j \ge 0} \eps_j \Vert f \Vert_{C^j} \] 
denote 
the (possibly infinite) {\em  Floer norm} of $f$, as in \cite[Appendix B]{wendl:sft}.
\noindent For sequences $\ul{\eps}$ which converge to  zero sufficiently fast,  the space $C^{\ul{\eps}}$ of functions with finite Floer norm contains
bump functions supported in arbitrarily small neighborhoods of any point.  
Similar spaces are defined for functions on any manifolds, sections of vector
bundles and so on. 
Let 
\[ \cP^{\ul{\eps}}_\Gamma = \{ P_\Gamma \in C^{\ul{\eps}} \} \]
denote the space of perturbations for type
$\Gamma$ of finite $\ul{\eps}$-norm.  Let 
\[ \B^{\univ}_{k,p,\ul{\eps},\lambda} = \M_\Gamma \times
\Map(C,X,L)_{k,p,\lambda} \times \cP^{\ul{\eps}} \]
be the space of triples $(C,u,P)$ where $P = (J,\zeta)$ is a
perturbation, and the subscript $k,p,\lambda$ indicates that the
restriction of $u$ to each component $T_e, S_v$ is in the weighted Sobolev class
$W^{k,p,\lambda}$.  Denote by
\begin{eqnarray*}
 \E^{\univ,i} :=  \E^{\univ,i}_{\Gamma,k-1,p,\ul{\eps},\lambda} &:=& \bigcup_{(C,u,P) \in \B^{\univ,i}_{\Gamma,k,p,\ul{\eps},\lambda}}
  \Omega^{0,1}(C, (u|S)^*
  TX)_{k-1,p,\lambda} \\
&:=& \bigcup_{(C,u,P) \in \B^{\univ}} \Omega^{0,1}(S, u^*
     TX)_{k-1,p,\lambda} 
\oplus \Omega^1( T_L, (u|T_L)^* T L )_{k-1,p} \\
&& \oplus \Omega^1( T_{\white,-}, (u|T_{\white,-})^* T \cR(\Lambda_-) )_{k-1,p} \\
&& \oplus \Omega^1( T_{\white,+}, (u|T_{\white,+})^* T \cR(\Lambda_+) )_{k-1,p} \\
&& \oplus \Omega^1( T_D, (u|T_D)^* T X )_{k-1,p}.
\end{eqnarray*}
Define a section 
\begin{multline}
    \cF_\Gamma^{\univ,i}: \B^{\univ,i}_\Gamma\to \E^{\univ,i}_\Gamma
    \times X^{\# \Edge_D(\Gamma)}
    ,   \\
    (C', u , P_\Gamma) \mapsto \left( \olp_{j(C'),J_\Gamma,H_\Gamma}
u_S , \left( \ccs + \zeta_{C',\Gamma,\cdot}( u_{T_\cdot} )\right)
,  (\ev_e(u))_{e \in \Edge_D(\Gamma)} \right) \end{multline}
where $\olp_{j(C'), J_\Gamma}$ is the Cauchy-Riemann operator defined by 
the choice of curve $C'$ in the domain, and $\zeta_{C', \Gamma,\cdot} $
is the perturbed vector field determined by the choice of domain $C'$
and $\cdot$ ranges over types of edges. 
Embed
\[ 
\ul{0}: \ \B^{\univ,i}_\Gamma \times D^{\# \Edge_D(\Gamma)} \to
\E^{\univ,i}_\Gamma
    \times X^{\# \Edge_D(\Gamma)} \] 
by the product of the zero section and 
a number of copies of the Donaldson hypersurface $D$.  This ends the Definition. \end{definition}

\begin{theorem} \label{transverse}
  The section $\F^{\univ,i}_\Gamma$ is  transverse to the zero section $\ul{0}$ at a map  $u: C \to X$ as long as there are no constant spheres $S_v, v \in \Ver(\Gamma)$ with more than one edge $e \in \Edge_D(\Gamma)$ attached.
\end{theorem} 

\begin{proof}
We treat transversality for constant, non-constant, and trivial strip components  separately.  Let $\eta$ be an element of the cokernel
$\coker(\ti{D}_u)$ of the linearized operator of \eqref{linop}.  We write the
  restriction of $\eta$ to the sphere or disk components
  resp. edges as 
  \[ \eta_v \in \Omega^{0,1}(S_v, u_v^* TX), v \in \Ver(\Gamma) \] 
  and
  \[ \eta_e \in \Omega^1(T_e, u_e^* TL), \ 
  e \in \Edge_L(\Gamma) , \quad 
  \eta_e \in \Omega^1(T_e, u_e^* T\Lambda), \ 
  e \in \Edge_\white(\Gamma)  . \]  
  We wish to show that all of the components $\eta_v$ and $\eta_e$ vanish. 

\vskip .1in \noindent {\em Step 1:  We show that the one-form vanishes on surface components   on which the map is non-constant.}   Suppose
$v \in \Ver(\Gamma)$ with map $u_v: = u | S_v$ mapping to $X$ such
that $u_v$ is non-constant.  Then $D_{u_v}^* \eta_v = 0 $ 
weakly and $\eta_v$
is perpendicular to variations $ K_\Gamma \in T_{J_\Gamma} \J_\Gamma(X) $ which are
domain-dependent.  By \cite[Lemma 3.2.2]{ms:jh}, $\eta_v$
vanishes in an open neighborhood of any point $z \in S_v$ on which $\d u_v $ is
non-zero.  On the other hand, we may identify $\eta_v$ with an element in the kernel $\ker(\ti{D}_{u_v}^*)$ of the adjoint, and so 
$\eta_v$ vanishes by the unique continuation principle in 
\cite[Section 2.3]{ms:jh}.  (This holds even for $\eta$ in the cokernel of the restriction of the linearized operator to the space of sections $\xi$
are constrained to vanish on the intersections $S \cap T$, as in 
\cite[Section 3.4]{ms:jh}.)

\vskip .1in \noindent {\em Step 2:  We show that the one-form vanishes on a surface component mapping to the neck on which 
the projection to the base is non-constant.  }   Suppose $v \in \Ver(\Gamma)$ with map $u_v: = u|  S_v$ such that  
\[ \ol{u}_v:= {p} \circ u_v: S_v \to Y \]
is non-constant.  The restriction $\eta_v$ is a weak solution to $D_{u_v}^* \eta_v  = 0 $  
and $\eta_v$ is perpendicular to one-forms produced by variations $ K_\Gamma$ 
in $J_\Gamma$.  Such variations satisfy $J_\Gamma K_\Gamma = - K_\Gamma J_\Gamma$
and vanish on the first factor in the decomposition 
\[ u^*TX \cong u^* \ker (D {p}) \oplus \ol{u}^* T Y .\]   
Suppose $\eta_v(z) \neq 0$.  Choose a $(J_\Gamma,j)$-antilinear map from $T_z S$ to $T_{u(z)} X$
pairing non-trivially with $\eta_v(z)$, which we may take
to be of the form $K_\Gamma (z) \d_z u j(z)$ for some $K_\Gamma(z)$ since the second component of $\d_z u$
is non-vanishing.  Using a cutoff function, we obtain a one-form 
$K_\Gamma \d u j$ that pairs non-trivially with $\eta_v$, which is a contradiction.  Thus $\eta_v$
vanishes.

\vskip .1in \noindent {\em Step 3:    The one-form vanishes on trivial cylinders an strips. }   The case of components mapping to fibers in the neck piece, 
that is, to the images of Reeb chords or orbits, requires special consideration and is typically the hardest part of the transversality argument in symplectic field theory.  Thus for the purposes of this paragraph we assume  $X = \R \times Z$ and is equipped with a projection $p: X \to Y$.  Any disk $u_v: S_v \to X$ contained in a
  fiber 
\[ {p}^{-1}(y) \cong \C^\times, y \in Y \]  is by assumption a punctured disk bounding 
  \[ L \cap {p}^{-1}(y) \cong
  e^{2\pi i\theta_1} \R_+  \cup 
  \ldots \cup e^{2\pi i \theta_k} \R_+  .\] 
Here the number of punctures is at least two, so that if the number of punctures is equal to two then the component $u_v$ is a trivial cylinder.   The linearized operator $D_{u_v}$ is a direct sum 
\[   D_{u_v} \cong D^{\bv}_{u_v} \oplus D^{\bh}_{u_v} \] 
of a trivial horizontal part $D^{\bh}_{u_v} \cong \olp$ taking values  in $T_{\ol{u}_v} Y$ and a vertical part
$D^{\bv}_{u_v}$ taking values in $\C$ which has boundary 
  condition $\R$.  We examine 
  the kernel and the cokernel of each part separately.   The kernel 
  of $D_{u_v}^{\bv} \cong \olp$ on $S_v $ consists of bounded holomorphic functions $\xi: S_v \to \C $
  with real boundary values.  Therefore $\ker(D_{u_v}^{\bv} )$ is the one-dimensional real vector space space of constant functions with values in the reals.     On the other hand, the cokernel 
   $\coker(D_{u_v}^{\bv} )$
   may be identified with holomorphic 
  $(0,1)$-forms $\eta$ of Sobolev class $-k,p^*,\lambda$ enforcing exponential decay (since constant functions near infinity lie in the domain).  In the local coordinate on $\ol{S}_v$ near each puncture such a holomorphic $(0,1)$-form is vanishing at the puncture.  Since there are no 
  such holomorphic $(0,1)$-forms $\eta$ on a disk $S_v$, the cokernel of $D_{u_v}^{\bv}$ vanishes.  
  We remark that the kernel and cokernel of the operator on the punctured 
  disk may be identified with the kernel and cokernel of the operator on the 
  disk by removal of singularities as in Abouzaid \cite[Lemma 4.7]{ab:ex}.  By Oh \cite{oh:rh}, any rank one Riemann-Hilbert problem on $\ol{S}$ has either vanishing kernel or vanishing cokernel.
The horizontal part of the linearized operator on $S_v$ is surjective for similar reasons.

\vskip .1in \noindent {\em Step 4:    The one-form vanishes on both types of edges.}  We apply the fundamental theorem of ordinary differential equations.   For any edge $e \in \Edge_{\white}(\Gamma)$ the cokernel of the operator $D_{u_e}$ on zero-forms on $T_e$ is identified, via Hodge star,  with the space of solutions to the adjoint 
  equation 
\[ D_{u_e}^* \xi_e = \nabla \xi_e +  \nabla \grad(\zeta(u_e)) \d
  s  \]
  with vanishing at the endpoints; there are no such solutions. 
  In the case without endpoints $T_e \cong \R$ that may occur in a broken trajectory, transversality follows from the 
  Morse-Smale assumption on the Morse function $f_\white$.     On the other hand, $D_{u_e}$ has kernel isomorphic to
  $T_{u_e(z)} L$ for any point $z \in T_e$, via evaluation.

\vskip .1in \noindent {\em Step 5:   The matching conditions are transversally cut out.  }  Consider a tree $\Gamma_0 \subset \Gamma$ such that for each 
  vertex $v \in \Ver(\Gamma_0)$,
  the component $u_v := u |S_v$
  is vertical.  The argument of the previous section shows that such configurations are transversally cut out if each $S_v$ is a disk or sphere and $\Gamma_0$ is a tree:  Let $u_0$ denote the restriction of $u$ to $S_0$
  and $\ker(D_{u_0})'$ the kernel of the linearized operator acting on
  sections without matching conditions enforced at the nodes of $S_0$.   By the previous paragraph $\ker(D_{u_0})'$ is a sum of factors of $\C \oplus T_{p \circ u(C)} Y$ for each 
  sphere component and $\R \oplus T_{p \circ u(C)} \Pi$ for each
  disk component.  Transversality is equivalent to the condition that the map 
  \[ \ker(D_{u'})' \to 
  \bigoplus_{e \in \Edge_{L,\fin}(\Gamma)} 
  T_{p \circ u(C)} Y \oplus 
  \bigoplus_{e \in \Edge_{\white,\fin}(\Gamma)} 
  T_{p \circ u(C)} \Pi
  \] 
  is surjective.  This follows from an induction  (for each connected component of the subset $S_0$ of the domain $S$ on which the horizontal projection of the map is constant). For each such component, the induction 
  starts with vertices $v \in \Ver(\Gamma_0)$ contained in terminal edges of $\Gamma_0$, so that the corresponding disk $S_v$ has a single node $w(e) \in S_v$ represented by an edge $e$ in $\Gamma_0$.
  Surjectivity of the linearized evaluation map for this component $S_v$ implies that the matching condition 
  at $w(e)$ is cut out transversally; the 
  inductive hypothesis then implies that the remaining matching conditions are transversally cut out.

  Putting everything
  together, we obtain that the operator $D_u'$ whose domain consists
  of sections without matching conditions is surjective, and  the
  kernel of $D'_{u_v}$ surjects onto any intersection point $TL |_{S_v \cap
  T}$ or
  $T\Pi |_{S_v \cap
  T}$.   Similarly the kernel of $D'_{u_e}$ surjects onto any
  intersection point $TL |_{T_e \cap S}$.   It follows that the matching
  conditions
  at the nodes are cut out transversally.    
  \end{proof}

Domain-dependent perturbations help us obtain transversality of moduli spaces by breaking the symmetry of the auxiliary choices such as the almost complex structure and Morse functions. Although beneficial for obtaining transversality, it comes with a cost:  we lose expected symmetry such as the divisor axiom for curve counts. To overcome this difficulty we use multi-valued domain dependent perturbation which allows us to restore part of the expected symmetry back while still preserving transversality.

\begin{definition}

A {\em  multi-valued domain-dependent perturbation} is a formal sum $P_\Gamma$ of domain-dependent perturbations $\ti{P}_{\Gamma,i}$ with real coefficients
\begin{equation} \label{mval} P_\Gamma := \sum_{i=1}^{k_\Gamma} c_i \ti{P}_{\Gamma,i} \end{equation}
where 
\[ \sum_{i=1}^{k_\Gamma} c_i = 1, \quad c_i > 0 , \forall i . \] 
%
We call the terms $\ti{P}_{\Gamma,i}$ {\em  sheets} of the multi-valued perturbation and the terms $c_i$ are the {\em  weights} of the corresponding sheets.   Two multi-valued perturbations are considered equivalent
if they are related by replacements of the form 
\begin{equation}\label{addsheets}
     c_i \ti{P}_{\Gamma,i} + c_j \ti{P}_{\Gamma,j} = (c_i + c_j) \ti{P}_{\Gamma,i}
= (c_i + c_j) \ti{P}_{\Gamma,j}
\end{equation}
whenever $\ti{P}_{\Gamma,i} = \ti{P}_{\Gamma,j}$.  The number of sheets is the minimum of the number
$k_\Gamma$ over all representatives, that is, the number of distinct $\ti{P}_{\Gamma,j}$.  From now on the word {\em perturbation}
will mean a multi-valued domain-dependent perturbation.
\end{definition}

We will choose perturbations via an inductive process based on the combinatorial type of the domain. We will require that the perturbations for each type agree with those 
already chosen on the boundary.  The following
morphisms of graphs $\Gamma_1 \to \Gamma_2$ correspond to inclusions
of $ \ol{\M}_{\Gamma_1}^{< E}(\Lambda)$ in the formal boundary of
$ \ol{\M}_{\Gamma_2}^{< E}(\Lambda)$.
\begin{definition} An {\em  elementary morphism} of domain types
  $\Gamma_1 \to \Gamma_2$ is one of the following.
\begin{enumerate} 
\item {\rm (Collapsing or breaking edges)}  An isomorphism $\Gamma_1 \to \Gamma_2$ where $\Gamma_2$ is
  obtained from $\Gamma_1$ by removing an edge $e$ from a set
  $\Edge_0(\Gamma_1)$ or $\Edge_\infty(\Gamma_1)$  and placing it by an edge with finite length in $\Edge(\Gamma_2)$; corresponding to a degeneration of $C$ in which the length $\ell(e)$  of
  $e \in \Edge(\Gamma_2)$ becomes infinite or zero. 
\item {\rm (Cutting edges)}  A morphism $\Gamma_1 \to \Gamma_2$ where $\Gamma_2$ is
  obtained from $\Gamma_1$ by identifying two vertices $v_-,v_+$ joined by an
  edge $e$ which is then collapsed, corresponding to a degeneration
  of $C$ in which a component $S_v \subset S$ develops a node, either in the
  interior or boundary depending on the collapsed edge.  Thus $S_v$
  becomes the union of two components $S_{v_1}, S_{v_2}$.
\end{enumerate} 
\end{definition}

\begin{remark} \label{rem:induct}
Perturbations are constructed
by an induction on the dimension of the source moduli space $\cS_\Gamma$. Suppose that single-valued perturbations
    $P_{\Gamma'}$ have been chosen for all types
    $\Gamma' \prec \Gamma$.  A natural gluing procedure gives
    perturbations $P_\Gamma$ in an open neighborhood $U_\Gamma$ of the boundary of
    $\ol{\M}_\Gamma$.  The gluing theorem for linearized operators 
    (see, for example, Wehrheim-Woodward \cite[Theorem 2.4.5]{orient} for the closed case) implies that 
    for $U_\Gamma$ sufficiently small, the perturbation $P_\Gamma$
    is regular over $U_\Gamma$ if all the perturbations $P_{\Gamma'}$
    are regular.     The case of multi-valued perturbations is the same,
    with the following caveat:   It could be the boundary 
    of $\M_{\Gamma}$ is disconnected, so that there are for example
    two types $\Gamma_1,\Gamma_2 \preceq \Gamma$ so that $P_{\Gamma_1}$
    and $P_{\Gamma_2}$ have different number of sheets, say $k_1$
    and $k_2$ respectively.  In this case there is an obvious replication process which $P_{\Gamma_1}, P_{\Gamma_2}$ are each replaced by 
    multi-valued perturbations with $k_1 k_2$ sheets, with each of the 
    original sheets in $P_{\Gamma_1}$ repeated $k_2$ times, and similarly for $P_{\Gamma_1}$.  The gluing construction then proceeds as before. 

    Similarly, given a type $\Gamma$ obtained from $\Gamma_1,\Gamma_2$
    by gluing along an edge, any multivalued perturbations $P_{\Gamma_1}, P_{\Gamma_2}$ for 
$\Gamma_1,\Gamma_2$ with $k_1,k_2$ sheets respectively induces a 
multivalued perturbation for type $\Gamma$ with $k_1 k_2$ sheets. 
\end{remark}

\begin{definition}  \label{def:preg}  A perturbation $P_\Gamma$ for maps of domain type $\Gamma$  is {\em  regular}   if for every map type $\bGamma$ with underlying domain type $\Gamma$ so that $\cM_{\bGamma}$  has expected dimension at most one, every map 
 $u \in \cM_{\bGamma}$ is regular. 
\end{definition}

\begin{theorem} \label{twolevels} Suppose that regular perturbations
    $P_{\Gamma'}$ have been chosen for all types
    $\Gamma' \prec \Gamma$.  Then there exists a comeager set 
    $\cP_\Gamma^{\reg} \subset \cP_\Gamma$ of
    perturbations $P_\Gamma$ agreeing with the glued perturbations
    arising from $P_{\Gamma'}$ in an open neighborhood $U_\Gamma$ of the boundary of
    $\ol{\M}_\Gamma$ such that any $P_\Gamma \in\cP_\Gamma^{\reg}  $ is regular in the 
    sense of Definition \ref{def:preg}.
\end{theorem} 

\begin{proof}  The statement of the Theorem is an application of Sard-Smale. 
By  Theorem \ref{transverse} the universal moduli space (locally over the
moduli space) is transversally cut out.  Consider the forgetful map 
\[ f_{\bGamma}^i : \M^{\univ,i}_{\bGamma,k,p,{\ul{\eps}},\lambda}(L) \to \cP^{{\ul{\eps}}} \quad (u: C \to X,
  P_\Gamma) \mapsto P_\Gamma .\] 
For map types $\bGamma$ of expected dimension at most one, the index
of $f_{\bGamma}^i$ is at most one and by Sard-Smale the set of regular
values of $f_{\bGamma}^i$ is comeager.  Since the space of map types
$\bGamma$ with underlying domain type $\Gamma$ is countable, there
exists a comeager set $\cP_\Gamma^{\reg} \subset \cP_\Gamma$  of perturbations $P_\Gamma$ extending the given
perturbations on the boundary so that
$ \M^{\univ,i}_{\bGamma,k,p,{\ul{\eps}},\lambda}(L) $ is regular for every map
type $\bGamma$ of expected dimension at most one.  
\end{proof}

\begin{theorem}  For any regular perturbation system 
$\ul{P} = (P_\Gamma)$, for any map type $\bGamma$ decorating
    $\Gamma$ of expected dimension at most one the moduli space
    $ \M_{\bGamma}(L) $ is transversally cut out and for any pair
    $\bGamma_1 \prec \bGamma_2$ with $\bGamma_1$ resp $\bGamma_2$
    having expected dimension zero resp. one, there is a natural
    gluing map $\M_{\bGamma_1}(L) \times [0,\eps) \to \M_{\bGamma_2}$.
\end{theorem}

This theorem, for which we do not give a proof, is a consequence of
gluing results similar to those in 
Ekholm-Etnyre-Sullivan \cite[Chapter 10]{ees:lch}, who prove gluing theorems in the non-degenerate case, and Palmer-Woodward \cite{pw:surg}, which deals with a similar Morse-Bott case. Given a regular
stratum-wise rigid treed holomorphic building $u: C \to \XX$ of type
$\Gamma_1$ representing a codimension one stratum $\M_{\Gamma_1}$ in
the boundary of a top-dimensional stratum $\M_{\Gamma_2}$ there exists
a unique family $u_\rho: C_\rho \to X_\rho$ of treed holomorphic
curves of type $\Gamma_2$ converging to $u$.   In the case of 
collapsing edges or making edges finite, each such sequence
$u_\rho$ is defined by first constructing an {\em  approximate solution}:
one removes small balls around a node and uses cutoff functions and geodesic exponential to construct the approximate solutions while in the case of trajectories one removes a small ball around the intermediate critical point and patches together using cutoff function.  On the other hand, 
the case of making zero-length edges positive length follows immediately from the implicit function theorem.

\subsection{Intersection multiplicities}
\label{intmult}

We develop a notion of fractional intersection multiplicity
with the divisors at infinity, related to the sum of 
actions 
of Reeb chords. Fractional intersection multiplicities with the divisors at infinity are defined
by integrating the Thom class.  

\begin{definition}
Let $u: (S,\partial S) \to (X,L)$ be a punctured
surface bounding $L$ asymptotic to Reeb orbits and chords. Then we get a compactified map $\ol{u}:\to (\ol{X},\ol{L})$.  The {\em  intersection number}
of $\ol{u}$ with $Y_\pm$ is the pairing of $[\ol{u}] \in H_2(\ol{X},\ol{L})$ with $[\on{Thom}_{Y_\pm}] \in H^2(\ol{X},\ol{L})$. 
\end{definition}

As usual, the intersection number may be computed as a sum of (now fractional) local intersection numbers, proportional to the sum of actions at the intersection points:

\begin{lemma} \label{angleeq} Suppose that $u: (S,\partial S) \to (X,L)$ is a punctured disk asymptotic to a collection of Reeb orbits
and chords $\gamma_e, e \in \mE(S)$.  The intersection
number $[\ol{u}] \cdot [Y_\pm] $ is equal to the sum of the 
actions
\[  [\ol{u}] \cdot [Y_\pm] = 
\sum_{e \in \mE_\pm(S)} \theta_e \] 
of the Reeb chords and orbits limiting to $Y_\pm.$
\end{lemma}

\begin{proof} The proof is an argument using Stokes' formula.  The cutoff-function $\rho$ is equal to $1$ near the zero section.   Hence
\begin{eqnarray*}
 [\ol{u}] \cdot [Y_\pm] &=& 
\lim_{ s \to \infty} \int_{S - \cup_{e \in \mE_\pm(S)} \kappa_e^{-1}( \pm (s, \infty) \times Z_\pm) } 
\d (\rho \alpha_\pm) \\ &=&
\lim_{ s \to \infty} \int_{\partial ( S - \cup_{e \in \mE_\pm(S)} \kappa_e^{-1}( \pm (s, \infty) \times Z_\pm) ) }
\alpha_\pm 
=  \sum_{e \in \mE_\pm(S)} \theta_e  \end{eqnarray*}
as claimed. 
\end{proof}

\subsection{Compactness}

We will need a version of Gromov compactness for buildings in 
cobordisms with Lagrangian boundary conditions, adapted to the case
of domain-dependent perturbations. We also need to construct domain dependent perturbations which are coherent with respect to the bubblings and breakings, see Remark \ref{rem:induct}. We begin with setting up the notion of coherence for multi-valued perturbations and then prove a version of the Gromov compactness statement relevant to our setup.

\begin{definition}[Multivalued pull-back]
     Let $\Gamma$ be a combinatorial type of treed disk and 
     $e$ an edge of infinite length $\ell(e) = \infty$.  Thus, the graph $\Gamma$ 
     is obtained from gluing two types $\Gamma_1$
     and $\Gamma_2$ along edges $e_1,e_2$ respectively. Let
     \[ {P}_{\Gamma_{i}} =
       \sum_{j=1}^{k_i} c_{k,i} \widetilde P_{\Gamma_{i},k} \]
     be multivalued
     perturbations for $i \in \{1,2 \}$.   The multivalued pull-back, ${P}_{\Gamma}$, of ${P}_{\Gamma_{1}}$ and ${P}_{\Gamma_{2}}$ is defined as $${P}_{\Gamma}=\sum_{j_1=1}^{k_1} \sum_{j_2=1}^{k_2} c_{j_1,1}c_{j_2,2} \widetilde P_{\Gamma,j_1,j_2},$$ where $\widetilde P_{\Gamma,j_1,j_2}$ is defined as the pull-back of $\widetilde {P}_{\Gamma_1,j_1}$ and $\widetilde {P} _{\Gamma_2,j_2}$ 
  under the isomorphism 
\[ {\cU}_{\Gamma} \to   \pi_1^* {\cU}_{\Gamma_1} \times \pi_2^* {\cU}_{\Gamma_2}. \]
Here 
\[ \pi_b: {\cM}_{\Gamma} \to {\cM}_{\Gamma_b}, b \in \{1, 2 \} \] 
are the natural projections. 
\end{definition}

We first note the following construction. For any vertex $v \in \Ver(\Gamma)$,  let 
$\Gamma(v)$ be the graph with a single vertex
$v$ and the edges $e$ containing $v$.   Let 
\[ \cU_{\Gamma,v} = \{ (C,z) \in  \cU_{\Gamma} | z \in S_v \}  \] 
denote the union  of points lying on the components $S_v \subset S$.

\begin{definition}  A collection of multi-valued perturbations 
\[ \ul{P} = ( {P}_\Gamma  \in \cP_\Gamma) \] 
for all types of domain $\Gamma$
is {\em  coherent} if the following holds.

\begin{enumerate} 
\item {\rm (Collapsing Edges)} Whenever $\Gamma_1 \to \Gamma_2$, so that $\Gamma_2$
is obtained from $\Gamma_1$ by collapsing edges, the number of sheets of $\mvP_{\Gamma_1}$ is equal to the number of sheets of $\mvP_{\Gamma_2}$ 
  and each sheet
  $P_{\Gamma_2,i}$ restricts to the perturbation $P_{\Gamma_1,i}$ obtained from gluing on a neighborhood $\U_{\Gamma_2}$ as in Remark \ref{rem:induct} and their corresponding weights match $c^{\Gamma_1}_i =  c^{\Gamma_2}_i $. 
  That is, 
\[ J_{\Gamma_2,i} : \cS_{\Gamma_2} \to \J(X) \] 
restricts
  to $J_{\Gamma_1,i}$ on $\cS_{\Gamma_1} \subset \ol{\cS}_{\Gamma_2}$,
  the vector field    
  \[ \zeta_{\Gamma_2,\black,i}: \cT_{\Gamma_2,\black,i} \to \Vect(X)  \] 
  restricts to  $\zeta_{\Gamma_1,\black,i}$ on $\cT_{\Gamma_1,\black,i}$
  and similarly for the vector fields $\zeta_{\Gamma_2,\white,i}, \zeta_{\Gamma_1,\white,i}$.

\item {\rm (Breaking Edges)} Whenever $\Gamma$ has an edge of infinite
  length, and so obtained from gluing $\Gamma_1$ and $\Gamma_2$, then $P_\Gamma$ restricts to the multivalued pull-backs of $\widetilde{P}_{\Gamma_1}$ and $\widetilde{P} _{\Gamma_2}$ 
  under the isomorphism 
\[ \ol{\cU}_{\Gamma} \to   \pi_1^*  {\cU}_{\Gamma_1} \times \pi_2^* {\cU}_{\Gamma_2} \]
where $\pi_b: {\cM}_{\Gamma} \to {\cM}_{\Gamma_b}, b \in \{1, 2 \}$ are the natural projections. 
\item {\rm (Locality)} Let $C_\white \subset C$ be the subset consisting only of disk components and
  boundary edges.  We have a natural 
map 
\[ f_{\Gamma,v}: \cU_{\Gamma,v} \to \cU_{\Gamma(v)} \times \M_{C_\circ} \] 
which sends the pair $(C, z \in S_v)$ to the pair
$((S_v, z \in S_v),C_\circ)$.  We require that  
the restriction of the perturbation $J_\Gamma$
to $S_v$ is pulled back under
$f_{\Gamma,v}$.  In particular, if $\Gamma'$
is another component obtained from $\Gamma$ by removing
an interior leaf $T_e$ on some sphere component $v$ then 
$P_\Gamma$ induces a perturbation datum $P_{\Gamma'}$
for the type $\Gamma'$ by taking $J_{\Gamma'}$
to be constant on $S_v'$ and equal to $J_\Gamma$
on the other components.  Any $P_\Gamma$ holomorphic 
map from a curve $C$ of type $\Gamma$ to $X$
induces a $P_{\Gamma'}$-holomorphic map from the 
curve $C'$ to $X$, where $C'$ is obtained from $C$
by forgetting $T_e$.
\end{enumerate} 
\end{definition}

We have the following version of sequential Gromov compactness for treed holomorphic buildings. 

\begin{theorem} \label{compactness}
Let $\ul{P} = (P_\Gamma)$ be a coherent system of perturbations.  Any sequence of  stable
  buildings $u_\nu: C_\nu \to \XX$ with boundary in $L$
  with bounded $\ol{\omega}$-area $A(u_\nu)$ and bounded number of 
  leaves $n(u_\nu)$ has a subsequence Gromov-converging to a stable building $u: C \to \XX$. 
  \end{theorem}

\begin{proof}  The limits of the surface parts and tree parts  of the sequence
may be constructed separately.

\vskip .1in \noindent {\em Step 1:  After passing to a subsequence, the 
the surface parts of the map converge.}  This is a version of compactness in symplectic field theory, for which there are the approaches in Bourgeois-Eliashberg-Hofer-Wysocki \cite{sft} and Cieliebak-Mohnke \cite{cm:com}.  In these approaches a domain-independent almost complex structure was used;  the case of domain-dependent almost complex structures follows as in Charest-Woodward \cite{cw:traj}.  Lagrangian boundary conditions that extend over the neck were incorporated in Chanda \cite{chanda};  see also \cite{vw:trop}.    Note that the assumption on the domain-dependent perturbations being given by the gluing construction in Remark \ref{rem:induct} near the boundary implies that the almost complex structures of any sequence converge to the almost complex structures on the limiting curves uniformly in all derivatives on compact subsets away from the nodes, while the perturbations are domain-independent on neighborhoods of the nodes by assumption.

\vskip .1in \noindent {\em Step 2: After passing to a subsequence
the tree parts of the sequence converge.}    Let $u_\nu: C_\nu \to \XX$ be a sequence as in the statement of the Theorem.
Each edge  $T_e, e \in \Edge(\Gamma)$ in a level $C_i$ mapping to $\R \times \Lambda$ gives rise to a  sequence of trajectories by restriction.
By compactness of gradient trajectories up to breaking,
the projections
\[ \ol{u}_{e,\nu}:= {p} \circ u_\nu: [-T_\nu,T_\nu] \to \Pi \]  
converge after passing to a subsequence to a (possibly broken) trajectory.  That is, there exist sequences $t_\nu$ so that 
the trajectories $\ol{u}_{e,\nu}(t + t_\nu)$
converge in all derivatives on compact sets to some limit
\[ \ol{u}_{e,i}: [-T_{e,i},T_{e,i}] \to \Lambda \] 
which is a trajectory of of $\grad(f_\circ)$, for some $T_{e,i} \in [0,\infty]$; see for example Audin-Damian \cite{audin}.  Choose a sequence 
$T_\nu \in \R$ so that the sequence $T_\nu \cdot u_{e,\nu}(0)$ (that is, the translation of 
$u_{e,\nu}(0)$ in the $\R$-direction by $T_\nu$) converges to 
some limiting point $u_{e}(0)$;  such  a sequence $T_\nu$
exists since $\Lambda$ is compact.   By the fundamental theorem of
ordinary differential equations, 
the sequence $u_{e,\nu}$ converges in all derivatives on compact sets
to a trajectory of $\zeta$.   Since the projections $\ol{u}_{e,i}$ form a broken trajectory,
the lifts $u_{e,i}$ form a broken trajectory.
\end{proof} 

 We analyze the boundary of the one-dimensional moduli
  spaces $\ol{\M}(\Lambda)_1$ resp. $\ol{\M}(L)_1$ from \eqref{dlocus},
  assuming that every map is regular.  
  
  \begin{definition} \label{truefake}
  \begin{enumerate} 
  \item Strata
  $\M_{\bGamma}(\Lambda)$ of configurations $u: C \to \XX$
  with a single edge $T_e \subset C $ of length  zero 
  and of expected dimension $\vdim \M_{\bGamma}(\Lambda) =0 $ will be called {\em  fake boundary
    components} of $\ol{\M}(\Lambda)_1$.  If  $ \M_{\bGamma}(\Lambda) $ is regular then 
    there are exactly two one-dimensional strata
  $\M_{\bGamma'}(\Lambda), \M_{\bGamma''}(\Lambda)$ containing $  \M_{\bGamma}(\Lambda) $ in their closure,
  %
  so that 
  \[ \ol{\M}_{\bGamma'}(\Lambda) \cap \ol{\M}_{\bGamma''}(\Lambda) =  \M_{\bGamma}(\Lambda) .\]
These strata  consist of configurations $u: C \to \XX$ with an edge $e$ of
  positive length $\ell(e)$, resp. the two adjacent disk components
  $S_{v_-}, S_{v_+}$ have been glued to form a single disk component.
  \item Strata with
  an edge of infinite length $\ell(e) = \infty$ and of expected dimension $\vdim \M_{\bGamma}(\Lambda) =0 $ are {\em  true boundary
    components} of $\ol{\M}(\Lambda)_1$.  There are two types of such edges $e$, depending on
  what type of edge has acquired infinite length:
\begin{enumerate} 
\item edges $e \in \Edge_{\circ}(\Gamma)$ corresponding to nodes $w \in \partial S$ separating components $S_{v_+}, S_{v_-}$ in the same level and
\item edges  $e \in \Edge_\diam(\Gamma) $
  joining different levels via Reeb   chords $\gamma \in {\cR}(\Lambda)$.

\end{enumerate} 
\end{enumerate}
The definitions of fake and true boundary strata of $\ol{\M}(L)_1$ are  similar.  
\end{definition}

We have the following description of the boundary strata of the one-dimensional moduli spaces
of buildings in cobordisms, stated in two separate theorems.

\begin{theorem} \label{twolevels2} Suppose that $\ol{P} = (P_\Gamma)$ is a regular, coherent 
collection of perturbation data for maps to $\R \times Z$ bounding a cylinder $\R \times \Lambda$ on a  Legendrian $\Lambda$.
The boundary of the one-dimensional stratum $\ol{\M}(\Lambda)_1$ is a union of strata $\M_{\bGamma}(\Lambda)$ where $\bGamma = (\bGamma_1,\bGamma_2)$
is a type of building $u: C \to \R \times Z $ with exactly two levels
$u_1: C_1 \to \R \times Z, u_2: C_2 \to \R \times Z$.
\end{theorem}

\begin{theorem} \label{twolevels3} Suppose that $\ol{P} = (P_\Gamma)$ is a regular, coherent 
collection of perturbation data for a tame Lagrangian cobordism pair $(X,L)$ with convex end $(Z_{+}, \Lambda_{+})$ and concave end $(Z_{-}, \Lambda_{-})$.
The  boundary of $\ol{\M}_\bGamma(L)_1$ is a union of strata $\M_\bGamma(L)$ where $\bGamma = (\bGamma_1,\bGamma_2)$ a type corresponding to a building that is either
\begin{enumerate}
    \item  a treed building $u: C \to \XX $ with exactly two levels $u_1:C_1 \to X,u_2: C_2 \to  \R \times Z_\pm$, one of which maps to $X$
    and the other to $\R \times Z_\pm$, 
    separated by Morse trajectory on $\cR(\Lambda_\pm)$ of infinite length;  or
    %
%
\item a building $u: C \to \XX$ with a single level mapping to $X$ 
consisting of two components $u_-: C_- \to X,u_+: C_+ \to X$ glued together
at an infinite length trajectory $u_e: T_e \to X$ at a critical point of the Morse function $f_{\black}$.
\end{enumerate}
  \end{theorem} 

\begin{proof}[Proof of Theorems \ref{twolevels2} and \ref{twolevels3}]  We prove the claim for the moduli space $\ol{\M}(L)_1$ of buildings
in a cobordism $L$ between $\Lambda_-$ and $\Lambda_+$ 
described in Theorem \ref{twolevels3};  the case of buildings in 
$\R \times \Lambda$ described in 
Theorem \ref{twolevels2}  is easier.   Let $u_\nu: C_\nu \to \XX$  be a sequence of rigid maps of type $\bGamma$ with boundary in $\LL$  in a moduli space of dimension one  converging to a limit $u : C \to \XX$ of some type $\bGamma_\infty$.

\vskip .1in  \noindent {\em Step 1: 
Interior nodes do not form in the limit.  }     By assumption $\bGamma$ has no interior nodes; 
if the type $\bGamma_\infty$ has an interior node then the dimension of the moduli space $\M_{\bGamma}(L)$ of 
buildings containing $u$ has dimension at least two less than that of $u_\nu$, and so negative expected dimension. This is a contradiction. 


\vskip .1in  \noindent {\em Step 2: Nodes at Reeb orbits do not form in the limit.   } 
Recall our tameness assumptions in Definition \ref{def:tamepair}. Denote by
$\Gamma_\circ$ the type obtained by collapsing all sphere components
$S_v \subset S$ to interior leaves. Suppose $S_{v_1}, S_{v_2}$ are connected by a
  node $w$ mapping to a Reeb orbit $\gamma$; that is, $S_{v_1}$ and $S_{v_2}$
  each have a cylindrical end corresponding to an edge
  $e \in \Edge_\white(\Gamma)$.  The disk components $S_{v_1}, S_{v_2}$ are parts
  of different levels in the domain of the building $u$.  Let $\bGamma = \bGamma_1 \cup
  \bGamma_2$  denote the decomposition of $\bGamma$ into subgraphs obtained by dividing at $e$. 
  Since the subgraph $\Gamma_\circ$ of disk
  components is connected, either $\bGamma_1$ or $\bGamma_2$ consists
  entirely of spherical vertices $v \in \Ver_\black(\bGamma)$. 
  Suppose $C = C_1 \cup C_2 $ is the corresponding decomposition of domains, 
  so that $C_2$ is a treed sphere and let $u_2 = u|_{C_2}$.
  Since the combinatorial type is a tree, we may assume that $C_2$
  has incoming punctures but no outgoing punctures or vice versa, and in particular,
  $u_2$ is not a cover of a trivial cylinder.

\vskip .1in  \noindent {\em Case:  $u_2$ maps to $X$ and has a single incoming puncture and no outgoing punctures.}   As in Lemma \ref{nospheres}, we identify the moduli space with maps to $\ol{X} - Y_-$
with a $[Y_+]. [\ol{u}_2] - 1$ vanishing derivatives at $Y_+$ and an asymptotic marker, so that 
\begin{eqnarray*} 
\dim \M_{\bGamma_2}(L) &=& \dim(Y_+) + 2 c_1(\ol{X} - Y_-). [\ol{u}_2]
- 2 ( [Y_+]. [\ol{u}_2] - 1) + 1  - 4 \\
&=& \dim(Y_+) + 2 c_1^{\on{log}}(\ol{X} - Y_-). [\ol{u}_2]
 - 1 \\
&\ge&  \dim(Y_+) + 2 (1 + \lambda_-)[\ol{\omega}] . [\ol{u}_2] - 1
> \dim(Z_+)  \end{eqnarray*} 
Here the $-4$ appears as the dimension of the automorphism group of a once-pointed genus zero curve; the term $+1$ appears because of the
asymptotic marker at the puncture, the term $\dim(Y_+)$ appears since the puncture is required to map to $Y_+$, the term 
$- 2 ( [Y_+]. [u_2] - 1)$ is the correction to the dimension 
arising from the requirement that the first $[Y_+]. [u_2] - 1$
derivatives of the map vanish at the intersection point with $Y_+$
and $2 c_1(\ol{X} - Y_-)$ is the usual contribution of the first Chern class in the Riemann-Roch formula.  
The last line follows from 
\[ 2 c_1^{\on{log}}(\ol{X} - Y_-). [u_2]
=2 (1 + \lambda_-)[\ol{\omega}] . [u_2] > 2 \] 
since $  c_1^{\on{log}}(\ol{X} - Y_-)$ is an integral class and $[\ol{\omega}]. [u_2]$ is positive.     Therefore, the moduli space cannot be made rigid by adding constraints at the puncture
at $Z_+$.

\vskip .1in  \noindent {\em Case: $u_2$ is a level with a single outgoing puncture and no incoming punctures.  }  This case also  violates Lemma \ref{nospheres}.
%
%

\vskip .1in  \noindent {\em Case: $u_2$ is a level mapping to $\R \times Z_\pm$ with one incoming puncture
and no outgoing punctures. }  

We view the moduli space of maps $u_2$ as the space of holomorphic genus zero maps to $\ol{X}$ with $[Y_\pm]. [\ol{u}_2] - 1$ vanishing derivatives at $Y_\pm$, and an additional asymptotic marker at the puncture. 
We have  
\begin{eqnarray*} \dim \M_{\bGamma_2}(\Lambda) &=&
\dim(Y_\pm) + 2 c_1(\mathbb{P}(N_+ \oplus \C)). [\ol{u}_2] 
- 2 ([Y_\pm]. [\ol{u}_2] - 1) +1-4-1 
\\
&= & \dim(Y_\pm) + 2 p^* c_1(Y_\pm). [\ol{u}_2]  - 4\\
&= & \dim(Y_\pm) + 2 c_1(Y_\pm). [p \circ \ol{u}_2]  - 4 \\ 
&>  &  \dim (Y_\pm) + 1 .\end{eqnarray*}
Here the term $1 + 4 - 1$ appears from the contributions of the asymptotic marker, the automorphisms of a once-punctured genus zero curve, and the translation automorphisms on the codomain, respectively; the second follows from the equality 
 $c_1(\mathbb{P}(N_+ \oplus \C)) =  2 p^* c_1(Y_\pm) + [Y_\pm]$, where $Y_\pm$ is embedded as the zero section, and 
the last inequality uses the fact 
that ${Y}_\pm$ has minimal Chern number at least three, by the monotonicity assumption in \eqref{def:tamepair}.  Thus, the moduli space cannot be rigid by the addition of a constraint, as in 
Remark \ref{rem:splittype}.

\vskip .1in  \noindent {\em Step 3: We describe the boundary configurations for the one-dimensional moduli spaces. } Let $u: C \to \XX$ be a configuration in the boundary of a one-dimensional component
of $\M(\L)$.   By the previous paragraphs, $C$ has a single 
boundary edge $T_e$ with zero or infinite length.   In the case of zero length, $u$ is contained in a fake boundary component in the sense of Definition \ref{truefake}, and is not a topological boundary point.  Thus the length of $T_e$ is infinite, and either
$T_e$ separates two levels $C_i, C_{i+1}$ or two components $C',C''$ of the same level $C_i$, as shown in Figure \ref{fig:zeroin}.  
These are the possibilities listed in the statement of the Theorem.
\end{proof}

\begin{figure}
    \centering
    
    \scalebox{.4}{
    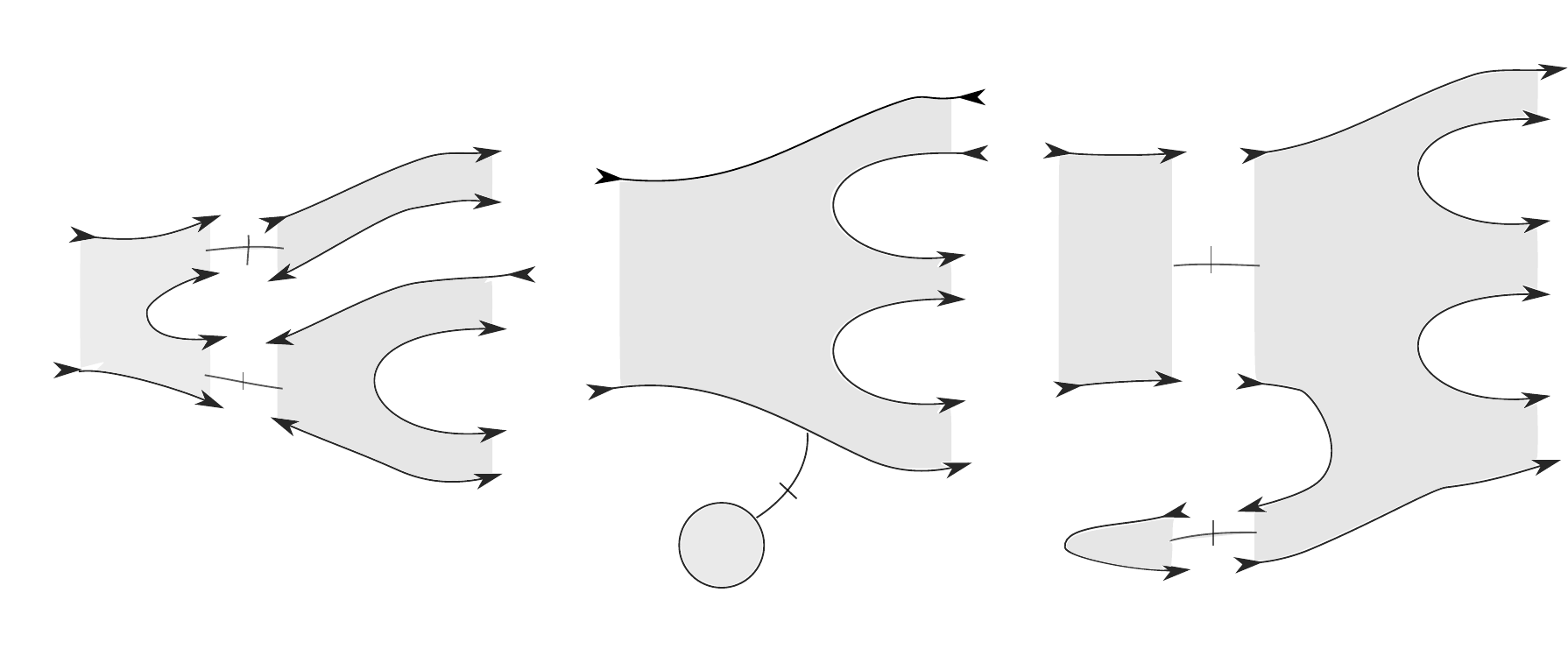}
    \caption{Three types of boundary configurations (a) two levels with one incoming puncture on each component (b) one level with a broken trajectory (c) two levels with zero incoming puncture on some component}
    \label{fig:zeroin}
\end{figure}

\subsection{Orientations} 
\label{sec:orientations}

Orientations of  the moduli spaces where
Legendrians resp. Lagrangians are equipped with relative spin structures have already been constructed by 
Ekholm-Etnyre-Sullivan \cite{ees:orient}, based on early work of 
Fukaya-Oh-Ohta-Ono \cite{fooo:part1}.   We provide a brief outline of how to orient the moduli spaces of treed disks in our setup.

Orientations on the moduli spaces are constructed by the following deformation argument
from Fukaya-Oh-Ohta-Ono \cite{fooo:part1}.  The tangent space at a regular element $u \in \M(L)$ is isomorphic to the kernel,  $\ker \ti{D}_u$, of the linearized operator $\ti{D}_u$
from \eqref{linop}:
\[ T_u \M(L) := \ker(\ti{D}_u) .\] 
The determinant line of $\ti{D}_u$ is denoted 
\[ \DD_u = \Lambda^{\top} (\ker(\ti{D}_u)) 
\cong \Lambda^{\top} (\ker(\ti{D}_u)) \otimes \Lambda^{\top}
(\coker(\ti{D}_u)^*)  . \] 
An orientation of $\M(L)$ is a non-zero element $o_u$ of the determinant line, modulo  multiplication by a positive scalar.

Any path of Fredholm operators induces an identification of
determinant lines up to isotopy, and orientations are induced by a particular 
equivalence class of deformations induced by the relative spin structure. 
First, we glue on Cauchy-Riemann operators on the ends to obtain Cauchy-Riemann operators on closed surfaces.  For each end $e$, let 
$\ti{D}_u^{e,\pm}$ be Cauchy-Riemann operators on punctured surfaces $S_e$ obtained  whose boundary condition is given by a path 
$\kappa$ in $T_{\gamma(t)}( \R \times Z)$
connecting $T_{\gamma(0)} (\R \times \Lambda)$ with $T_{\gamma(1)} (\R \times \Lambda)$
in the family of isomorphic  $T_{\gamma(t)} \R \times Z $; see for example 
\cite[Section 2.4]{orient}.
The Cauchy-Riemann operators $\ti{D}_u$ on 
collection of disks $\ol{S}_v$ without punctures can be glued 
with the operators  $\ti{D}_u^{e,\pm}$.  

The orientations on the moduli spaces are defined by orientations on suitable determinant lines.
Let $ \Sigma^\pm_{\ev_e(u)}$ denote the stable resp. unstable 
manifolds associated to the critical point $\ev_e(u)$
at the end of the edge $e$.  Define 
\[ \DD_e^\pm := \det(\Sigma^\pm_{\ev_e(u)}) \otimes 
\DD_{\kappa(e)} \]
where $\DD_{\kappa(e)}$ is the determinant of the index of the
Cauchy-Riemann operator $\ti{D}_u^{e,\pm} $ associated to the capping path $\kappa$, tensored with
the determinant line 
$\det(T_{\gamma(0)} \Lambda \cap T_{\gamma(1)} \Lambda)$ (using the identifications
of $T_{\gamma(t)} (\R \times Z)$ 
for various  values of $t$ given by parallel transport).  For each surface component
$S_v$ we have an isomorphism of $\DD_u$ with a product of determinant lines
\begin{equation} \label{components} \DD_u \cong \det(T_{[S_v]} \M_\Gamma) \otimes
  \bigotimes_{e \in \cE_-(\Gamma_v)} \DD_e \otimes \DD_{u}^{\bv} \otimes
  \bigotimes_{e \in \cE_+(\Gamma_v)} \DD_e^+ .
\end{equation}
The  relative spin structure on  $L$ gives a deformation of 
$\DD_u^{\bv}$ to a combination of Cauchy-Riemann operators on spheres and operators
on disks with trivial bundles and boundary conditions, as in \cite{fooo:part1}
and \cite[Section 4]{orient}.  For the subset $\M(L)_0$ of rigid maps denote by
\[ 
\eps :  \M(L)_0 \to \{ \pm 1 \} \] 
the sign obtained by comparing the given orientation to the canonical
orientation of a point.   In the case $L = \R \times \Lambda$,
we obtain orientations on $\M(\Lambda) = \M(L)/\R$ as well, and in particular 
a sign map $ \eps: \M(\Lambda)_0 \to \{ \pm 1 \}$.

This completes the construction of transversality, compactness, and orientations of the moduli spaces of treed holomorphic buildings. In particular, there is a well-defined signed count of treed holomorphic buildings as in Figure \ref{fig:zeroin}. In the next part of the series, we use  Theorems \ref{twolevels} and \ref{twolevels2} to construct Legendrian contact homology for fibered contact manifolds satisfying the above assumptions, and maps in homology associated to Lagrangian cobordisms.

\bibliography{leg}{}
\bibliographystyle{plain}
\end{document}